\documentclass[12pt,reqno]{amsart}
\usepackage{cite}
\usepackage{mathtools}
\usepackage{amsmath,color}
\usepackage[v2,all]{xy}
\usepackage{amsfonts,amssymb}
\usepackage{mathrsfs}
\usepackage{xr-hyper}
\usepackage{hyperref}

\def\Sha{\amalg \!\!\amalg}
\usepackage{amssymb}

\usepackage{bold}
 \def\unit{\Eins}
 
 \def\DDelta{\mathbbnew \Delta}

\catcode `\@=11
\def\numberbysection{\@addtoreset{equation}{section}
         \renewcommand{\theequation}{\thesection.\arabic{equation}}}
\numberbysection
\def\subsubsection{\@startsection{subsubsection}{3}%
  \normalparindent{.5\linespacing\@plus.7\linespacing}{-.5em}%
  {\normalfont\bfseries}}
\newtheorem{thm}{Theorem}[section]

\newtheorem{lem}[thm]{Lemma}

\newtheorem{prop}[thm]{Proposition}

\newtheorem{cor}[thm]{Corollary}

\theoremstyle{definition}

\newtheorem{df}[thm]{Definition}

\newtheorem{rmk}[thm]{Remark}

\newtheorem{ex}[thm]{Example}
\newtheorem{example}[thm]{Example}
\newtheorem{as}[thm]{Assumption}

\def\C{\mathcal C}

\def\O{\mathcal{O}}
\def\Ab{\mathcal{A}b}
\def\Z{\mathbb{Z}}
\def\Q{\mathbb{Q}}
\def\Vect{\mathcal{V}ect}
\def\kVect{\Vect_k}
\def\dgVect{dg\mdash\kVect}
\def\cat{\mathcal{C}at}
\def\Set{\mathcal{S}et}

\def\H{\mathscr H}

\def\SSigma{$\Sigma$}

\makeatother
\newcommand{\EZDIAG}[5]{\xymatrix
@C+=2.5cm{*+[r]{#1}
\ar@(u,l)_(0.62){\displaystyle #5}[]
\ar@<1ex>^-{#3}[r]&\ar@<1ex>^-{#4}[l]#2}}

\def\Coop{\check\O}
\def\id{{\mathrm{id}}}
\def\hI{\hat I}
\def\HG{\H_G}
\def\PQ{\HG}
\def\HB{\H_B}
\def\SS{\mathbb{S}}
\def\odo{\otimes \dots\otimes}
\def\cdc{\cdot \ldots\cdot}
\def\kdk{,\dots,}
\def\eps{\epsilon}

\def\B{\mathscr{B}}

\def\Bred{\B^{red}}
\def\nn{\nonumber}
\def\I{\mathscr {I}}
\def\J{\mathscr {J}}
\def\C{\mathscr {C}}
\def\Gr{Gr}
\def\proj{pr}
\def\Rad{\mathcal R}

\def \colim {\mathop {\rm colim}}

\def\FF{\mathfrak F}

\def\N{\mathbb{N}}
\def\dottree{\bullet\hskip-4.5pt |}
\def\simpcat{{\Delta}}

\def\nn{\nonumber}
\def\vardel{\delta}
\def\eps{\epsilon}

\def\Vect{\mathcal{V}ect}
\def\CalC{{\mathcal C}}

\def\Set{{\mathcal S}et}

\def\operads{{\mathfrak O}}

\def\FF{\mathfrak F}

\def\C{\CalC}
\def\Z{{\mathbb Z}}
\def\N{{\mathbb N}}

\def\colim{\mathrm{colim}}

\def\Surj{\mathfrak{FS}}

\def\a{\alpha}

\def\O{{\mathcal O}}
\def\P{{\mathcal P}}

\def\SS{{\mathbb S}}

\def\odo{\otimes \cdots \otimes}

\def\wt{\mathrm{wt}}
\def\deg{\mathrm{deg}}

\def\t{\tau}
\def\con{con}

\def\Ab{{\mathcal Ab}}
\def\gAb{g\Ab}

\def\I{{\mathcal I}}
\def\D{\mathcal D}
\def\inthom{\underline{Hom}}

\def\ta{\twoheadrightarrow}

\def\length{l}
\def\weight{wt}

\def\Bred{\B^{red}}

\def\mdash{\text{-}}

\def\un{\underline{n}}
\def\uone{\underline{1}}

\numberbysection

\begin{document}

\title[Hopf Algebras I: operadic \&  simplicial aspects]{Three Hopf algebras from number theory, physics \& topology, and their common background\\  I: operadic \&  simplicial aspects}

\author[I.\ G\'alvez--Carrillo]{Imma G\'alvez--Carrillo}
\email{m.immaculada.galvez@upc.edu}
\address{Departament de Matem\`atiques, Universitat Polit\`ecnica de Catalunya,
Edifici TR1-Palau d'Ind\'ustries, Carrer Colom, 1.
08222 Terrassa,
Spain.}

\author[R.~M.~Kaufmann]{Ralph M.\ Kaufmann}
\email{rkaufman@math.purdue.edu}

\address{Purdue University Department of Mathematics, West Lafayette, IN 47907
}

\author[A.~Tonks]{Andrew Tonks}
\email{apt12@le.ac.uk}
\address{Department of Mathematics,
University of Leicester,
University Road,
Leicester, LE1 7RH, UK.}

\def\emptyset{\varnothing}




\begin{abstract}
We consider three {\it a priori} totally different setups for Hopf algebras from number theory, mathematical physics and algebraic topology.
These are the Hopf algebra of Goncharov for multiple zeta values, that of Connes--Kreimer for renormalization, and a Hopf algebra constructed by Baues to study double loop spaces.
We show that these examples can be successively unified by considering simplicial objects,
co--operads with multiplication and Feynman categories at the ultimate level. These considerations open the door to new constructions and reinterpretations of known constructions in a large common framework, which is presented step--by--step with examples throughout. In this first part of two papers, we concentrate on the simplicial and operadic aspects.
\end{abstract}
\maketitle

\tableofcontents

\section*{Introduction}
Hopf algebras have long been known to be a highly effective tool for
classifying and methodologically understanding complicated structures. In this vein,
we start by recalling three Hopf algebra constructions, two of which are rather famous and lie at the center of their respective fields. These are Goncharov's Hopf algebra of multiple zeta values  \cite{Gont} whose variants lie at the heart of the recent work \cite{brown}, for example, and the ubiquitous Connes--Kreimer Hopf algebra of rooted forests \cite{CK}. The third Hopf algebra predates them but is not as well publicized: it is a Hopf algebra discovered and exploited by Baues \cite{Baues} to model double loop spaces.
We will trace the existence of the first and third of  these algebras back to a fact known to experts\footnote{As one expert put it: ``Yes this is well--known, but not to many people''.}, namely that simplices form an operad. It is via this simplicial bridge that we can push the understanding of the Hopf algebra of Goncharov to a deeper level and relate it to Baues' construction which comes from an {\it a priori} totally different setup.
Here, we prove a general theorem, that any
simplicial object gives rise to a bi--algebra.

The tree  Hopf algebra of Connes and Kreimer fits into this picture through a map given by contracting all the internal edges of the trees. This map also furnishes an example {\it par excellence} of the complications that arise in this story.
A simpler example is given by restricting to the sub-Hopf algebra of three-regular trees. In this case the contraction map exhibits the corresponding Hopf algebra as a pull-back of a simplicial object. This relationship is implicit in \cite{Gont} and is now put into a more general framework.
Another Hopf algebra that is closely related, but more complicated is the Connes--Kreimer Hopf algebra for renormalization defined on graphs.

We show that the essential key to obtain a Hopf structure in all
these examples is the realization that the Hopf algebras are quotients of bi--algebras and that these bi--algebras have a natural origin coming from Feynman categories.
This explains the ``raison d'\^etre'' of the co--product formulas as simply being the dual to the partial product given by the composition in Feynman categories, which are special monoidal categories. The quotient is furthermore identified as the natural quotient making the bi--algebras connected.

In the first three examples, there is an intermediate explanation in terms of operad theory. These correspond to particularly simple Feynman categories. In this part of the paper, we will discuss this intermediate stage and relay the general categorical treatment to the second part \cite{HopfPart2}.

In this first part, we start with (co)--operads and
co--operads with multiplication, a new notion that we introduce.\footnote{For the experts, we wish to point out that due to different gradings (in the operad degree) this is neither what is known as  a Hopf operad nor its dual.}  We prove a general theorem which states that a co--operad with multiplication always yields a bi--algebra. In the most natural construction, one starts with a unital operad, then dualizes it to obtain a co--unital co--operad. For this co--operad, one regards the free algebra it generates, and this algebra yields a co--operad with multiplication.  This is an instance of a so--called non--connected construction first discussed in \cite{KWZ} that is natural from the point of view of Feynman categories. Starting  at the second step, that is with a co--unital co--operad and taking the free algebra will be called the free construction.
There is also an intermediate quotient, which can be seen as a $q$--deformation. As $q\to 1$, we obtain the Hopf algebra.

In the general setting of co--operads with multiplication, these bi--al\-ge\-bras are neither unital nor co--unital. While there is no problem adjoining a unit, the co--unit is a subtle issue in general, and we discuss the conditions for its existence in detail. We show that the conditions are met in the special cases at hand, as they stem from the dual of unital operads.
A feature of the more general case is that there is a natural ``depth'' filtration. We furthermore elucidate the relation of the general case to the free case  by  proving that there is always a surjection from a free construction to the associated graded.  Going further, we prove the following structural theorem: if the bi--algebra has a left co--algebra co--unit, then  it is a deformation of its associated graded and moreover this associated graded is a quotient of the free construction of its first graded piece.
These deformations are of interest in themselves.

Another nice result comes about by noticing that just as there are operads and pseudo-operads, there are co--operads and pseudo--co--operads. We show that these dual structures lead to bi--algebras and a version of infinitesimal bi--algebras. The operations corresponding to the dual of the partial compositions of pseudo-operads are then dual to the infinitesimal action of Brown \cite{brownann}. In other words, they give the infinitesimal Lie-co--algebra structure dual to the pre-Lie structure.

Moving from the constructed bi--algebras to Hopf algebras is possible under the extra condition of almost connectedness. If the co--operad  satisfies this condition, which technically encompasses the existence of a split bi--algebraic co--unit, then there is a natural quotient of the bi--algebra which is connected and hence Hopf.
In the three examples taking this quotient is  implemented in the original constructions by assigning values to degenerate expressions.

A further level of complexity is reflected in the fact that there are several variations of the construction of the Connes--Kreimer Hopf algebra based for example on planar labeled trees, labeled trees, unlabeled trees and trees whose external legs have been ``amputated'' --- a term common in physics.
We show, in general, these variations correspond respectively to non-Sigma co--operads, coinvariants of symmetric co--operads and
certain colimits, which are possible in semi--simplicial co--operads.

An additional degree of understanding is provided by the insight that the underlying co--operads for the Hopf algebras of Goncharov and of Baues are  given by a simplicial
structure. This also allows us to understand the origin of the shuffle product and other relations commonly imposed in the theory of multiple zeta values and motives from this angle.
For the shuffle product, in the end it is as Broadhurst remarked; the product comes from the fact that we want to multiply the integrals, which are the amplitudes of connected components of disconnected graphs. In simplicial terms this translates to the compatibility of  different naturally occurring free monoid constructions, in the form of the Alexander--Whitney map and a multiplication based on  the relative cup product.
There are more surprising direct correspondences between the extra relations, like the contractibility of a 2-skeleton used by Baues and a relation on multiple zeta values essential for the motivic co--action.

These classes of examples bring us to the ultimate level of abstraction and source of Hopf algebras of this type: the Feynman categories of \cite{feynman}, which we treat in the second part \cite{HopfPart2}. These constructions are more general in the sense that there are other Feynman categories besides those which yield  co-operads with multiplication. One of the most interesting examples going deeper into mathematical physics is the Feynman category whose morphisms are graphs. This allows us to obtain the graph Hopf algebras of Connes and Kreimer mentioned above.  Going further, there are also the Hopf algebras  corresponding to cyclic operads, modular operads, and  new examples based on 1-PI graphs and motic graphs, which yield the new Hopf algebras of Brown \cite{brown}. Here several general constructions on Feynman categories, such as enrichment, decoration, universal operations, and free construction, come into play and give interrelations between the examples.

This first part of the paper is organized as follows: We proceed in steps.
 To be self-contained, we write out the relevant definitions at work in the background at each step. We also start each step with a short overview of the following constructions and their goals.

 We begin by recalling the three Hopf algebras and their variations in \S\ref{hopfpar}. We give all the necessary details and add a discussion after each example indicating its position within the whole theory.

In \S\ref{operadpar}, we consider the non--connected or free case, in which the co--operads have free multiplication.
In order to make the technical details and the build--up of complexity more transparent, we start with a road--map in \S\ref{roadmappar} that runs through the different Connes--Kreimer constructions on trees.
The main results for the non--symmetric case are
Theorem \ref{bialgthm} and Theorem \ref{hopfthm}. These explain the examples of Goncharov, Baues and the planar version of Connes--Kreimer trees.
The infinitesimal structure and the $q$ deformation are summarized in Theorem \ref{indecthm}. As an example, we reconstruct Brown's derivations.
The results for the technically more demanding symmetric version are Theorems \ref{symbialgthm}, \ref{symhopfthm} and \ref{symcopreliethm}. We then proceed to examine the ``amputated case'' in \S\ref{freeampsec} resulting in
Theorem \ref{freeampthm}. We end the paragraph with a discussion of co--actions \S\ref{coactionpar}.

In \S\ref{gencooppar}, we
 give the  definition of a co--operad with multiplication and the constructions of bi--algebras and Hopf algebras.  This paragraph also contains a discussion of the filtered and graded cases. This setup is strictly more general than the three examples, which all have a free multiplication.
 The result for the bi--algebra structure is Theorem \ref{genbialgthm}.
 The discussion of units and co--units is intricate and summarized in \S\ref{unitsummarypar}. The results about bi--algebra deformations are to be found in Theorem \ref{defthm}. The results on Hopf, infinitesimal structure and
deformations all transfer to this more general setting under the assumption of having a bi--algebra unit.

 Given that the origin of the co--operad structure for Goncharov's and Baues' Hopf algebras is simplicial, we develop the general theory for the simplicial setting in \S\ref{simplicialpar}. We give a particularly clean construction for the bi--algebras starting from the observation that simplices form an operad, yielding Proposition \ref{bialgfromsimp}. We then discuss the examples from Baues and Goncharov in this setting. Further results pertain to the cubical structure \S\ref{cubicalpar} and to a co--lax monoidal structure given by simplicial strings \S\ref{stringsec}. Both these observations have further ramifications which will be explored further in the future.

We include a short appendix on co-- and Hopf algebras and an appendix
for the definition of Joyal duality. More about Joyal duality and its applications can be found in \cite[\S\ref{P2-constexpar}]{HopfPart2}.

\subsection*{Acknowledgments}
We would like to thank  D.~Kreimer, F.~Brown, P.~Lochack,  Yu.~I.~Manin, H.~Gangl, M.~Kapranov, JDS Jones, P.~Cartier and A.~Joyal for enlightening discussions.

RK gratefully acknowledges support from  the Humboldt Foundation and the Simons Foundation, the
Institut des Hautes Etudes Scientifiques and the Max--Planck--Institut
for Mathematics in Bonn and the University of Barcelona for their support. RK also thanks the Institute for Advanced Study in Princeton and the Humboldt University in Berlin for their hospitality and
 their support.

IGC was partially supported by Spanish Ministry of Science and Catalan government grants
MTM2012-38122-C03-01, MTM2013-42178-P, 2014-SGR-634, MTM2015-69135-P, MTM2016-76453-C2-2-P (AEI/ FEDER, UE), MTM2017-90897-REDT, and 2017-SGR-932,
and AT by
MTM2013-42178-P, and MTM2016-76453-C2-2-P (AEI/FEDER, UE)
all of which are gratefully acknowledged.

We also thankfully acknowledge the Newton Institute where the idea for this work was conceived during the activity on ``Grothendieck-Teichm\"uler Groups, Deformation and Operads''.
Last but not least, we thank the  Max--Planck--Institut
for Mathematics for the activity on ``Higher structures in Geometry and Physics'' which allowed our collaboration to put  major parts of the work  into their current form.

\subsection*{Notation}
As usual for a set $X$ with an action of a group $G$, we will denote the invariants by $X^G=\{x|g(x)=x\}$
and the  coinvariants by $X_G=X/\sim$ where $x\sim y$ if and only if there exists a $g\in G:g(x)=y$.

For an object $V$ in a monoidal category, we denote by $TV$ the free unital algebra on $V$, that is $TV=\bigoplus_n V^{\otimes n}$, in the case of an Abelian monoidal category,  and by $\bar TV$ the free non--unital algebra on $V$, that is reduced the tensor algebra on $\bar TV=\bigoplus_{n\geq 1}V^{\otimes n}$  in the case of an Abelian monoidal category. Similarly $SV=\bigoplus_{n\geq 0}V^{\odot n}$ denotes the free symmetric algebra and $\bar SV$ the free non--unital symmetric algebra. We use the notation $\odot$ for the symmetric aka.\ symmetrized, aka.\ commutative tensor product: $V^{\odot ^n}=(V^{\otimes n})_{\SS_n}$ where $\SS_n$ permutes the tensor factors.

Furthermore, we use $\underline{n}=\{1,\dots,n\}$ and  denote by $[n]$  the category with $n+1$ objects $\{0,\dots, n\}$ and morphisms generated by the chain $0\to 1 \to \dots \to n$.

\section{Preface: Three Hopf algebras}
\label{hopfpar}
In this section, we will review the construction of the main Hopf algebras which we wish to put under one roof and generalize. After each example we will give a discussion paying special attention to their unique features.
\subsection{Multiple zeta values} We briefly recall the setup of Goncharov's Hopf algebra of multiple zeta values. Given $r$ natural numbers $n_1,\dots, $ $n_{r-1}\geq 1$ and $n_r\geq 2$, one considers the real numbers
\begin{equation}
\zeta(n_1,\dots, n_r):=
\sum_{1\leq k_1 \leq \dots \leq k_r}\frac{1}{k_1^{n_1}\dots k_r^{n_r}}
\end{equation}
The value $\zeta(2)=\pi^2/6$, for example, was calculated by Euler.

Kontsevich remarked that there is an integral representation for these, given as follows.
If $\omega_0:=\frac{dz}{z}$ and $\omega_1:=\frac{dz}{1-z}$ then
\begin{equation}
\label{zetaeq}
\zeta(n_1,\dots, n_r)=\int_0^1 \omega_1\underbrace{\omega_0\dots \omega_0}_{n_1-1}\omega_1\underbrace{\omega_0\dots \omega_0}_{n_2-1}\dots \omega_1\underbrace{\omega_0\dots \omega_0}_{n_r-1}
\end{equation}
Here the integral is an iterated integral and the result is a real number.
The {\it weight} of (\ref{zetaeq}) is $N=\sum_1^r n_i$ and its {\it depth} is $r$.

\begin{ex} As it was already known by Leibniz,
\begin{equation}
\zeta(2)=\int_0^1\omega_1\omega_0
=\int_{0\leq t_1\leq t_2\leq 1}\frac{dt_1}{1-t_1}\frac{dt_2}{t_2}
\end{equation}
\end{ex}

One of the main interests is the independence over $\Q$ of these numbers:
some relations between the values come directly from their representation as iterated integrals, see e.g.\ \cite{browndecomp} for a nice summary.
As we will show in Chapter \ref{simplicialpar}, many of these formulas can be understood from the fact that simplices form an operad and hence simplicial objects form a co--operad.

\subsubsection{Formal symbols}
Following Goncharov, one turns the iterated integrals into formal symbols $\hI(a_0;a_1,\dots,a_{n-1};a_n)$ where the $a_i\in \{0,1\}$.
That is, if $w$ is an arbitrary word in $\{0,1\}$ then $\hI(0;w;1)$ represents the iterated integral from $0$ to $1$ over the product of forms according to $w$, so that
\begin{equation}\hI(0;1,\underbrace{0,\dots,0}_{n_1-1},1,\underbrace{0,\dots, 0}_{n_2-1}{},\dots,1,\underbrace{0\dots0}_{n_r-1}{};1)\end{equation}
is the formal counterpart of \eqref{zetaeq}. The weight is now the length of the word $w$ and the depth is the number of $1$s.
Note that the integrals \eqref{zetaeq} converge only for $n_r\geq2$, but may be extended to arbitrary words using a regularization described e.g. in \cite[Lemma 2.2]{browndecomp}.

\subsubsection{Goncharov's first Hopf algebra}
Taking a more abstract viewpoint, let $\PQ$ be the
polynomial algebra on the formal symbols $\hI(a;w;b)$ for elements $a,b$ and any nonempty word $w$  in the set $\{0,1\}$, and let
\begin{equation}
\label{emptycondeq}
\hI(a;\varnothing;b)=\hI(a;b)=1
\end{equation}

On $\PQ$ define a comultiplication $\Delta$ whose value  on a polynomial generator
is
\begin{multline}
\label{gontcoprod}
\Delta(\hI(a_0;a_1,\dots,a_{n-1};a_n))=
\!\!\!\!\!\!\!\!\!
\sum_{\substack{k\geq 0\\0= i_0 <i_1< \dots < i_k=n}}
\!\!\!\!\!\!\!\!\!\!\!\!
\hI{(a_{i_0};a_{i_1},\dots;a_{i_k})}\\
\otimes
\hI{(a_{i_0};a_{i_0+1},\dots;a_{i_1})}\hI{(a_{i_1};a_{i_1+1},\dots;a_{i_2})}\cdots
\hI{(a_{i_{k-1}};a_{i_{k-1}+1},\dots;a_{i_k})}
\end{multline}
\begin{thm}\cite{Gont}\label{gone}
If we assign $\hI(a_0;a_1,\dots,a_{m};a_{m+1})$ degree $m$ then $\PQ$ with the co--product \eqref{gontcoprod} (and the unique antipode) is a connected graded Hopf algebra.
\end{thm}
\begin{rmk}
The fact that it is unital and connected follows from \eqref{emptycondeq}.
\end{rmk}
\begin{rmk}
The letters $\{0,1\}$ are actually only pertinent insofar as to get multiple zeta values at the end; the algebraic constructions work with any finite set of letters $S$. For instance, if $S$ are complex numbers, one obtains polylogarithms.

\end{rmk}
\subsubsection{Goncharov's second Hopf algebra and the version of Brown}
There are several other conditions one can impose, which are natural from the point of view of iterated integrals or multiple zeta values, by taking quotients. They are
\begin{enumerate}
\item The shuffle formula
\begin{multline}
 \label{shufflecond}
 \hI(a;a_1,\dots,a_{m};b)\hI(a;a_{m+1},\dots,a_{m+n};b)\\=\sum_{\sigma \in \Sha_{m,n}}\hI(a;a_{\sigma(1)},\dots,a_{\sigma(m+n)};b)
    \end{multline}
where $\Sha_{m,n}$ is the set of $(m,n)$-shuffles.

\item
The path composition formula
  \begin{multline}
  \label{pathcompcond}
\forall x\in \{0,1\}: \hI(a_0;a_1,\dots,a_{m};a_{m+1})\\=
\sum_{k=1}^m
\hI(a_0;a_1,\dots,a_{k};x)\hI(x;a_{k+1},\dots,a_{m};a_{m+1})
       \end{multline}
\item The triviality of loops
  \begin{equation}
  \label{trivcond}
\hI(a;w;a)=0
\end{equation}
\item
The inversion formula
\begin{equation}
\label{dualcond}
 \hI(a_0;a_1,\dots, a_{n};a_{n+1})=(-1)^{n}\hI(a_{n+1},a_n,\dots,a_1;a_0)
\end{equation}
\item The exchange formula
\begin{equation}
\label{exchangeformula}
     \hI(a_0;a_1,\dots, a_{n};a_{n+1})=\hI(1-a_{n+1};1-a_n,\dots, 1-a_{1};1-a_{0})
\end{equation}
Here the map $a_i\mapsto 1-a_i$ interchanges $0$ and $1$.
\item 2--skeleton equation
\begin{equation}
\label{2skeq}
    \hI(a_0;a_1;a_2)=0
\end{equation}

\end{enumerate}

\begin{df}
$\tilde\PQ$ be the quotient of $\PQ$ with respect to the  homogeneous relations stemming from conditions
(1),(2),(3) and (4), let $\HB$ be the quotient of $\PQ$ with respect to relations of the conditions
(1), (3), (4) and let $\tilde\HB$ be the quotient by the relations given in the conditions (1),(2),(4),(5) and (6).
\end{df}
Again one can generalize to a finite set $S$.

\begin{thm}\cite{Gont,brownann,browndecomp}
$\Delta$ and the grading descend to $\tilde \PQ$ and using the unique antipode is a  graded connected  Hopf algebra. Furthermore
 (1), (2), (3) imply (4). $\HB$ and $\tilde\HB$ are graded connected Hopf algebras as well.
\end{thm}
\subsubsection{Discussion} In the theory of multiple zeta values it is essential that there are two parts to the story. The first is the motivic level. This is represented by the Hopf algebras and co--modules over them. The second are the actual real numbers that are obtained  through the iterated integrals. The theory is then an interplay between these two worlds, where one tries to get as much information as possible from the motivic level. This also explains the appearance of the different Hopf algebras since the evaluation in terms of iterated integrals factors through these quotients. In our setting,
we will be able to explain many of the conditions naturally. The first condition \eqref{emptycondeq} turns a naturally occurring non-connected bi--algebra into a connected  bi--algebra and hence a Hopf algebra. The existence of the bi--algebra itself follows from a more general construction stemming from co--operad structure with multiplication. One example of this is given by simplicial objects and the particular co--product \eqref{gontcoprod} is of this simplicial type.  This way, we obtain the generalization of $\PQ$.
Condition
\eqref{emptycondeq}
is understood in the simplicial setup in Chapter \ref{simplicialpar} as the contraction of a 1-skeleton of a simplicial object.
The relation
(2)
is actually related to a second algebra structure, the so-called path algebra structure \cite{Gont}, which we will discuss in the future.
The relation
(3)
is a normalization, which is natural from iterated integrals.
The condition (1) is natural within
the simplicial setup, coming from the Eilenberg--Zilber and Alexander--Whitney maps and interplay between two naturally occurring monoids.  That is we obtain a generalization of $\HB$  used in the work of Brown \cite{brownICM,brownann}.

The Hopf algebra $\tilde\HB$ is used in \cite{browndecomp}.
The relation (5), in the simplicial case, can be understood in terms of orientations.
Finally, the equation (6) corresponds to contracting the 2-skeleton of a simplicial object.
It is intriguing that on one hand (6) is essential for the coaction \cite{brownprivate} while it is essential in a totally different context to get a model for chains on a double loop space \cite{BauesHopf}, see below.

Moreover, in his proofs, Brown essentially uses operators $D_r$ which we show to be equal to the dual of the $\circ_i$ map used in the definition of a pseudo-co--operad, see \S\ref{pseudocooperadpar}.
There is a particular normalization issue with respect to $\zeta(2)$ which is handled in \cite{brownICM} by regarding the Hopf co--module $\HB\otimes_{\Q}\Q(\zeta^{\mathfrak{m}}(2))$ of $\HB$. The quotient by the second factor then yields the Hopf algebra above, in which the element representing $\zeta(2)$ vanishes. Natural co--actions are discussed in \S\ref{coactionpar}.

\subsection{Connes--Kreimer}
\subsubsection{Rooted forests without tails} We will consider graphs to be given by vertices, flags or half-edges and their incidence conditions; see Appendix A for details. There are two ways to treat graphs: either with or without tails, that is, free half-edges. In this section, we will recapitulate the original construction of Connes and Kreimer and hence use graphs without tails.

A tree is a contractible graph and a forest is a graph all of whose components are trees.
A rooted tree is a tree with a marked vertex. A rooted forest is a forest with one root per tree.
 A rooted subtree of a rooted tree is a subtree which shares the same root.

\subsubsection{Connes--Kreimer's  Hopf algebra of rooted forests}
We now fix that we are talking about isomorphism classes of trees and forests. In particular, the trees in a forest will have no particular order.

Let $\H_{CK}$ be the free {\em commutative} algebra, that is, the polynomial algebra, on rooted trees, over a fixed ground commutative ground ring $k$.
A forest is thus a monomial in trees
and the empty forest $\varnothing$, which is equal to ``the empty rooted tree'', is the unit $1_k$ in $k$.
We denote the commutative multiplication by juxtaposition and the algebra is graded by the number of vertices.

Given a rooted subtree $\tau_0$ of a rooted tree $\tau$, we define $\tau\setminus \tau_0$ to be the forest obtained by deleting all of the vertices of $\tau_0$ and all of the edges incident to vertices of $\tau_0$ from $\tau$: it is a rooted forest given by a collection of trees whose root is declared to be the unique vertex that has an edge in $\tau$ connecting it to $\tau_0$.

One also says that $\tau\setminus \tau_0$ is given by an admissible cut \cite{CK}.

Define the co--product on rooted trees as:
\begin{equation}
\label{ckintroeq}
\Delta(\tau)\;:=\;\tau\otimes 1_k\;+\;1_k\otimes \tau\;+\sum_{\substack{\tau_0\text{ rooted subtree of }\tau\\\tau_0\neq\tau}} \tau_0\otimes \tau\setminus \tau_0
\end{equation}
and extend it multiplicatively to forests,  $\Delta(\tau_1\tau_2)=\tau_1^{(1)}\tau_2^{(1)}\otimes\tau_1^{(2)}\tau_2^{(2)}$ in Sweedler notation.
One may include the first two terms in the sum by considering also $\tau_0=\tau$ and $\tau_0=\varnothing=1_k$ (the empty sub--forest of $\tau$), respectively, by declaring the empty forest to be a valid rooted sub--tree.
In case $\tau_0$ is empty $\tau\setminus \tau_0=\tau$ and in case $\tau_0=\tau$: $\tau\setminus \tau_0=\varnothing=1_k$.

\begin{thm}\cite{CK}
The comultiplication above together with the grading define a  structure of connected graded Hopf algebra.
\end{thm}
Note that, since the Hopf algebra is graded and connected, an antipode exists.
\subsubsection{Other variants}
There is a planar variant, using planar planted trees.
Another variant which is important for us is the one using trees with tails.
This is discussed in \S\ref{freecooppar} and \cite[\S\ref{P2-summarypar}]{HopfPart2} and  Appendix \ref{Quilapp}.
There is also a variant where one uses leaf labeled trees. For this it is easier not to pass to isomorphism classes of trees and just keep the names of all the half edges during the cutting. These will be introduced in the text, see also \cite{foissyCR1,foissyCR2}.

Finally there are algebras based on graphs rather than trees, which are possibly super-graded commutative by the number of edges. In this generality, we will need Feynman categories to explain the naturality of the constructions. Different variants of interest to physics and number theory are  discussed in \cite[\S\ref{P2-graphexpar}]{HopfPart2}.

\subsubsection{Discussion}
This Hopf algebra, although similar, is more complicated than the example of Goncharov. This is basically due to three features which we would like to discuss. First, we are dealing with isomorphism classes, secondly, in the original version, there are no tails and lastly there is a sub-Hopf algebra of linear trees. Indeed the most natural bi--algebra that will occur will be on planar forests with tails. To make this bi--algebra into a connected Hopf algebra, one again has to take a quotient analogous to the  normalization \eqref{emptycondeq}, implemented by the identification of the forests with no vertices (just tails) with the unit in $k$. To obtain the commutative, unlabeled case, one has to pass to coinvariants. Finally,
if one wants to get rid of tails, one has to be able to `amputate' them. This is an extra structure, which in the case of labeled trees is simply given by forgetting a tail together with its label. Taking a second colimit with respect to this forgetting construction yields the original Hopf algebra of Connes and Kreimer. The final complication is given by the Hopf subalgebra of forests of linear, i.e.\ trees with only binary vertices. This Hopf subalgebra is again graded and connected. In the more general setting, the connectedness will be an extra check that has to be performed. It is related to the fact that for an operad $\O$, $\O(1)$ is an algebra and dually for a co--operad $\check{\O}$, $\check\O(1)$ is a co--algebra, as we will explain. If $\O$ or $\check\O$ is not reduced (i.e.\ one dimensional generated by a unit, if we are over $k$), then this extra complication may arise and in general leads to an extra connectedness condition.

\subsection{Baues' Hopf algebra for double loop spaces} The basic starting point for Baues \cite{Baues}
is a simplicial set $X$, from which one  passes to the chain complex
$C_*(X)$. It is well known that $C_*(X)$ is a co--algebra under the diagonal approximation chain map $\Delta:C_*(X)\to C_*(X)\otimes C_*(X)$,
and to this co--algebra one can apply the cobar construction: $\Omega C_*(X)$ is the free algebra on $\Sigma^{-1}C_*(X)$, with a natural differential which is immaterial to the discussion at this moment.

The theorem by Adams and Eilenberg--Moore is that if $\Omega X$ is  connected then $\Omega C_*(X)$ is
a model for chains on the based loop space $\Omega X$ of $X$.
This raises the question of iterating the construction, but, unlike $\Omega X$, which can be looped again, $\Omega C_*(X)$ is now an algebra and thus does not have an obvious cobar construction. To remedy this situation Baues introduced the following co--multiplication map:

\begin{multline}
\Delta(x)=
\sum_{\substack{k\geq 0\\0= i_0 <i_1< \dots < i_k =n}}
x_{({i_0},{i_1},\dots,{i_k})}\;\otimes\;\\
x_{({i_0},{i_0+1},\dots,{i_1})} x_{({i_1},{i_1+1},\dots,{i_2})}\cdots
x_{({i_{k-1}},{i_{k-1}+1},\dots,{i_k})}
\end{multline}
where $x\in X_n$ is an $(n-1)$-dimensional generator of $\Omega C_*(X)$, and $x_{(\alpha)}$ denotes its image under the simplicial operator specified by a monotonic sequence $\alpha$.

\begin{thm}\cite{Baues} If $X$ has a reduced one skeleton $|X|^1=*$, then the comultiplication, together with the free multiplication and the given grading, make $\Omega C_*(X)$ into a Hopf algebra.
Furthermore if $\Omega\Omega |X|$ is connected, i.e.\ $|X|$ has trivial 2-skeleton,
 then $\Omega\Omega C_*(X)$ is a chain model for $\Omega\Omega |X|$.
\end{thm}

\subsubsection{Discussion}
Historically, this is actually the first of the types of Hopf algebras we are considering.
With hindsight, this is in a sense the graded and noncommutative version of Goncharov and gives the Hopf algebra of Goncharov a
simplicial backdrop. There are several features, which we will point out. In our approach, the existence of the diagonal (co--product), written by hand in \cite{Baues}, is  derived from the fact that simplices form an operad. This can then
be transferred to a co--operad structure on any simplicial set. Adding in the multiplication as a free product (as is done in the cobar construction), we obtain a bi--algebra with our methods. The structure can actually be pushed back into the simplicial setting, rather than just living on the chains, which then explains the appearance of the shuffle products.

To obtain a Hopf algebra, we again need to identify $1$ with the generators of the one skeleton. This quotient  passes  through the  contraction of the one skeleton, where one now only has one generator. This is the equivalent to the normalization \eqref{emptycondeq}. We speculate that the choice of the \emph{chemin droit}  of Deligne can be seen as a remnant of this in further analysis. We expect that this gives an interpretation of \eqref{exchangeformula}.
The condition \eqref{dualcond} can be viewed as an orientation condition, which suggests to work with dihedral instead of non-\SSigma{} operads, see e.g.\ \cite{decorated}. Again this will be left for the future.

Lastly, the condition \eqref{2skeq} corresponds to the triviality of the 2-skeleton needed by Baues for the application to double loop spaces. At the moment, this is just an observation, but we are sure this bears deeper meaning.

\section{Bi-- and Hopf Algebras from (Co)-Operads}
\label{operadpar} \label{subsect:free}

In this section, we give a general construction, which encompasses all the examples discussed in \S\ref{hopfpar}.
We start by collecting together the results needed about operads, which we will later dualize to co--operads in \S\ref{freecooppar} in order to generalize the constructions.
 There are many  sources for further information about operads. A standard reference is \cite{MSS} and \cite{woods} contains the essentials with figures for the relevant examples.
Before going into the technical details, we will consider various Connes--Kreimer type examples of tress and forests for concreteness.
This will also lay out a  blueprint for the constructions.
This includes a discussions of the symmetric and non--symmetric case, an infinitesimal version together with co-derivations, Connes--Kreimer type amputations and grading.

There is an even more general theory using co-operads and co--operads with multiplication which is treated in
\S\ref{gencooppar}.

\subsection{Connes--Kreimer trees as a road map}
\label{roadmappar}
There are several versions of the Connes-Kreimer Hopf algebra which we will discuss and generalize.
Here we give a first look.

\subsubsection{Planar planted trees/forests}
The first are    planar planted trees. In a planar planted tree, the leaves are naturally ordered by the planar embedding. A planar planted tree has a marked half--edge or flag at the root vertex that is unpaired, viz. is not part of an edge. All other unpaired half--edges are called leaves. Leaf vertices are not allowed.
These planar planted trees form an operad, by gluing the operations $\tau_1\circ_i \circ \tau_2$ which glues the tree $\tau_2$ to $\tau_1$ by joining the root half--edge to the i--th leaf forming a new edge, see Figure \ref{treegraftfig}.

\begin{figure}[h]
    \centering
    \includegraphics[height=1in]{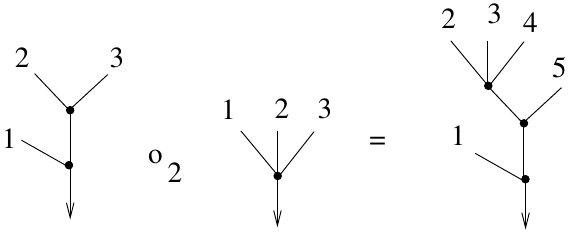}
    \caption{\label{treegraftfig}Grafting trees with labeled leaves. The tree is grafted onto the leaf number 2.}
\end{figure}

If $\tau$ has $k$ labeled leaves, there is a full gluing operation $\gamma(\tau,\tau_1,\dots,\tau_k) $ which  simultaneously glues all the trees $\tau_i$, where $\tau_i$ is glued onto the i--th leaf. We allow the tree $|$ that consists only of one leaf. This is a sort of identity. Consider the linear version ---that is the free Abelian group or vector space based on the set of leaf labeled trees. This is a graded space $\bigoplus_{n\geq 1} \O(n)$ where $\O(n)$ is generated by $n$ labeled trees, and we insist that there is at least one leaf.

In the dual $\Coop(n)=Hom(\O(n),k)$, we view the trees as their dual characteristic functions: $\tau\leftrightarrow \delta_\tau$ where $\delta_{\tau}(\tau)=1$ and $\delta_{\tau}(\tau')=0$ for all $\tau'\neq \tau$.
One gets operations dual to gluing by cutting  trees using an admissible cut $c$: $\check\gamma_c(\tau)=\tau_0\otimes (\tau_1\odo \tau_k)$,
where $\tau_0$ is the left over stump which has $k$ leaves, $c$ is a collection of edges or leaves, such that if they are cut, one is left with the spump $\t_0$ and a collection of trees $\t_1\kdk\t_k$. The leaves and root half--edge cut into two leaves, see the Figure  \ref{fig:coproducts2a}, where we have replaced the free multiplication $\otimes$ by juxtaposition. If $n_i$ are the number of leaves of the $\tau_i$ and $n$ is the number of leaves of $\tau$, then $\sum_{i=1}^k n_i=n$, see Figure \ref{fig:coproducts2a}.

\begin{figure}[h]
    \centering
    \includegraphics[width=0.8\textwidth]{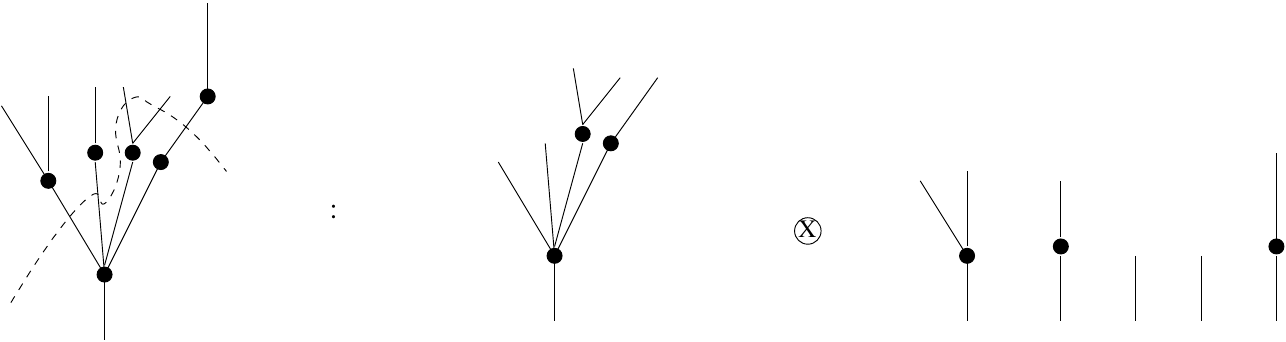}
    \caption{ \label{fig:coproducts2a} A cut in the non--symmetric planar case. Then the right side is  an ordered tensor product. Alternatively, the same cut can be seen as a term on the full coinvariants, viz.\ the unlabeled version where on the right side, the forest is a symmetric product. }

\end{figure}

Let $\Coop=\bigoplus_n \Coop(n)$.
Summing over all possible cuts, one gets the map $\check\gamma:\Coop \to  \bar{T}\Coop$ dual to $\gamma$, where $\bar{T}\Coop$ is the  reduced tensor algebra, which is the free (non--unital)   algebra on $\Coop$.
One can then obtain a bi--algebra $\B'=T\Coop$ by extending $\check\gamma$ to $T\Coop$ via the bi--algebra equation: $\Delta(ab)=\sum a^{(1)}b^{(1)}\otimes a^{(2)}b^{(2)}$ in Sweedler notation, again replacing the free multiplication $\otimes$ by juxtaposition. The bi-algebra $\B'$ is the bi--algebra of planar planted forests. It is unital and co-unital,
with the co-unit evaluating to 1 on a generator of the form $||\cdots |$ and $0$ on all other generators. Here the empty forest is the unit $1\in k\subset T\Coop$.  The bi--algebra is graded. The degree of a forest is the total number of leaves $n$ minus the number of trees in the forest $l$. The generalization of this construction is Theorem \ref{bialgthm}

To obtain a Hopf algebra, one needs to take the quotient by the two sided ideal $\I$ generated by $1-|$, i.e.\ $\H=\B'/\I$.

The result is a   connected  bi--algebra and hence Hopf. This is generalized by Theorem \ref{hopfthm} where we obtain a non--commutative Hopf algebra from a non-$\Sigma$ operad under certain conditions.  The conditions guarantee that the quotient is connected  bi--algebra. The grading descends to the quotient, and  is related to the coradical degree, see Example \ref{gradingex} and \S\ref{gradingpar}.

\subsubsection{Leaf--labeled rooted trees/forests}
The construction above, with modifications,  works in the case of leaf labeled rooted trees. In this case there is no natural order on the leaves, which is why one has to add labels to them in order to define the gluing. Labeling the leaves of a tree with $n$ leaves by $1,\dots, n$, we can again define the $\tau_1\circ_i\tau_2$, as well as $\gamma$. Here we have to take care about the new labeling; this is done in the standard operadic way, keep the labeling of the unglued tails of $\tau_1$ up to label $i-1$ as before, this is followed by the enumerated tails of $\tau_2$ in their increasing order, and then continues with the rest of the unglued tails of $\tau_1$ in their order.
When we want to dualize, we, however see, that cut would yield no labels on the leaves of the stump $\tau_0$, see Figure \ref{fig:coprducts}.

\begin{figure}[h]
    \centering
    \includegraphics[width=0.8\textwidth]{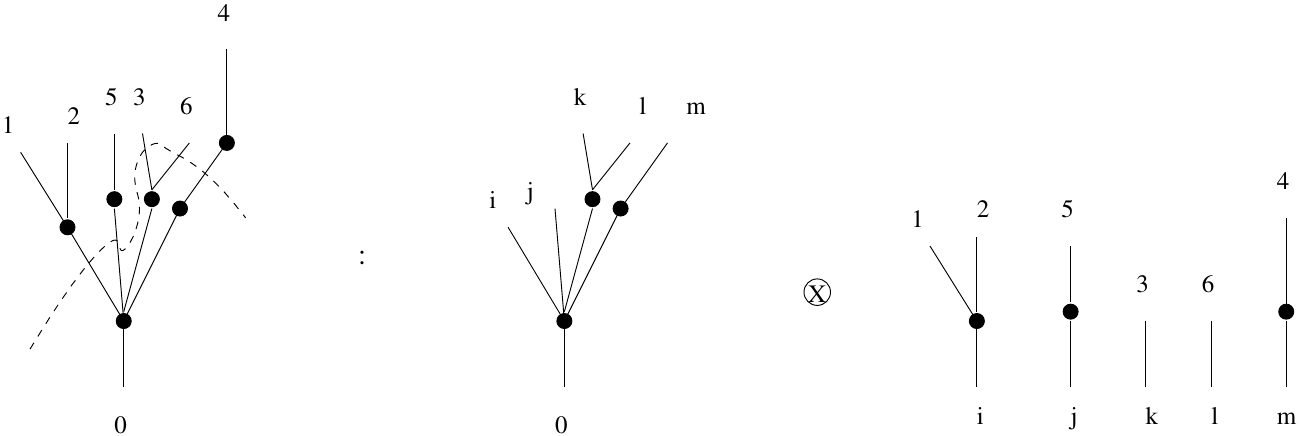}
    \caption{ One term of the dual to $\gamma$ as given by a cut in the labeled case. The labels $i,j,k,l,m$ indicate the parring of the half edges after severing the edges.
    }
    \label{fig:coprducts}
\end{figure}

Furthermore, since there is no labeling, there is also no order on the forest that results from  cutting or deleting $\tau_0$.
This is why one should consider the unlabeled duals, that is the coinvariants $\Coop(n)_{\SS_n}$, where the symmetric group $\SS_n$ permutes the labels. Set $\Coop_{\SS}=\bigoplus_n \Coop(n)_{\SS_n}$ then a cut on the unlabeled tree gives a morphism
morphism $\bar\gamma^{\vee}:\Coop_{\SS}\to \Coop_{\SS}\otimes S\Coop_{\SS}$ where $S$ is the free symmetric algebra on $\Coop_\SS$, see Figure \ref{fig:coproducts2a} for an example.

To obtain a Hopf algebra, we again take the quotient by the two sided ideal $\I$ generated by $1-|$. This is generalized by Theorem \ref{symhopfthm}, where we obtain a commutative Hopf algebra from an operad, again under certain conditions that guarantee that the quotient $\H=\B'/\I$ is connected.
The grading is as in the planar case.

There are several intermediate cases, one of which uses the equivariant tensor product, see Figure \ref{fig:coproducts2b}  and Remark \ref{coinvrmk}. Another version is given by incorporating certain symmetry factors; cf.\ \cite{CK,CK1,CK2,CL} and \cite[\S\ref{P2-alternatepar}]{HopfPart2}.

\begin{figure}[h]
    \centering
    \includegraphics[width=0.8\textwidth]{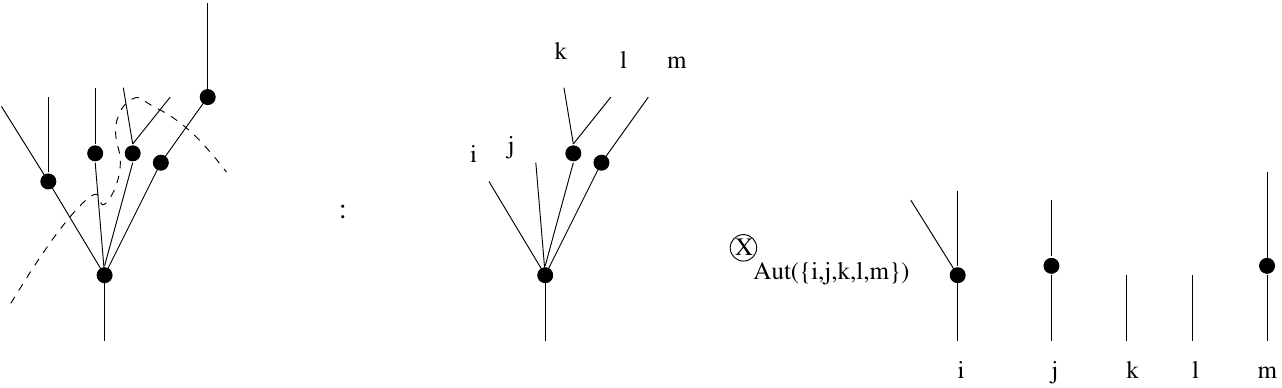}
    \caption{An example of a cut  for the coinvariants yielding the relative tensor product according to Remark \ref{coinvrmk} part (2).}
    \label{fig:coproducts2b}
\end{figure}

\subsubsection{Original version}
In the original version of Connes and Krei\-mer \cite{CK}, the trees are rooted, but have no half--edges or tails. There is a planar and a non--planar version, see e.g.\ \cite{foissyCR1,foissyCR2}. These trees are not glued, but only cut using admissible cuts. During the cutting both half edges of a cut edge are removed, see Figure \ref{CKcoprod}.
In order to obtain this structure from the ones above, one has to amputate the leaves. In the specific example this can easily be done, but in the general setup
this can be achieved by adding certain structure maps, see \S\ref{freeampsec}. An alternative view of this procedure is given by adding trees without tails, see \S\ref{colimpar}.

\begin{figure}[h]
    \centering
    \includegraphics[width=0.8\textwidth]{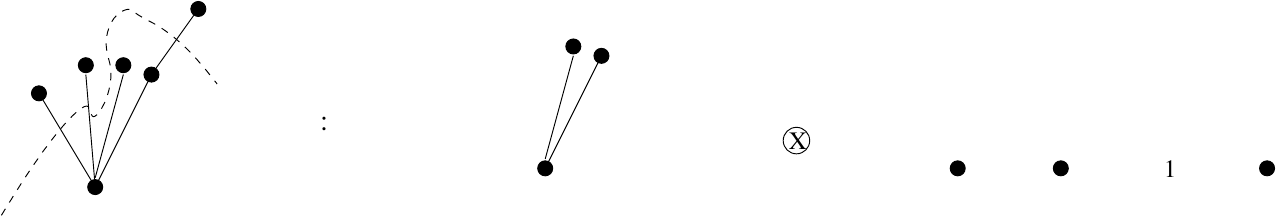}
    \caption{Co--product for the amputated version. The same example for the amputated version: First all tails are removed.  After cutting all newly formed  tails are amputated and  empty trees/forests are represented by $1=1_K$. Notice that indeed $||$ from Figure \ref{fig:coproducts2a} is set to $1_K$ as is done in the Hopf quotient.}
    \label{CKcoprod}
\end{figure}

Notice that the amputation of tails identifies $|$ with $1$ and  does not preserve the degree. The grading on $\H_{CK}$ is instead given by the co-radical degree, which in the case at hand is the number of vertices. The relationship is discussed in \S\ref{gradingpar}.

\subsection{Operads}
We now formalize and generalize the construction above, starting
with a review of operads in \S\ref{operadsec}. We will present the relevant details here and refer to \cite{MSS,woods}, for a full treatment.

In general, a non--$\Sigma$ operad has the simultaneous gluing operations $\gamma$. If some  additional conditions are satisfied,  these dualize to $\check\gamma$, which yields a co--product on the free algebra over the graded dual to the operad. The single gluing operations $\circ_i$ assemble to a pre--Lie product, which dually gives the structure of an infinitesimal co--pre--Lie algebra, see \S\ref{pseudocooperadpar}.
For  symmetric operads, we obtain a similar structure on the free symmetric algebra over the graded dual, see \S\ref{freesymsec}.
To obtain Hopf algebras, one takes a quotient by an ideal, see \S \ref{freehopfsec}. The result is connected --and hence Hopf--- if certain conditions are met. Finally, the amputated version is discussed in \S\ref{freeampsec}.

\subsubsection{Non-\SSigma{} pseudo-operads}
\label{operadsec}
Loosely an operad is a collection of objects or ``somethings'' with $n$ inputs and one output, like functions of several variables. And just like for functions there are permutations of variables and substitution operations.

To make things concrete: consider the  category $\kVect$ of $k$--vector spaces with the monoidal structure $\otimes$ given by the tensor product $\otimes_{k}$.
{\em A non-\SSigma{} pseudo-operad} in this category is given by a collection $\O(n)$ of  Abelian groups,  together with structure maps
\begin{equation}
\label{circieq}
\circ_i: \O(k)\otimes \O(m)\to \O(k+m-1) \text{ for } 1\leq i\leq k
\end{equation}
which are associative in the appropriate sense, that is:
\begin{equation}
(-\circ_i -)\circ_j -=\left\{\begin{array}{cl}
-\circ_i( -\circ_{j-i+1}-)&\text{ if  }i\leq j <m+i\\
((-\circ_j -)\circ_{i+n-1}-)\pi&\text{ if }1\leq j<i.\end{array}\right.
\end{equation}
Here $\pi=(23):\O(k)\otimes \O(m)\otimes \O(n)
\cong\O(k)\otimes \O(n)\otimes \O(m)$.

\begin{rmk}
The data we need to write down the equations above are  a monoidal, aka.\ tensor, product $\otimes_k$---to write down the morphisms $\circ_i$,
   associativity constraints $A_{UVW}:(U\otimes V)\otimes W\stackrel\cong\longrightarrow U\otimes (V \otimes W): (u\otimes v)\otimes w\mapsto u\otimes (v\otimes w)$, ---these are needed to re--bracket---  and commutativity constraints $C_{UV}:U\otimes V\stackrel\cong\longrightarrow \otimes U: u\otimes v \mapsto v \otimes u$ ---these are needed to permute the factors---. Additional data for a monoidal category are a unit $\unit$, $\unit=k$ in $\kVect$, and the unit constraints $U_L:V\otimes \unit \stackrel\cong\longrightarrow V, U_R: k\otimes V\to V$.

Thus, in general the $\O(n)$ are objects in a symmetric monoidal category, which is the following data:  a category $\C$ together with a functor $\otimes: \C\times \C \to \C$,  the $A_{XYZ}, C_{XY}, \unit, U_L,U_R$ which have to satisfy natural compatibility  axioms, see e.g.\cite{Kassel}.

The categories we will consider are the category of sets $\Set$ with disjoint union $\amalg$ and unit $\emptyset$,
$k$--vector spaces  $\kVect$ with $\otimes_k$, differential graded $k$ vector spaces $\dgVect$ with $\otimes_k$, unit $k$ in degree $0$,  the usual additive grading $\deg(u\otimes v)=\deg(u)+\deg(v)$  and the graded commutativity constraint $C_{UV}(u\otimes v)=(-1)^{\deg(u)+\deg(v)}(v\otimes u)$,
 Abelian groups $\Ab$ with $\otimes_\Z$ with unit $\Z$, or $\gAb$ graded Abelian groups with $\otimes_\Z$, unit $\Z$ in degree $0$,  additive grading and graded commutativity. The associativity and unit constraints are the obvious ones.

\end{rmk}
We call  $\O$ reduced if $\O(1)=\unit$,
the unit of the monoidal category.

\subsubsection{Pseudo-operads}
If we add the condition that each $\O(n)$ has an action of the symmetric group $\SS_n$ and that the $\circ_i$ are  equivariant with respect to the symmetric group actions in the appropriate sense, we arrive at the definition of a  pseudo-operad. Given a non--$\Sigma$ pseudo--operad, we can always produce a pseudo operad by tensoring $\O(n)$ with the regular representation of $\SS_n$.

\begin{ex}
\label{endex}
A very instructive example is that of multivariate functions, given by the collection $\{End(X)(n)=Hom(X^{\otimes n},X)\}$. The $\circ_i$ act as substitutions, that is, $f_1\circ_i f_2$  substitutes the function $f_2$ into the $i$th variable of $f_1$. The symmetric group action permutes the variables. The equivariance then states that it does not matter if one permutes first and then substitutes or the other way around, provided that one uses the correct permutation.

As it is defined above $\{End(X)(n)\}$ is just an pseudo--operad in sets. If
 $X$ is  a vector space $V$ over $k$, and $\otimes$ is the tensor product over $k$ and the functions are  multilinear and
$Hom(V^{\otimes n},V)$ are again a vectors spaces. That is $Hom$ is actually what is called an internal Hom, denoted by $\underline{Hom}$, i.e.\ $\underline{Hom}(V^{\otimes n},V)$ is the vector space of multilinear maps.  If one takes $X$ to be a set or a compactly generated Hausdorff space $\otimes$ stands for the Cartesian product  $\times$ and one uses the compact-open topology on the set of maps to obtain a space --- again an internal Hom. More generally, if the monoidal category $\C$ is closed, that means that internal Homs exist and $\otimes$ and $\underline{Hom}$ are adjoint, viz.\   $Hom(U\otimes V,W)\leftrightarrow Hom(U,\underline{Hom}(V,W))$ are in natural bijection, then the $\underline{End}(n)=\underline{Hom}(V^{\otimes n},V)$ form an operad  in $\C$.
\end{ex}

\subsubsection{The  main examples}

\label{examplesec}
Here we give the main examples which underlie the three Hopf algebras above.  Notice that not all of them directly live in $\Ab$ or $\kVect$, but for instance live in $\Set$. There are then free functors, which allow one to carry these over to $\Ab$ or $\kVect$ as needed.
\begin{ex} \label{trees-example}
The operad of leaf-labeled rooted trees. We consider the set of rooted trees with $n$-labeled leaves, which means a bijection is specified between the set of leaves and  $\{1,\dots,n\}$. Given a $n$-labeled tree $\tau$ and an $m$-labeled tree $\tau'$, we define an $(m+n-1)$-labeled tree $\tau\circ_i\tau'$ by grafting the root of $\tau'$ onto the $i$th leaf of $\tau$ to form a new edge. The root of the tree is taken to be the root of $\tau$ and the labeling first enumerates the first $i-1$ leaves of $\tau$, then the leaves of $\tau'$ and finally the remaining leaves of $\tau$, see Figure \ref{treegraftfig}.

The action of $\SS_n$ is given by permuting the labels.

There are several interesting sub--operads, such as that of trees whose vertices  all have valence $k$. Especially interesting are the cases $k=2$ and $3$: also known as the linear and the binary trees respectively. Also of interest are the trees whose vertices have valence at least $3$.
\end{ex}

\begin{ex} The non-\SSigma{} operad of
(unlabeled)
planar planted trees. A planar planted tree is a planar rooted tree with a linear order at the root. Planar means that there is a cyclic order for the flags at each vertex. Adding a root promotes the cyclic order at all of the non-root vertices to a linear order, the flag in the direction of the root being the first element. For the root vertex itself, there is no canonical choice for a first vertex, and planting makes a choice for first flag, which sometimes called the root flag. With these choices, there is a linear order on all the flags and in particular there is a linear order to all the leaves, that is non--root tails.\footnote{Here and in the following, if not otherwise explained, we refer to Appendix \cite[Appendix \ref{P2-graphsec}]{HopfPart2} for the nomenclature and definitions.} Thus, we do not have to give them extra labels for the gluing: there is an unambiguous $i$-th leaf for each planar planted tree with $\geq i$ leaves, and $\tau \circ_i \tau'$ is the tree obtained by grafting the root flag of $\tau'$ onto that $i$-th leaf.

Restricting the valency of the vertices to be either constant, e.g.\ 3--valent vertices only, or less or equal to a given bound yields sub--operad.
\end{ex}

\begin{ex}
\label{ordsurjex}
The  operad of order preserving surjections, also known as planar planted labeled corollas, or just the associative operad.
Consider $n$-labeled planar planted corollas, that is, rooted trees with one vertex, leaves labeled by $\underline{n}=\{1,\dots,n\}$ and an order on them.
For an $m$-labeled planar planted corolla $c_m$ and an $n$-labeled planar planted corolla $c_n$ define
$c_m\circ_i c_n=c_{n+m-1}$. This is  the  $(n+m-1)$-labeled planar planted corolla with the same relabeling scheme as in example \ref{trees-example} above.
This corresponds to splicing together the orders on the sets.
Alternatively, the gluing can be thought of as the gluing on planar planted labeled trees followed by the edge contraction of the new edge, see Figure \ref{corollafig}.

\begin{figure}[h]
    \centering
    \includegraphics[height=.8in]{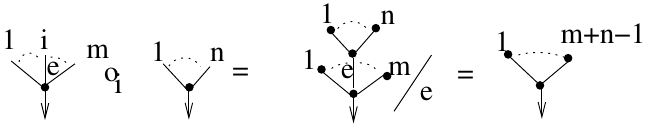}
    \caption{\label{corollafig}Grafting of rooted corollas as first grafting trees and then contracting the new edge.}
    \end{figure}

Alternatively we can think of such a corolla as the unique order preserving map from the ordered set $(\underline{n}=\{1,\dots, n\},<)$, to the one element set $\underline{1}=\{1\}$ with its unique order.
The composition $\circ_i$,  of the maps is given by substitution, that splicing in $(\underline{n},<)$ into the position of $i$. This corresponds to gluing the planar planted corollas.
The $\SS_n$ action permutes the labels and acts effectively on the possible orders. Forgetting the $\SS_n$ operations this is a non-$\Sigma$ operad.

There is another non-\SSigma{} version, that of unlabeled planar planted corollas.
If $c_n$ is unique unlabeled planar planted corolla, then the operations are $\circ_i:c_m\otimes c_n\mapsto c_{n+m-1}$. We obtain back the symmetric operad above as $c_n\times \SS_n$, where the operad structure on $\SS_n$ is given by block permutations, see e.g.\ \cite{MSS,woods} and  $\SS_n$ acts on itself. The identification uses that an element $\sigma$ of $\SS_n$ gives a unique order to the set $\underline{n}: \sigma(1)<\cdots < \sigma(n)$ and the block permutation corresponds to splicing in the orders, which is alternatively just the re--labeling scheme, see \cite{woods}.
Forgetting the $\SS_n$ action this also identifies  the unlabeled planar planted corollas, with the
the non--\SSigma{} sub--operad of order preserving surjections of the sets $\underline{n}$ with their natural order.
Vice--versa, the unlabeled
version is given by the $\SS_n$ coinvariants.

\end{ex}
\begin{ex}
Simplices form a non-\SSigma{} operad (see also Proposition~\ref{deltaoperad} for another dual operad structure). We consider $[n]$ to be the category with $n+1$ objects $\{0,\dots, n\}$ and morphisms generated by the chain $0\to 1 \to \dots \to n$. The $i$--th composition of $[m]$ and $[n]$ is given by the following functor $\circ_i:[m]\sqcup [n]\to [m+n-1]$. On objects of $[m]: \circ_i(l)=l$ for $l<i$ and $\circ_i(l)=l+n-1$ for $l\geq i$. On objects of $[n]:\circ_i(l)=i-1+l$. On morphisms:  the morphism $l-1\to l$ of $[m]$ is sent to the morphism $l-1\to l$ of $[m+n-1]$ for all $l<i$, the morphism $i-1\to i$ of $[m]$ is sent \ to the composition of $i-1\to i \dots \to i+n-1$ in $[m+n-1]$, the morphism $l-1\to l$ of $[m]$ to $l+n-1\to l+n$ of $[m+n-1]$ for $l>i$ and finally  sends the morphism $k\to k+1$ of $[n]$ to $k+i\to k+1+i$.

In words, one splices the chain $[n]$ into $[m]$ by replacing the $i$-th link, see Figure \ref{splicefig}. This is of course intimately related to the previous discussion of order preserving surjections. In fact the two are related by Joyal duality, cf.\ \ref{Joyalapp}, as we will explain in \cite[\S\ref{P2-anglepar}, in particular Figure \ref{P2-angle}]{HopfPart2}.

\begin{figure}
\resizebox{\textwidth}{!} { $$
\xymatrix@C=24pt{
0'\ar[r]& 1'\ar[r]& \dots\ar[r]&  n'& \circ_i &0''\ar[r]& 1''\ar[r]&\dots\ar[r]& m''}
$$}\\[2mm]

\resizebox{\textwidth}{!}{
$$\xymatrix@C=12pt{
&&&\circ_i(0'')\ar@{=}[d]\ar[r]&\circ_i(1'')\ar@{=}[d]\ar[r]&\dots\ar[r]&\circ_i(m'')\ar@{=}[d]&\\
0\ar@{=}[d]\ar[r]& 1 \ar@{=}[d]\ar[r]& \dots\ar[r]& i-1\ar@{=}[d]\ar[r]&i\ar[r]&\dots\ar[r]&
i+m-1\ar@{=}[d]\ar[r]&\dots\ar[r]&m+n-1\ar@{=}[d]\\
\circ_i(0')\ar[r]&\circ_i(1')\ar[r]&\dots\ar[r]&\circ_i((i-1)')\ar[rrr]&&&\circ_i(i')\ar[r]&\dots\ar[r]&
\circ_i(n')
}
$$}
\caption{\label{splicefig}
Splicing together simplices. Primes and double primes are mnemonics only}
\end{figure}

\end{ex}

\subsubsection{The $\circ$-product aka.\ pre-Lie structure}
One important structure going back to Gerstenhaber \cite{Gerst} is
the following bilinear map:

\begin{equation}a\circ b:=\sum^n_{i=1} a\circ_i b \text{ if $a$ has operad degree $n$}
\end{equation}

This product is neither commutative nor associative but pre--Lie, which means that it satisfies the equation $(a\circ b)\circ c- a\circ (b \circ c)=
(a\circ c)\circ b- a\circ (c \circ b)$.

An important consequence is that  $[a,b]=a\circ b -b\circ a$ is a Lie bracket.

\begin{rmk}
\label{shiftedgerstrmk}
One considers a graded version with ``shifted'' degrees in which $\O(n)$ has degree $n-1$. The operation $\circ_i$ are then of degree $1$ and one sets:

\begin{equation}a\circ b:=\sum^n_{i=1} (-1)^{(i-1)(n-1)}a\circ_i b \text{ if $a\in \O(n)$}
\end{equation}
The algebra is graded pre--Lie \cite{Gerst} and the commutator is odd Lie.
This is done e.g.\ in the co--bar construction and is highly relevant for several constructions, see \cite{KWZ} for a full discussion

\end{rmk}

\subsubsection{(Non-\SSigma{}) Operads}
Another almost equivalent way to encode the above data is as follows. A non--\SSigma{} operad is a collection $\O(n)$ together with structure maps
\begin{equation}
\label{mayeq}
\gamma=\gamma_{k;n_1,\dots,n_k}:\O(k)\otimes \O(n_1)\odo \O(n_k)\to \O(\sum_{i=1}^k n_i)
\end{equation}
 Such that map $\gamma$ is associative in the sense that if $(n_1,\dots,n_k)$ is $k$ a partition of $n$, and $(n^i_1,\dots, n^i_{l_i})$ are $l_i$ partitions
 of the  $n_i$,  $l=\sum_{i=1}^k l_i$ then
\begin{multline}
\label{assopeq}
\gamma_{k;n_1,\dots, n_k}\circ id \otimes \gamma_{l_1;n^1_1,\dots, n^1_{l_1}}\otimes \gamma_{l_2;n^2_1,\dots, n^2_{l_2} } \odo  \gamma_{l_kn^k_1,\dots, n^k_{l_k}}=\\
\gamma_{l;n^1_1,\dots,n^1_{l_1},n^2_{1},\dots, n^2_{l_2},\dots, n^k_{1},\dots, n^k_{l_k}}\circ
  \gamma_{k;l_1,\dots, l_k}\otimes id^{\otimes l}\circ\pi
\end{multline}
as maps $ \O(k)\otimes \bigotimes_{i=1
}^k( \O(l_i)\otimes \bigotimes_{j=1}^{l_i} \O(n_i^j))\to \O(n)$,
where $\pi$ permutes the factors of the $\O(l_i)$ to the right of $\O(k)$.
Notice that we chose to index the operad maps, since this will make the operations easier to dualize.
The source  and target of the map are then determined by the length $k$ of the index, the indices $n_i$ and their sum.

For an operad, aka.\ symmetric operad, one adds the data of an $\SS_n$ action on each $\O(n)$ and demands that the map $\gamma$ is equivariant, again in the naturally appropriate sense, see Example \ref{endex} or \cite{MSS,woods}.

\begin{df}
A (non--$\Sigma$) operad is called {\em locally finite}, if any element $a_n\in \O(n)$ is in the image of only finitely many $\gamma_{k;n_1,\dots,n_k}$,
where $n=\sum_i n_i$.
\end{df}

\begin{lem}
\label{locfinOlem}
If $\O(0)$ is empty in $\Set$ or zero in an Abelian category, then $\O$ is locally finite.
\end{lem}
\begin{proof}
There are only finitely many partitions of $n$ into $k$ non--zero elements.
\end{proof}

\begin{ex}
If we consider planar planted leaf labeled trees and allow leaf vertices, that is vertices with no inputs, then there is an $\Coop(0)$, namely trees without any leaves, but the operad is still locally finite.
Indeed the number of vertices is conserved under gluing, so there are only finitely many possible pre--images.
\end{ex}

\subsubsection{Morphisms}
Morphisms of (pseudo)--operads $\O$ and $\P$ are given by a family of morphisms $f_n:\O(n)\to\P(n)$ that commute with the structure maps. E.g.\ $f_n(a)\circ_i^{\P} f_m(b)=f_{n+m-1}(a\circ_i b)$.
If there are symmetric group actions, then the maps $f_n$ should be  $\SS_n$ equivariant.

\begin{ex}
\label{mapexamples}
If we consider the operad of rooted leaf labeled trees $\O$ there is a natural map to the operad of  corollas $\P$ given by $\tau \mapsto \tau / E(\tau)$,
where $\tau / E(\tau)$ is the corolla that results from
contracting all edges of $\tau$.  This works in the planar and non--planar version as well as in the pseudo-operad setting, the operad setting and the symmetric setting. This map contracts all linear trees and identifies them with the unit corolla $\dottree$, which has one leaf.
Furthermore, it restricts to operad maps for the sub--operad of $k$-regular or at least $k$-valent trees.

An example of interest considered in \cite{Gont} is the map restricted to planar planted $3$-regular tress (sometimes called binary).
The kernel of this map is the operadic ideal generated by the associativity equation between the two possible planar planted binary trees with three leaves.
\end{ex}

\subsubsection{Units}
The two notions of pseudo-operads and operads become equivalent if one adds a unit.

\begin{df} A \emph{unit} for a pseudo-operad is an element $u\in \O(1)$ such that
$u\circ_1 b=b$  and $b\circ_i u=b$ for all $m$,
for all $1\leq i\leq m$ and all $b\in\O(m)$.

A unit for an operad is an element $u\in \O(1)$ such that
\begin{equation}
\label{uniteq}
\gamma_{1;n}(u;a)=a\text{ and }
\gamma_{n;1,\dots,1}(a;u,\dots,u)=a
\end{equation}
\end{df}

There is an equivalence of categories between unital pseudo--operads and unital operads. It is given by the following formulas: for $\a\in \O(n), b\in \O(m)$ with $n\geq i$
\begin{equation}
\label{unitcircieq}
a\circ_ib=\gamma_{n,1,\dots,1,m,1,\dots,1}(a;u,\dots,u,b,u,\dots,u) \text{ $b$ in the $i$--th place}
\end{equation}
and vice--versa for $a\in \O(k)$:
\begin{equation}
\gamma_{k;n_1,\dots,n_k}(a;b_1,\dots,b_k)=((\dots ((a\circ_k b_k)\circ_{k-1} b_{k-1}) \dots )\circ_{1} b_1)
\end{equation}

Morphisms for (pseudo)--operads with units should preserve the unit.
\begin{rmk}
\label{argebrarmk}
The component $\O(1)$ always forms an algebra via $\gamma:\O(1)\otimes \O(1)\to \O(1)$.
If there is an operadic unit, then this algebra is unital.
More precisely, the algebra is over $R=Hom(\unit,\unit)$.
In the case of operads in $k$-vector spaces the algebra is a $k=Hom(k,k)$--algebra, in the case of operads in Abelian groups, this is a $\Z$ algebra and in
in the case of operads in sets, being an algebra reduces to being monoid.
\end{rmk}

\begin{rmk}
 In order to transport $Set$ operads to Abelian groups and vector spaces or $R$-modules, we can consider the free Abelian groups generated by the sets $\O(n)$ and  extend coefficients.
 In particular, we can induce co--operads in different categories, by extending coefficients, say from $\Z$ to $\Q$ or a field $k$  in general.
 More generally, we can consider,  the adjoint to the forgetful functor \cite{kellybook} for any enriched category.
\end{rmk}

\subsubsection{Non--connected version of an operad}

\begin{as}
\label{opas}
In this section for concreteness, we will assume that we are in  an  Abelian monoidal categories  whose bi--product distributes over tensors.
and use $\bigoplus$ for the biproduct. The usual categories to keep in mind are those of Abelian groups $\Ab$, graded Abelian groups $g\Ab$, $\kVect$ vector spaces over a field $k$.
If $\O$ is a $\Set$ operad, we tacitly consider its Abelianization, that is $\{\Z\O(n)\}$ which we still denote by $\O$.
We will also assume that $\O$ is locally finite.
\end{as}

In \cite[\S A.1]{KWZ}, we introduced the non--connected operad corresponding to an operad.
\begin{equation}
\O^{nc}(n):=\bigoplus_{k,(n_1,\dots,n_k):\sum_i n_i=n}\O(n_1)\odo \O(n_k)
\end{equation}
There is a natural multiplication $\mu:=\otimes:\O^{nc}(n)\otimes \O^{nc}(m)\to \O^{nc}(n+m)$ which identifies $\bigoplus_n\O^{nc}(n)=\bar T(\bigoplus_n\O(n))$ that is the reduced tensor algebra.
The pseudo--operad structure naturally extends, see \cite{KWZ} and the operad structure extends via
\begin{multline}
\label{opnceq}
 \gamma^{nc}(a_k\otimes b_l; a_{n_1} \kdk  a_{n_k}, b_{m_1}\kdk b_{m_l})\\=\gamma(a_k;  a_{n_1}\kdk  a_{n_k})\otimes \gamma(b_l; b_{m_1}\kdk b_{m_l})\end{multline}
for $a_i, b_i\in \O(i)$.
More precisely,
$\gamma^{nc}\circ (\mu\otimes id) \circ \pi=\mu\circ (\gamma\otimes \gamma)$ where $\pi$ is the permutation that permutes the correct factor into the second  place.
The association $\O\to \O^{nc}$ is functorial.

\begin{rmk}
In fact, the operad structure on $\O^{nc}$ is up to permutation exactly the operation that appears in l.h.s.\ of the associativity equation \eqref{assopeq}. The nc--version also appears naturally in the formulation in terms of Feynman categories, cf.\ \cite[\S\ref{P2-enrichedpar}]{HopfPart2}.
\end{rmk}
\begin{ex}
In the example of planar planted trees, $\O^{nc}(n)$ are planar planted forests, these are ordered collections of planar planted trees,
with $n$ leaves. In particular $\O(n_1)\odo \O(n_k)$ contains  forests with $k$ trees. The $i$-th tree has $n_i$ leaves and the total number of leaves is $n:=\sum_{i=1}^kn_i$.
The operad gluing grafs a forest of $n$ trees onto a forest with $n$ leaves by grafting the $i$--th tree to the $i$--th leaf. The pseudo--operad structure does one of these graftings and leaves the other trees alone, but shifts them into the right position, see \cite{KWZ} or use \eqref{unitcircieq}.
\end{ex}

\subsubsection{Bi--grading and algebra over monoid structure}
There is another way to view the operad $\{\O^{nc}(n)\}$. First notice that there is an internal grading by tensor length. Set $\O=\bigoplus_k \O(k)$ and let $\O^{nc}(n,k)\subset \O^{nc}(k)$ be the tensors of length $k$.
Then $\O^{nc}(n)=\bigoplus_k\O^{nc}(n,k)$ an set $\O^{nc}=\bigoplus_n\O^{nc}(n)$. Summing the $\gamma_{k;n_1\kdk n_k}$ over the partitions $(n_1\kdk n_k)$ with fixed $k$ and $n$, one obtains  maps
$\gamma_{k,n}:\O(k)\otimes \O^{nc}(n,k)\to \O(n)$. Further summing over the $\gamma_{k,n}$ over $k$, we obtain a map $\gamma_n:\O\otimes \O^{nc}(n)\to \O(n)$, lastly summing over $n$, we obtain a map $\gamma:\O\otimes \O^{nc}\to \O$.
Note that by un--bracketing tensors there is an identification:
$(\O^{nc})^{nc}=\bar T(\bar T \O)\simeq \bar{T}\O= \O^{nc}$.

\begin{prop}
\label{opmonoidprop}
Under  Assumption \ref{opas}, the associativity of $\gamma$ implies that $\gamma^{nc}$ induces
 an associative monoid structure on $\O^{nc}$ and $\O$ is a left module over $\O^{nc}$ via $\gamma$.
\end{prop}
\begin{proof}
Indeed $\tilde\gamma^{nc}:\O^{nc}\otimes \O^{nc}\stackrel\simeq\to \O^{nc}\otimes (\O^{nc})^{nc}\stackrel{\gamma^{nc}}{\to} \O^{nc}$ is a multiplication. The multiplication $\tilde\gamma^{nc}$ is
associative by the associativity of $\gamma^{nc}$, which follows from that of $\gamma$ via the definition.
 The associativity diagram corresponding to \eqref{assopeq} is
\begin{equation}
\label{assocdiagrameq}
\xymatrix{
\O\otimes \O^{nc}\otimes \O^{nc}\ar[r]^-{id\otimes \gamma^{nc}}\ar[d]^{\tilde\gamma\otimes id}& \O\otimes \O^{nc}\ar[d]_{\gamma}\\
\O\otimes \O^{nc}\ar[r]_{\gamma}&\O
}
\end{equation}
which is, at the same time, the statement that $\O$ is a left module over $\O^{nc}$.\end{proof}

\subsection{Co--operads}
\label{freecooppar}
The relevant constructions will all involve the dual notion to  operads, that is co--operads. In terms of trees this provides the transition from grafting to cutting.

\subsubsection{Non-\SSigma{} co--operads}

Dualizing the notion of an operad, we obtain
 the notion of a co--operad. That is, there are structure maps dual to the $\gamma_{k;n_1\kdk n_k}$   \eqref{mayeq}  for all $n,k$ and partitions $(n_1,\dots,n_k)$ of $n$:
  \begin{equation}
\check\gamma_{k;n_1, \dots, n_k}:\Coop(n)\to   \Coop(k)\otimes \Coop(n_1)\odo \Coop(n_k)
\end{equation}
 which satisfy the dual relations to \eqref{assopeq}.
That is,
\begin{multline}
\label{coopcoasseq}
id \otimes \check\gamma_{l_1;n^1_1,\dots, n^1_{l_1}}\otimes \check\gamma_{l_2;n^2_1,\dots, n^2_{l_2} } \odo  \check\gamma_{l_k;n^k_1,\dots, n^k_{l_k}}
\circ\check\gamma_{k;n_1,\dots, n_k}\\=
\pi\circ ( \check\gamma_{k;l_1,\dots, l_k}\otimes id^{\otimes l})\circ\check\gamma_{l;n^1_1,\dots,n^1_{l_1},n^2_{1},\dots, n^2_{l_2},\dots, n^k_{1},\dots, n^k_{l_k}}
\end{multline}
as maps $\Coop(n)\to \Coop(k)\otimes \bigotimes_{i=1}^k (\Coop(l_i)\otimes \bigotimes_{j=1}^{l_i} \Coop(n_i^j))$, for any
 $k$-partition $(n_1,\dots,n_k)$  of $n$ and $l_i$-partitions $(n_1^i,\dots, n_{l_i}^i)$ of  $n_i$. Either side of the relation determines these partitions and hence determines the other side.
Here $l=\sum l_i$ and $\pi$ is the permutation permuting the factors  $\Coop(l_i)$ to the left of the factors $\Coop(n_i^j)$.

The example of a co--operad that is pertinent to the three constructions is given by dualizing an operad $\O$.   In particular,  if $\O$ is an operad in (graded) Abelian groups then $\Coop(n)=\underline{Hom}(\O(n),\Z)$, that is the
group homomorphisms considered as a (graded) Abelian group, is the dual co--operad.
Note that if $\O(n)$ is just a set, then $Hom_{\Set}(\O(n),\Z)$ is an Abelian group and coincides with $\inthom_{\Ab}(\Z[\O(n)],\Z)$, where $\Z[\O(n)]$ is the free Abelian group on $\O(n)$.
In $\dgVect$ the dual co--operad is $\Coop(n)=\underline{Hom}(\O(n),k)$.
 This construction works in any closed monoidal category:   $\check\O(n)=\underline{Hom}(\O,\unit)$, where $\underline{Hom}$ is the internal Hom.

\begin{lem}
The dual of an operad, $\check\O(n)=\underline{Hom}(\O,\unit)$, in a closed monoidal category is a co--operad and this association is functorial. Likewise, if the objects in the monoidal category are graded, the graded dual of $\O$ is also functorially a co--operad.
\end{lem}
\begin{proof}
The association $\O(n) \to \Coop(n)$ is contravariant and all the diagrams to check are the dual diagrams. The functoriality is straightforward.
\end{proof}

\begin{rmk}
In  a linear category the maps $\check\gamma_{k,n_1\kdk n_k}$ can be $0$. If this is not the case, e.g.\ in $\Set$, one can weaken the conditions to state that the $\check\gamma_{k,n_1\kdk n_k}$ are partially defined functions, \eqref{coopcoasseq} holds whenever it is defined, and  the r.h.s.\ exists, if and only if the l.h.s.\ does.

\end{rmk}

\subsubsection{Examples based on free constructions on $\Set$ operads}

\label{freesetpar}
 If the operad $\{\O(n)\}$ is a $\Set$ operad then we can consider the free Abelian groups $\Z[\O(n)]\subset Hom(\O(n),\Z)$ or (free) vector spaces $k[\O(n)]\subset Hom(\O(n),k)$ in $\Ab$ or in $\kVect$.
 In this case, there is standard notation. For an element $\tau\in \O(n)$, we have the characteristic function $\delta_\tau:\O(n)\to \unit$ given by $\delta_{\tau}(\tau')=1$ if $\tau=\tau'$ and $0$ else.
 Then an element in the free Abelian group or the vectors space generated by $\O(n)$ is just a finite formal sum $\sum_{i\in I}n_i\delta_{\tau_i}$.
 By abuse of notation this is often written as $\sum_{i\in I}n_i \tau_i$.
 The dual on these spaces are then again formal sums of characteristic functions $\sum_{\alpha \in J}n_\alpha \tau^*_{\alpha}$, where the $\tau^*=ev_{\tau}$ are the evaluation maps at $\tau$.  This is course the known embedding $V\to (\check V)^{\vee}$ for vector spaces and the fact that the dual of a direct sum is a  product; see also \S\ref{assumptions}.

 \begin{rmk} We collect the following straightforward facts:
\begin{enumerate}
\renewcommand{\theenumi}{\roman{enumi}}
\item { If the $\O(n)$ are finite sets or finite free Abelian groups then the formal sums reduce to finite formal sums.} Dropping the superscript $*$, we again can identify elements of $\Coop(n)$ as finite formal linear combinations.

\item In the general case: {Finite formal  sums are a sub--co--operad if and only if $\O(n)$ is locally finite.}

\item The analogous statements hold for graded duals of free graded $\O(n)$. If the internal grading is preserved by $\gamma$ and the bi--graded pieces of $\O^{nc}$ are finite dimensional, then the graded dual has only finite sums. This is usually the starting point for the constructions mentioned in the introduction.
\end{enumerate}
\end{rmk}

\begin{ex}
In the case of leaf labeled rooted trees, the sets $\O(n)$ are finite precisely if one excludes vertices of valence $2$. Otherwise, the $\O(n)$ are infinite.
There is an internal grading by the number of vertices. This grading is respected by the operad and hence the co--operad structure. Adding the internal grading, the graded pieces $\O_r(n)$ with $r$ vertices are finite. The bi--graded pieces $\O_r^{nc}(n)$ are all finite dimensional. This is a consequence of the fact that $\O(0)$ has only positive internal degree ---there are no non--vertex trees without leaves. Using the vertex grading, one can also directly see that the operad is locally finite and the graded dual are only finite sums.
\end{ex}

\begin{rmk}
There is also the notion of a partial or colored operad. This means that there is a restriction on the gluing morphisms that the inputs can only be glued to like outputs, cf.\ e.g.\ \cite[\S\ref{P2-coloroppar}]{HopfPart2}.
The dual of such a partial structure is actually a co--operad. The key point is that in the dual there are no restrictions as one only decomposes what has previously been composed; see \cite{woods}.
\end{rmk}

\begin{ex}[The overlapping sequences (co)--operad]
\label{overlapex}
Consider a set $S$. Let $\Coop(n)=$ the free Abelian group (or vector space) on the set of finite sequences of length $n+1$ in $S$.  We define the co--operad structure as the decomposition of the set into overlapping ordered partitions:
\begin{multline}
\label{overlapeq}
\check\gamma_{k,n_1,\dots,n_k}(a_0;a_1 \kdk a_n-1; a_n)=\\
\sum_{0=i_0<i_1<\cdots<i_{k}=n} (a_0=a_{i_0};a_{i_1},a_{i_2}, \dots; a_{i_k}=a_n)\\\otimes (a_0=a_{i_0}; a_1 \kdk a_{i_1-1};a_{i_1}) \otimes (a_{i_1}; a_{i_1+1}\kdk a_{i_2-1};a_{i_2})\otimes \dots \\
\dots\otimes (a_{i_{k-1}};a_{i_{k-1}+1}\kdk a_{n-1};a_{i_k}=a_n)
\end{multline}
The co--operad structure is dual to the free extension of the partial (a.k.a.\ colored) $\Set$ operad structure, where  if $a_i=b_0$ and $a_{i+1}=b_{n+1}$:
\begin{multline}
(a_0;a_1 \kdk a_n; a_{n+1})\circ_i (b_0; b_1 \kdk, b_m ;b_{m+1})=\\(a_0;a_1 \kdk a_{i-1},a_i=b_0,b_1 \kdk b_{m}, b_{m+1}=a_{i+1}, a_{i+2} \kdk a_n;a_{n+1})
\end{multline}
This gives the connection of Goncharov's Hopf algebra, to Joyal duality, see Appendix \ref{Joyalapp}, see also \cite[\S\ref{P2-anglepar}]{HopfPart2}. It  is why we chose the notation using semi--colons as it gives the colors, and the fact that there are double base--points, the first and the last element.

Note that this co--operad has sub--co--operads given by fixing a particular sequence and considering all subsequences.

This is the example relevant for Goncharov's Hopf algebra and that of Baues when suitably shifted, see Remark \ref{shiftedrmk}. It will be  further discussed in \S\ref{overlappar} below. A more in depth consideration explaining the existence of this co--operad is given in \S\ref{simplicialpar}.

\end{ex}
\begin{rmk}
Note that the indices of the sequence are given by $\check\gamma(0,\dots,n)$ where $(0,\dots,n)$ is thought of as a sequence in $\N_0$. The iterations of the indices, then correspond to splitting or splicing intervals in $\N_0$. This makes contact with the operad structure on simplices and is the basis for the simplicial considerations section \S\ref{simplicialpar}. The partial operad structure becomes natural when considering Feynman categories where the partial operad structure corresponds to the partial structure of composition of morphisms in a category. The sequences can also be thought of as a decorations of angles of  corollas. This is shown in \cite[Figure \ref{P2-angle}]{HopfPart2}  and more precisely as decoration in the technical sense for Feynman categories \cite{decorated}, see \cite[\S\ref{P2-anglepar}]{HopfPart2}.
\end{rmk}

\subsubsection{Co--operadic co--units}

A morphism $\eps:\Coop(1)\to \unit$  is a left co--operadic co--unit if its extension  $\eps_1$ by $0$ on the $\Coop(n),n\neq 1$ satisfies\footnote{Here and in the following, we suppress the unit constraints in the monoidal category and tacitly identify $V\otimes \unit\simeq V\simeq \unit\otimes V$.}:
\begin{equation}
\label{lcouniteq}
\sum_k (\eps_1\otimes id^{\otimes k})\circ\check\gamma =id
\end{equation}
and a right co--operadic unit if
\begin{equation}
\label{rcouniteq}
\sum_k (id \otimes \eps_1^{\otimes k})\circ\check\gamma =id
\end{equation}
A co--operadic co--unit is a right and left co--operadic co--unit.
We will use $\eps_1:\Coop\to \unit$ for its extension by $0$ on
all $\Coop(n):n\neq 1$.

\begin{rmk}
Note, if there is only one tensor factor on the right, then the left factor has to be $\Coop(1)$ by definition. If $\eps_1$ would have support outside $\Coop(1)$,
the  $\check\gamma$ would have to vanish on the right side for all elements having  that left hand side, which is rather non--generic. This is why we assume $\eps_1$ vanishes outside $\Coop(1)$.
It is also the notion naturally dual to an operadic unit.
\end{rmk}

\begin{lem} The dual of a unital operad is a co--unital co--operad and this association is functorial.

\end{lem}
\begin{proof} A unit
 $u\in \O(1)$, can be thought of as a map of $u:\unit\to \O(1)$,
where $\unit$ is $\Z$ for Abelian groups or in general the unit object, e.g.\
$k$ for $\kVect$.
Its dual is then a morphism $\check u:=\Coop(1)\to \unit$.
Now,
$\check u:=\eps$ is a left/right co--operadic co--unit if it satisfies he equations \eqref{lcouniteq} and \eqref{rcouniteq}, but these are the diagrams dual to the equations
(\ref{uniteq}). Functoriality is straightforward.

 \end{proof}

\subsubsection{Morphisms}
Morphisms of co--operads $\check \O$ and $\check \P$ are given by a family of morphisms $f_n:\check \O(n)\to\check\P(n)$ that commute with the structure maps \begin{equation}\check \gamma^{\check\P}_{k;n_1,\dots,n_k}\circ f_n=
(f_k\otimes f_{n_1}\odo f_{n_k})\circ \check \gamma^{\check\O}_{k;n_1,\dots,n_k}\end{equation}

\subsubsection{Completeness and Finiteness Assumptions}
\label{assumptions}

\begin{as}[Completeness Assumption]
\label{completenessassumption}
If the monoidal category in which the co--operad lives is complete and certain limits (in particular, products) commute with taking tensors, then by summing over $k$ and the $k$--partitions of $n$, we
can define
 \begin{equation}
 \check\gamma_{n}:\Coop(n)\to \prod_{k}\prod_{(n_1,\dots,n_k):\sum_{i=1}^k n_i=n}
 \Coop(k)\otimes \Coop(n_1)\odo \Coop(n_k).
 \end{equation}
\end{as}

For the applications, we will use free algebras, which are based on  finite products of the $\Coop(n)$. In the  Abelian monoidal categories of (graded) vector spaces
$\kVect$, differential graded vector spaces $\dgVect$,    Abelian groups $\Ab$, or $\gAb$ graded Abelian groups,  these finite product are direct sums.
In order to write down the multiplication and the co--multiplication, we will need the maps $\check\gamma_n$ to be locally finite.

\begin{df}
We call $\{\Coop(n)\}$ {\em locally finite} if for any $a_n\in \Coop(n)$ there are only finitely many $k$ partitions of $n$ with
$\check\gamma_{k;n_1,\dots,n_k}(a_n)\neq 0$.
\end{df}

\begin{lem}
If there is no $\Coop(0)$, then $\{\Coop(n)\}$ is locally finite.
\end{lem}
\begin{proof} As in Lemma \ref{locfinOlem},
there are only finitely many partitions $(n_1\kdk n_k)$ of $n$ with $n_i\geq 1$.
\end{proof}
This implies that  in the limits and the limits reduce to  finite limits as there are only finitely many maps.

\begin{as}[Basic assumption]
\label{basicas}
We will assume that the co--operads are locally finite and  that the co--operads are in an Abelian category with bi--product $\oplus$ which are distributive with the tensor product.
\end{as}
Set $\Coop=\bigoplus_m\Coop(m)$, summing over the $m$ we obtain morphism
$
\check\gamma: \Coop\to  \Coop \otimes \bigoplus_{k}\Coop^{\otimes k}=\Coop\otimes \bar T\Coop
$.
The right hand side is actually multi--graded. Set

\begin{equation}
\label{nccoopeq}
\Coop^{nc}(n):=\bigoplus_{k,(n_1,\dots,n_k):\sum_i n_i=n}\Coop(n_1)\odo \Coop(n_k)
\end{equation}

\subsection{Bi--algebra structure on the non--connected dual of a non--$\Sigma$ operad}
\label{nonsigmasec}
\subsubsection{Non--connected co--operad}
 Just like for an operad, we can define a non--connected version for a co--operad. For trees this is again the transitions to forests.

Composition of tensors, or un--bracketing, gives a multiplication $\mu:\Coop^{nc}(n)\otimes \Coop^{nc}(m)\to \Coop(n+m)$.
Set
$\Coop=\bigoplus_n\Coop(n)$ and let $\B=\bar T\Coop=\bigoplus_{n\geq 1}\Coop^{nc}(n)$ be the free algebra on $\Coop$, then $\mu$ is just the free multiplication.
This is indeed again a co--operad, which we will use to generalize in \S\ref{gencooppar}.
\begin{prop}
Under the basic assumption, $\Coop^{nc}$ is a co--operads with respect to $\check\gamma^{nc}$ defined by
\begin{multline}
\label{freecompatexpleq}
\check\gamma^{nc}(a_n\otimes b_m)=\\ \sum_{\mbox{\tiny\begin{tabular}{l}$k,(n_1\kdk n_k):\sum_i n_i=n$\\$l, (m_1\kdk m_k):\sum_i m_i=m$\end{tabular}} }
\sum (a_{k}^{(0)}\otimes b_{l}^{(0)})\otimes a_{n_1}^{(1)}\odo a_{n_k}^{(k)}\otimes b_{m_1}^{(1)}\odo b_{m_l}^{(l)}
\end{multline}
using a multi--Sweedler notation for the $\check\gamma$ and indication the co--operadic degree by subscripts. More precisely,
 \begin{equation}
 \label{freecompateq}
\check \gamma^{nc}\circ \mu=(\mu\otimes \mu)\circ \pi\circ (\check\gamma\otimes \check\gamma)f
 \end{equation}
 where $\pi$ permutes the first factor of the second $\check\gamma^{nc}$ into the second place.

If $\Coop$ is the dual of a locally finite operad $\O$, then $\Coop^{nc}$ is the dual co--operad of the operad $\O^{nc}$.
\end{prop}
\begin{proof}
Using the equation iteratively, we see that the components of $\check\gamma^{nc}$ on $\Coop^{nc}(n)$ are $\Coop(n_1)\odo \Coop(n_k)\to (\bigoplus_{(l_1\cdc l_k)}\Coop(l_1)\odo\Coop(l_k))\otimes \bigoplus
\Coop(n_1^1)\odo \Coop(n^1_{l_1})\otimes \bigoplus \Coop(n^2_1)\odo \Coop(n^2_{l_2})\odo \bigoplus\Coop(n^k_1)
\odo \Coop (n^k_{l_k})\subset \Coop^{nc}(l)\otimes \Coop^{nc}(n_1)\odo \Coop^{nc}(n_l)$ where the sums are over partitions and $\sum_i l_i=l, \sum_k n_k^i=n_k$ and $\sum_k n_k=n$

The co--associativity for the co--operad structure $\check\gamma^{nc}$ follows readily from that of $\check\gamma$.
The last statement follows from the fact that \eqref{freecompatexpleq} is the dual of \eqref{opnceq}.
\end{proof}

\begin{ex}[Bar of an operad/algebra]
\label{cobarex}
A natural way to obtain a co--operad from an operad it given  by the operadic bar transform, see e.g.\ \cite{MSS}. One can then consider the free algebra on this co--operad. This is much bigger than just doing the tensor algebra on the dual of an operad. A reasonably small version is provided by the bar construction of an algebra, which is a co--algebra, that can also be thought of as a co--operad. This parallel to the fact that an algebra is an operad with only $\O(1)$; see also \cite[\S\ref{P2-coopex}]{HopfPart2} and \cite[\S\ref{P2-cobarpar}]{HopfPart2}.
\end{ex}

\subsubsection{From non-\SSigma{} co--operads to bi--algebras}

There is another way to interpret the co--operad structure on $\Coop^{nc}$ in which $\check\gamma^{nc}$ becomes a co--multiplication.
Let $\Coop=\bigoplus_{n\geq 1} \Coop(n)$, $\B=\bar T\Coop$ and $\B'=T\Coop=\unit \oplus \B$ be the free and free unital associative algebras on $\Coop$, then we can decompose $\B=\bigoplus_n \B(n)$:
 \begin{equation}
  \B(n)=\bigoplus_{k}\bigoplus_{(n_1,\dots, n_k):\sum_i n_i=n}\Coop(n_1)\odo \Coop(n_k)=\Coop^{nc}(n)
 \end{equation}
 The free multiplication is composition of tensors and on components is given by $\mu_{n,m}:\B(n)\otimes \B(m)\to \B(n+m)$.

For $\B'$ we let $\B(0)=\unit$ and tacitly use the unit constraints $u_R:X\otimes \unit \stackrel{\sim}{\to} X$ and $u_L: \unit \otimes X \stackrel{\sim}{\to} \unit$
to shorten any tensor containing a unit factor, and hence make $\B'$ unital. This defines the unit components of $\mu$: $\mu_{0,n}:\B(0)\otimes \B(n)\to \B(n)$ and $\mu_{n,0}:\B(n)\otimes \B(0)\to \B(n)$.

In this notation, the maps $\check\gamma_{k,n_1\kdk n_k}$ can be seen as maps: $\Coop(n)\to  \Coop(k)\otimes \Coop^{nc}(n)=\Coop(k)\otimes \B(n)$.
Summing the $\check\gamma_{k;n_1,\dots,n_k}$ over all non--empty $k$ partitions of $n$, we obtain a map
 \begin{equation}
\check\gamma_{k,n}: \Coop(n)\to\Coop(k)\otimes \B(k)
\end{equation}

The extension of $\check\gamma$ to $\Coop^{nc}$ using \eqref{freecompateq} gives maps
\begin{equation}
\label{freeDeltakneq}
\Delta_{k,n}:\B(n)\to \B(k)\otimes \B(n)
\end{equation}

If we sum the $\Delta_{k,n}$ over $k$ and $n$ we obtain maps $\Delta_n=\sum_k \Delta_{k,n}$ which define $\Delta=\sum_n\Delta_n$
\begin{equation}
\Delta:\B\to \B\otimes \B
\end{equation}

In $\B'$ we let $\B'(0)=\unit$ and $\Delta$ extends  to $\B'$ via $\Delta(1)=1\otimes 1$, for $1=id_\unit\in Hom(\unit,\unit)=\unit$.
Thus, on $\B'(0,0)$, we have the additional component $\Delta_{0,0}:\B'(0,0)\to \B'(0)\otimes \B'(0)$.
This also extends the co--operadic structure of $\Coop^{nc}$ to $\B'=\unit\oplus\Coop^{nc}$.

\subsubsection{Grading}
 We  have the grading by co-operadic degree $\deg$ with degree of $\Coop(n)$ being $n$.
  In this grading, the graded dual of $\O$ is $\Coop$:
$(\bigoplus_n \O(n))^{\vee}$ $=\bigoplus_n \Coop(n)$.
On $T\Coop$, the natural degrees are additive degrees, i.e.\ for $a\in \Coop^{nc}(n)$, $\deg(a)=\sum n_i= n$. We set the degree of elements in $\unit$ to be zero.
This coincides with the co--operadic grading for $\Coop^{nc}$.
We also have the grading in $\B$ by tensor length $\length$.
This gives a double grading $B=\bigoplus_{n,p}\B(n,p)$ where
\begin{equation}
\B(n,p)=\Coop^{nc}(n,p)=\bigoplus_{(n_1,\dots,n_p), \sum n_i=n}\Coop(n_1)\odo\Coop(n_p)
\end{equation}
In $\B'$ the unit is defined to have length $0$: $\unit=\B'(0,0)$.
Using the bi--grading the components of $\mu$ maps are:
\begin{equation}
\label{mupneq}
\mu:\B'(n_1,p_1)\otimes \B'(n_1,p_2)\to \B'(n_1+n_2,p_1+p_2)
\end{equation}
and the non--vanishing components of $\Delta$ are:
\begin{equation}
\label{deltapkneq}
\Delta_{k,n}:\B(n,p)\to \B(k,p)\otimes \B(n,k) \qquad 1\leq p\leq k\leq n
\end{equation}
the restriction comes from the fact that  $\Delta$ is an algebra morphism, and hence it does not change the length of the first factor. By definition the degree of the first factor is the length of the second factor and hence the $\Delta_{k,n}$ are the only non--zero components, so that $\Delta_k=\sum_{n\geq k}\Delta_{k,n}$ and $\Delta=\sum_{k,n:k\leq n}\Delta_{k,n}$.
On $\B'(0,0)$, we have the additional component $\Delta_{0,0}:\B'(0,0)\to \B'(0,0)\otimes \B'(0,0)$.

We define the weight  grading on $\B'$ to be given by $\weight=\deg-\length$. An additional internal grading by considering operads in $\dgVect$ or $g\Ab$ can be incorporated by
adding the external and internal gradings; as usual.

\begin{prop}
  $\B$ is a bi--algebra and $\B'$ is a unital bi--algebra both graded with respect to $\wt$. This association is functorial.
\end{prop}
\begin{proof}
We have to check that $\Delta$ and $\mu$ satisfy the bi--algebra equation that is,
if $c$ is decomposable $c=\mu(a\otimes b)=a\otimes b$, then it must satisfy
\begin{equation}
\label{bialgeq}
\Delta\circ\mu (a\otimes b)= (\mu\otimes\mu) \circ \pi_{2,3}\circ(\Delta(a)\otimes \Delta(b))
\end{equation}
where $\pi_{2,3}$ permutes the second and third factor in $\B'\otimes \B'\otimes \B'\otimes \B'$.
But, this equation is \eqref{freecompateq} when interpreted in terms of $\B$.
 Co--associativity follows from \eqref{coopcoasseq}. The extension to $\B'$ follows readily.

 For the grading, looking at \eqref{mupneq}, we obtain $n_1+n_2-p_1-p_2$ as the degree for both sides of the multiplication. For the co--product, we see that on both sides of \eqref{deltapkneq} the degree is $n-p=k-p+n-k$.
 The functoriality is clear.
\end{proof}

\begin{rmk}
To obtain the bi--algebra, we could have alternatively just defined $\B=\bar T\Coop^{nc}$,  as the free algebra, defined $\rho=\check\gamma:\Coop\to \Coop\otimes \B$ and then extended $\rho$  to all of $\B$ as $\Delta:\B\to \B \otimes\B$ via \eqref{bialgeq}, without defining the co--operad structure on $\Coop^{nc}$.
Note that $\rho$ makes $\Coop$ into a co--module over $\B$ and $\B'$. The way, it is set up now ---co--operad with multiplication--- will allow us to generalize the structure in \S\ref{gencooppar}.
\end{rmk}

\begin{rmk}[Shifted  version]
\label{shiftedrmk}
One  obtains the weight naturally using the suspensions $\Coop(n)[1]$ of the $\Coop(n)$ in  \eqref{nccoopeq}. The suspended operadic degree is the weight.  This is analogous to the conventions of signs in the graded pre--Lie structure and in general to using odd operads   \cite{KWZ}.

This is also very similar to the co--bar transform $\Omega C$ for a co--algebra $C$, but
without the differential. The differential is instead replaced by the co--operad structure, or the co--product.
A similar situation is what happens in Baues' construction. Here one can think of a co-bar transform of an algebra of simplicial objects, where the simplicial structure gives the (co)operad structure, see \S\ref{simplicialpar}.

\end{rmk}

\begin{rmk}
This shifted version only a small part of the operadic co--bar construction which would have components for any tree and the ones in the shifted construction correspond in a precise sense only to level trees that are of height 2. The two constructions are related by enrichment of Feynman categories and $B_+$ operators. We will not go in to full details here and refer to \cite[\S3,\S4]{feynman} and future analysis.
\end{rmk}

\subsubsection{Co--module structure}
Dual to \eqref{assocdiagrameq} we can write the co--as\-so\-cia\-tivity of $\check\gamma$ as

\begin{equation}
\xymatrix{
\Coop\otimes \Coop^{nc}\otimes \Coop^{nc}& \ar[l]_-{id\otimes
\check {\gamma}^{nc}} \Coop\otimes \Coop^{nc}\\
\Coop\otimes \Coop^{nc}\ar[u]^{\check{\gamma}\otimes id}&\Coop\ar[l]_{\check\gamma}\ar[u]_{\check\gamma}
}
\end{equation}
From this, we obtain the dual to \S\ref{opmonoidprop}.
\begin{prop}
\label{coopcomonoidprop}
Under the  basic assumption, the co--as\-sociativity of $\check\gamma$ implies that
$\Coop^{nc}$ is a co--associative co--monoid induced by $\check\gamma^{nc}$ and $\Coop$ is a left co--module over $\Coop^{nc}$ via $\check\gamma$. \qed
\end{prop}

\subsubsection{Finiteness, $\Coop(0)$ and co--modules}
If $\Coop(0)$ is empty or zero, then $\{\Coop(n)\}$ is locally finite. If there is an internal grading for the $\Coop(n)$ that is preserved under $\check\gamma$ and positive on $\Coop(0)$ then again,
$\{\Coop(n)\}$ is locally finite. In these cases, there is no problem in considering $\{\Coop(n)\}, n\geq 0$.
Generally, $\{\Coop(n)\},\geq 1$ is a sub--co--operad of  $\{\Coop(n)\}, n\geq 0$.

\begin{rmk}
One may in this case also consider an $\Coop(0)$ of rooted trees without leaves, but not without vertices.
The leaf labeled trees are then replaced by rooted trees in general.
Algebras over this are algebras over the operad together with a module. Dually, this yields the co--algebra over the co--operad structure.
\end{rmk}

\subsubsection{Operadic units, co--operadic co--units and bi--algebraic co--units}

If $\eps$ is a co--unit for the co--operad $\{\Coop(n)\}$, $\eps:\Coop(1)\to \unit$,  we extend it by $0$ to $\eps_1:\Coop\to\unit$. We further extend $\eps_1$ to the co--operad $\Coop^{nc}$   by  $\eps_{tot}=\bigoplus_{n\geq 1} \eps_1^{\otimes n}$
and to $\B'$ by $\eps_1^{\otimes 0}=id:\B(0)\to \unit$ that is $\eps_{tot}=\bigoplus_{n\geq 0} \eps_1^{\otimes n}:\B'\to \unit$.

\begin{prop}
\label{freeco--unitprop}
The map $\eps_{tot}: \O^{nc}\to \unit$ is a co--operadic co--unit for $\check\gamma$.
As a map $\eps_{tot}:\B'\to \unit$
is a bi--algebraic
co-unit for $\B'$ and its restriction to $\B$ is a bi--algebraic unit for $\B$.
Vice--versa, for $\B'$ to have a co--unit, $\Coop$ has to be co--unital  with co--operadic co--unit $\eps$  with
$\eps_1=\eps_{tot}|_{\Coop}$ where $\Coop\subset \B'$ as $\Coop=\bigoplus_{n\geq 1}\B(n,1)$.

Moreover, the associations $\{\Coop(n)\}\to \B$ and $\{\Coop(n)\}\to \B'$ of a co--unital respectively unital and co--unital graded bi--algebra to a co-unital co--operad are functorial.
\end{prop}

\begin{proof}
On the indecomposables $\O^{nc}$ of $\B$ the fact that $\eps_{tot}$ is a co--unit is just the fact that $\eps_1$ is an operadic co--unit, i.e.\ satisfies \eqref{lcouniteq} and \eqref{rcouniteq}.
For decomposables, we can use induction by using the bi--algebra equation \eqref{bialgeq} and the fact that $\eps_{tot}\circ \mu=\mu_{\unit} \circ ( \eps_{tot}\otimes \eps_{tot})$ where
$\mu_{\unit}:\unit \otimes \unit \to \unit$ is given by the unit constraints. This is also the compatibility of the multiplication and the co--unit. The compatibility of the unit and the co--multiplication says that $1$ is group--like. Finally, by definition $\eps_{tot}(1)=\eps_1^{\otimes 0}(1)=1$.

The fact that this is  a necessary condition is Proposition \ref{freeunitprop}.
For the functoriality: Any map $\P\to \O$ induces a dual map $\Coop\to \check{\P}$, which in turn induces a morphism on the free algebras.
 It is straightforward to check that this map is also a morphism of
bi--algebras preserving grading and the unit. In the case of a co--unital co-- operad case the co-unit is preserved by the morphisms $\eps_{\Coop}=\eps_{\check P}\circ f$, by definition, and hence also is the bi--algebra co--unit.
\end{proof}

\begin{ex}
In the example of leaf labeled rooted forests. A unit for the operad is given by the ``degenerate'' tree or leaf $|$, where gluing on a leaf leaves the tree invariant. This means that $\O(1)=\unit\oplus \bar\O(1)$, where $\unit$ is spanned by $|$ and $\bar\O(1)$ has generators $\t$ which are ``ladders'', viz.\ all vertices are bi--valent.
The dual co--operadic co--unit is the characteristic function $\eps_1=\vardel_|$. Indeed $\check\gamma(\t)$ is the sum over all cuts. Applying $\sum_{k\geq 1}\eps_1\otimes id^{\otimes k}$, we see that only the cut which cuts the root half--edge, that is the term $|\otimes \tau$ evaluates to a non--zero value and $\sum_{k\geq 1}(\eps_1\otimes id^k)\circ\check\gamma(\t)=\tau$. As for the right co--unit property, we see that if $\t\in \O(n)$
then the only terms that survive $id\otimes \eps_{tot}$ are the ones with only factors of $|$ on the right, that is $\t'\otimes | \odo |$. There is precisely one such term with $n$ occurrences of $|$ on the right, which comes from the cut through all the leaves. Indeed, $(\id\otimes\eps_{tot})\circ \check\gamma(\t)=\t$.

When looking at decomposables in $\B$ or $\B'$, these are forests with more than on tree.
The only terms surviving $\eps_{tot}\otimes id$ are the ones with only factors of $|$ on the right, which is just one term corresponding to the cut cutting all root half--edges. Similarly to be non--zero under $id\otimes \eps_{tot}$ the right terms must all be $|$, and again there is only one cut, namely the cut that cuts all leaves of all the trees in the forest. So, indeed we get $(\eps_{tot}\otimes id)\circ\Delta=id=(id\otimes \eps_{tot})\circ \Delta$.
\end{ex}

Summing up the results:

\begin{thm}
\label{bialgthm}
Under the basic assumption,
given a co--operad $\{\Coop(n)\}$  $\B$ is a graded  bi--algebra and $\B'$ is a graded unital bi--algebra. If on only if the underlying co--operad $\{\Coop(n)\}$ is a co--unital co--operad, the bi--algebras $\B$ and $\B'$ also have a co--unit.
 The association of (co--unital) operads to graded (unital), (co--unital) bi-algebras is a covariant functor.
 The association of (unital) operads to graded (unital), (co--unital) bi-algebras is a contravariant functor.
 \end{thm}

\begin{rmk}\mbox{}
This is an example which comes from  the enriched Feynman categories $\Surj_{\O}$, see \cite{feynman} and \cite[\S\ref{P2-feynmanpar}, \S\ref{P2-enrichedpar}]{HopfPart2}.
\end{rmk}

\begin{ex}[Morphisms of operads and co--operads for types of trees]
Let $\O^{ppl}$ be the operad of  planar planted leaf--labeled tress, $\O^{ppl}_3$ the sub--operad of  planar planted trivalent trees and $\P$ the operad of planar corollas.
\begin{enumerate}
\item The inclusion $\O^{ppl}_3\to \O^{ppl}$ gives a morphism $\B_{\O^{ppl}}\to \B_{\O^{ppl}_3}$. This is the map that maps all $\delta_\tau$ for non--trivalent $\tau$ to $0$.
\item There is also an inclusion of $\Coop^{ppl}_3\to\Coop^{ppl}$ which is defined by the inclusion of the generating set. This yields a morphism $\B_{\O^{ppl}_3}\to \B_{\O^{ppl}}$ which is the inclusion as a sub--bi--algebra.
\item There is a morphism of operads $\con:\O^{ppl}\to \P$ given by contracting all internal edges of a tree $\t$. This restricts to $\O^{ppl}_3\to \P$. These morphisms give rise to maps $\B_\P\to \B_{\O^{ppl}}$ and
    $\B_\P\to \B_{\O^{ppl}_3}$. These morphisms send  $\delta_{c}$ to $\sum_{\tau\in\con^{-1}(c)}\delta_\tau$, where the sum is restricted to only the trivalent pre--images for $\B_{\O^{ppl}_3}$. This morphism and its angle decoration are considered in \cite{Gont}, see also \cite[\S\ref{P2-anglepar}]{HopfPart2}. Combinatorially this corresponds to associating to each multiplication $(a_1\cdc a_n)$ all possible bracketings. In this point of view it can be seen as the boundary operation for the associahedra and the co--operad map is the boundary map. The fact that one gets a bi--algebra morphism then states that the operad map on associahedra is a dg--map.
\end{enumerate}
\end{ex}

\subsection{Hopf algebra structure for co--connected co--operads}
\label{freehopfsec}
Under certain conditions a quotient of the bi--algebra $\B'$ is a Hopf algebra. These conditions guarantee connectedness and co--nilpotence of the quotient.
When considering Hopf algebras, we will {\em always} make the following assumption:
\begin{as}
\label{Hopfassump}
The tensor structure and taking kernels commute in both variable. Under this assumption the notions of omnipotent and connected are equivalent.
\end{as}
For example, this is the case if we are working in $\kVect$.

\subsubsection{(Co)--connected (co)--operads}
\label{hopfcondpar}
For a co--unital co--operad, we will say that the co--unit $\eps$ is split if $\Coop(1)=\unit \oplus ker(\eps)= \unit \oplus \Coop(1)^{red}$.
This is automatic if we are in the category of vector spaces, or the $\O(n)$ are free Abelian groups, e.g.\ if they come from an underlying $\Set$ operad.
In the case that $\eps$ is split, let $|$ be the generator of $\unit$ with $\eps(|)=1$.

For an operad, we will say that a unit $u$  is {\em split} if  $\O(1)=\unit\oplus \bar\O(1)$, where $u=1\in\unit$.
If $u$ is split, dualizing $\O(1)=\unit\oplus \bar \O(1)$ to $\Coop(1)=(\O(1))^{\vee}=\unit\oplus \Coop(1)^{red}$ where $|=\delta_u=1\in \unit\subset \Coop(1)$ and $\Coop(1)^{red}=(\bar{\O}(1))^{\vee}=ker(\eps_1)$.
Whence the dual of a split unital operad is a split co--unital co--operad.

\begin{rmk}
 In an operad $\O(1)$ forms an algebra via the restriction of $\gamma$: $\gamma_{1;1}:\O(1)\otimes \O(1)\to \O(1)$. Dually $\Coop(1)$ forms a co--algebra via $\Delta:=\check\gamma_{1;1}:\Coop(1)\to \Coop(1)\otimes \Coop(1)$. $\Delta$ is the restriction of the co--product on $\B$ to $\Coop(1)$.
 If the co--operad has a split co--unit, this co-algebra is pointed by the element $|$.
\end{rmk}
\begin{df}
A co--unital operad  is co--connected, if

 \begin{enumerate}
 \item The co--unit is split.
  \item The element $|$ is group--like: $\Delta(|)=\check\gamma(|)=|\otimes |$
\item  $(\Coop(1),|,\eps)$ is connected as a pointed co-algebra in the sense of Quillen \cite{Quillen} (see Appendix B).
\end{enumerate}

\end{df}
If a unital operad is  reduced  ---$\O(1)=\unit$--- it is  automatically split and $|$ is group--like.
Likewise a co--unital co--operad is reduced if $\Coop(1)=\unit$. It then automatically satisfies (1) and (2).

As the dual of a split unital operad is a split co--unital co--operad, we can illustrate the conditions (2) and (3) in a practical fashion. Consider $\O(1)$ as  an algebra.:

\begin{lem}
The dual split co--unital co--operad of a split  unital operad is satisfies (2)
 if and only if $O(1)$ does not contain any left or right invertible elements except for multiples of the identity. It satisfies (3)
if and only if
\begin{itemize}
\item [(3')] any element
$a\in \O(1)$  the decompositions $a=\prod_{i\in I}a_i$ with all $\eps(a_i)=0$, have bounded length, i.e.\ $|I|$ is bounded.
\end{itemize}
\end{lem}
\begin{proof}
Recall that the co--product in $\Coop(1)$ is dual to multiplication in $\O(1)$, that is, the co--product is decomposition.
Let $u$ be the unit, then $|=\delta_u$.

Now, $\Delta(|)(a\otimes b)\neq 0$ means that $\gamma(a,b)=u$ and hence, $a$ is a left inverse to $b$ and $b$ is a right inverse to $a$. If $|$ is group like, then we need that $a,b\in \unit\subset\unit\oplus \bar\O(1)$, which is the first statement.
Likewise, since, the co--product is decomposition, being co--nilpotent is  equivalent to the given finiteness condition.
\end{proof}

An obstruction to being co--connected are group like elements in $\Coop(1)^{red}$. Such a group like element will be dual to an idempotent.

\begin{cor}
If $\O(1)$ contains any isomorphisms or idempotents except for  multiples of the unit, then $\Coop$ is not co--connected. More precisely, if $\O(1)$ splits as $\unit\oplus \bar\O(1)$, then $\bar \O(1)$ may not contain any invertible elements or any idempotent elements.
\end{cor}
\begin{proof}
Indeed, if $a\in \bar\O(1)$ is  an isomorphism it has factorizations of any unbounded length: $a=a(a^{-1}a)^n$ for any $n$ and $\eps(a)=\eps(a^{-1})=0$. Likewise,
if $p\in\bar\O(1)$ is an idempotent then $p=p^n$, again for any $n$ yielding infinitely many factorizations.
\end{proof}
\begin{ex}
\label{moeex}\mbox{}

\begin{enumerate}
    \item
If the unital operad $\O$ is reduced its dual is also reduced.
 This is the case for the surjection and the simplex operads.
 \item
More generally, if  for a split co--unital co--operad,  $\Coop(1)^{red}$ is free of finite rank
as a co--monoid, then $\Coop$ is co--connected. This is the case for the dual co--operad of an operad whose $\O(1)$ is a free unital algebra of finite rank. An example are
planar planted trees, where $\Coop(1)^{red}$
is free of rank $1$ with the generator being the rooted corolla with one tail.
As previously, the generator corresponding to the dual of the identity can be depicted as the degenerate  ``no vertex'' corolla with one input and output $|$ and the other generator as $\dottree$.

This is  linked to the considerations of \cite{MoerdijkKreimer} in the rank $1$ case and those of higher rank to \cite{moevan},  see also \cite[\S\ref{P2-Moealg}]{HopfPart2}, where the generators can the thought of as $\dottree \, c$, where $c$ is a color index.
\item Assume $\O(1)$ is split unital and (2) holds. If $\O(1)$ as an algebra is an algebra presented by homogenous relations, then $\O$ is co--connected.
This follows, since homogenous relations do not change the length of  a decomposition.
\end{enumerate}

\end{ex}
For a  split co---unital co--operad,
let $\I$ be the two-sided  ideal of $\B'$ spanned by $1-|$.
Set
\begin{equation}
\H:=\B'/\I
\end{equation}
Notice that in $\H$ we have that $|^{k}\equiv 1 \mod \I$ for all $k$.

\begin{prop}
\label{Iprop}
If  $\{\Coop(n)\}$ has a split co--unit and $|$ is group--like, then $\I$ is a co--ideal of $\B'$ and hence $\H$ is a co--algebra. The unit $\eta$ descends to a unit $\bar \eta:\unit\to \H$
and the co--unit $\eps_{tot}$ factors as $\bar\eps$  to make $\H$ into a bi--algebra.
\end{prop}
\begin{proof}
$\Delta(1-|)=1\otimes 1- |\otimes |=(1-|)\otimes | + 1\otimes (1-|)  \in \I\otimes\B +\B\otimes \I$ and $\eps_{tot}(1-|)=1-1=0$.
\end{proof}

\begin{thm}
\label{hopfthm}
If  $\{\O(n)\}$ is co--connected then $\H$ is co--nilpotent and hence admits a unique structure of  Hopf algebra.
\end{thm}

\begin{proof} Let $\pi=id-\bar\eps\circ\bar \eta$ be the projection $\H=\unit \oplus \bar \H\to \bar \H$ to the augmentation ideal. We have to show that each element lies in the kernel of some
$\pi^{\otimes m}\circ \Delta^m$. For $\unit$ this is clear, for the image of $\Coop(1)=\Coop^{nc}(1)$ this follows from the assumptions, from the Lemma above and the identification of $1$ and $|$ in the quotient. Now we proceed by induction on $n$, namely,
for $a\in \B(n)$, we have that $\Delta(a)\in \bigoplus_{k,n} \B(k)\otimes \B(n)$.
Since the co--product is co--associative, we see that all summands with $k<n$ are taken care of by the induction assumption. This leaves the summands with $k=n$. Then the right hand side of the tensor product is the product of elements which are all in $\B(1)=\Coop(1)$. Since $\Delta$ is compatible with the multiplication, we are done by the assumption on $\Coop(1)$ using co--associativity.
\end{proof}

\begin{ex}[Hopf algebra of leaf labeled planar planted trees]
In the example of leaf labeled planar planted trees the Hopf algebra is one of the versions found in \cite{foissyCR1,foissyCR2}. It simply means that all occurrences of $|$ are replaced by $1$ that is just eliminated unless all the factors are $|$; in which case it is replaced by $1$.
We obtain $\Delta_\H(\t)=1\otimes \t + \t \otimes 1 +\sum_{\t_0\subset \t}\t_0\otimes \t \setminus \t_0$ as in \eqref{ckintroeq}, where now all the cut off leaves and cut off root half edges are ignored, or better set to $1$.
\end{ex}

\begin{ex}[The operad of order preserving surjections, aka.\ planar planted corollas]
\label{corrolladeltaex}
In the example of ordered surjections or planar planted corollas, which is isomorphic to order preserving surjections.
Let $c_n$ be the planar planted corolla with $n$ leaves and let $\pi_n:\un\ta \uone$ the unique surjection. The isomorphism is given by $\pi_n\leftrightarrow c_n$.
The non--$\Sigma$ pseudo--operad structure is $c_n\circ_i c_m=c_{n+n-1}$. The operad structure is given by $\gamma(c_k;c_{n_1},\dots, c_{n_k})=c_n$ with $n=\sum_i n_i$.
As surjections this is the composition of maps
\begin{equation}
\gamma(\pi_k;\pi_{n_1}\kdk \pi_{n_k})=\pi_{n_1}\amalg \dots \amalg \pi_{n_k}\circ \pi_k:\un=\underline{n_1}\amalg\dots \amalg \underline{n_k}\ta \underline{k}\ta \uone
\end{equation}
Keeping the notation $c_n$ we get
\begin{equation}
\Delta(c_nn)=\sum_{k,n_1\kdk n_k:\sum n_i=n, n_i\geq1} c_k\otimes c_{n_1}\odo c_{n_k}
\end{equation}
In $\H$:
 \begin{multline}
 \Delta_{\H}(c_n)=1\otimes c_n + c_n \otimes 1 + \\
\sum_{1<k< n} \sum_{(n_1,\dots, n_j)\sum n_i +k-j =n, n_i\geq 2} {k\choose j} c_k\otimes c_{n_1}\odo c_{n_j}
 \end{multline}
 Here the $c_1$ have made into units in $1\in \unit$ and hence the sum is over partitions not containing ones, but there are now multiplicities  according to the number of spots where they were originally inserted. In particular,
 \begin{eqnarray}
&& \Delta_\B(c_3)=1\otimes c_3+ c_3\otimes c_1 \otimes c_1\otimes c_1+ c_2\otimes c_1\otimes c_2+c_2\otimes c_2\otimes c_1\nn\\
&& \Delta_\H(c_3)=1\otimes c_3+c_3\otimes 1 +2 c_2\otimes c_2
 \end{eqnarray}
 \end{ex}

\subsection{The Hopf algebra as a deformation}
\label{defsec}
Rather than taking the approach above, we can produce the Hopf algebra in  two separate steps. Without adding a unit,
we will first mod out by the two-sided ideal $\C$ of $\B$ generated by $|a-a|$. This forces $|$ to lie in the center. We denote  the result  by
$\H_q:=\B/\C$, where the image of $|$ under this quotient is denoted by $q$. This allows us to view $q$ as a deformation parameter and view  $\H$ as the classical limit $q\to 1$ of $\H_q$.

\begin{prop}
\label{Cprop}
If $\{\Coop(n)\}$ is split co--unital and $|$ is group--like, then
$\C$ is a co--ideal and hence $\H_q$  is a co--unital bi--algebra.

\end{prop}
\begin{proof} Using Sweedler notation:
\begin{multline*}
\Delta(|a-a|)=
|a^{(1)}\otimes |a^{(2)}-a^{(1)}|\otimes a^{(2)}|\\
=(|a^{(1)}-a^{(1)}|)\otimes |a^{(2)}+a^{(1)}|\otimes ( |a^{(2)}-a^{(2)}|)
\in \C\otimes \B+\B\otimes \C
\end{multline*}

Furthermore $\eps(|a-a|)=\eps(a)-\eps(a)=0$.
\end{proof}

In the case of a split co--unital co--operad, we  split  $\Coop=\unit\oplus\Coop^{red}$ where $\Coop^{red}=ker(\eps_1)$ that is $\Coop(1)^{red}=ker(\eps)\subset \Coop(1)$ and
$\Coop(n)^{red}=\Coop(n)$.
We set
\begin{equation}
\label{coopncredeq}
\Coop^{nc, red}(n)=\bar T \Coop^{red}(n)=\bigoplus_{k\geq 1}\bigoplus_{(n_1,\dots n_k):\sum n_i=n}\Coop^{red}(n_1)\odo \Coop^{red}(n_k)
\end{equation}
Notice that the image of $|^n$ is $q^n$ and if we give $q$ the degree and length $1$, then the grading by operadic degree passes to the quotient
as well as does the length grading and any combination of them. In particular, $\H_q$
decomposes as $\H_q=\bigoplus_{d}\H_q(d)$ according to the length grading. Furthermore, $q$ is group-like and $\eps(q)=1$.

\begin{prop}
\label{Hqdecompprop}
For a split unital $\O$ with a group--like $|$,
 \begin{equation}
 \label{qeq}
 \H_q(d)\simeq \bigoplus_{n\leq d} q^{d-n}{\Coop}^{red,nc}(n) \text { and } \H_q\simeq \bar T\Coop^{red}[q]
 \end{equation}
\end{prop}
\begin{proof}
In $\H_q$ one can move all the images of $|$, that is factors $q$ to the left. This leaves terms of the form $q^k a$ with $a$ in the image of $\Coop^{red,nc}$. This yields a unique standard form for any element in $\H_q$ and establishes the first isomorphism.
The second isomorphism is a reformulation using \eqref{coopncredeq}.
\end{proof}

\begin{cor}
\label{defcor}
$\H_q$ is a deformation of $\H$ given by $q\to 1$.
\end{cor}
\begin{proof}
Setting $q=1$ corresponds to taking the quotient $\H_q'=T\Coop^{red}[q]=\unit\oplus \H_q$ by the bi--algebra ideal $\I'$ generated by $1-q$.  Including $\C$ into $\B'$: $\B'/\C=\H_q'$. Thus the double quotient satisfies $\H_q'/I'=(\B'/\C)/\I'=\B'/\I=\H$.
\end{proof}

\subsection{The infinitesimal structure}
\label{freeinfsec}
The infinitesimal version corresponds to dualizing the pseudo--operad structure, notably $\circ$.
In order to obtain  the infinitesimal structure, we consider the filtration in terms of factors of $|$.
Using the construction of the double quotient, that is first identifying $|$ with $q$,  gives context to the name infinitesimal.

\begin{as}
In this subsection, for simplicity, we will assume that $\Coop(0)$ is empty or $0$. The arguments also work if  $\Coop$ is locally finite. For the general case, see \S\ref{coactionpar}.
\end{as}

\subsubsection{Pseudo co--operads and the co--pre--Lie Poisson}
\label{pseudocooperadpar}
As before, if we denote by $\eps_1:\Coop\to \unit$ the dual of $u:\unit\to \O(1)$ extended  to all of $\Coop$ by first projecting to $\Coop(1)$ --- in the linear case this is just the extension by $0$.
The dual of the $\circ_i$ expressed as in \eqref{unitcircieq} becomes a morphism $\check\circ_i:\Coop(n)\to \bigoplus_{k=1}^{n-1} \Coop(k) \otimes \unit^{\otimes i-1}\otimes \Coop(n-k) \otimes \unit^{\otimes k-i}$ which for $a\in \Coop(n)$, $1\leq i\leq n$ is defined by
\begin{equation}
\label{freecocirceq}
\check\circ_i(a)=\sum_{k=1}^{n-1}(id \otimes \eps_1^{\otimes i-1} \otimes id\otimes \eps_1^{\otimes k-i})
(\check\gamma_{k;1,\dots 1,n-k,1 \cdc 1}(a))
\end{equation}
with $id$ in the 1st and $i+1$-st
 place.
Using the unit constraints for the monoidal category to eliminate tensor factors of $\unit$ resulting from the maps $\eps_1$, we get maps
\begin{equation}
\delta_i:\Coop(n)\to \bigoplus_{k=1}^{n-1} \Coop(k)\otimes \Coop(n-k+1) \text{ for $1\leq i\leq n$}
\end{equation}
These maps and their compatibilities constitute a (non--$\Sigma$ pseudo--co--operad structure.
One can add symmetric group actions as well and postulate equivariance, as before, to obtain the notion of a pseudo co--operad.

One can reconstruct the $\check\gamma$  from the $\check \circ_i$ and the $\delta_i$. This fact is used with great skill in \cite{brown,brownICM,browndecomp}.
Summing over all the $\delta_i$ we get a map
\begin{equation}
\label{deltaeq}
\delta_c:\Coop(n)\to \bigoplus_{k=1}^{n} \Coop(k)\otimes \Coop(n-k+1)
\end{equation}

The subscript $c$ stands for connected.
As $\check\circ$ is dual to the pre--Lie product $\delta_c$ it is co--pre--Lie.

 We will extend $\delta_c$ to $\B$ by
\begin{equation}
\label{freeinfalgeq}
\delta\circ\mu\;\;=\;\;(\mu\otimes id)\circ \pi_{23}\circ(\delta \otimes \id)+(\mu\otimes \id)\circ(\id\otimes \delta)
\end{equation}
where $\pi_{2,3}$ switches the second and third factor, i.e.\
$\delta(ab)=\sum a^{(1)}b\otimes a^{(2)}+\sum ab^{(1)} \otimes b^{(2)}$,
which can be thought of as a Poisson condition for the co--pre--Lie--product $\delta$ that is obtained by dualizing the Poisson bracket to $\delta$, but keeping the multiplication in the usual Poisson equation.

\begin{ex}
In the example of leaf labeled planar planted forests, $\delta$ corresponds to cutting a single edge of the forest.
In simplicial terms, $\delta$ defines the $\cup_1$ product; see  \cite[\S\ref{P2-cupsec}]{HopfPart2} for more details.

For the operad of planted planar corollas, we have $\delta(c_1)=c_1\otimes c_1$ and  for $n\geq 2$: $\delta(c_n)=\sum_{k=2}^{n} k \; c_{k}\otimes c_{n-k+1}$.
\end{ex}

\subsubsection{Filtration and structure of infinitesimal co--algebra}
 Then there is an exhaustive  filtration of $\B$ by the powers of $\J$. That is
$\bar T\Coop^{red}=\J^{\leq 0} \subset \J^{\leq 1} \subset \dots \subset \J^{\leq k} \subset \dots \subset \B=\bar T\Coop$. An element $a\in \J^{\leq k}$ if its summands contains less or equal to $k$ occurrences of $|$.
This filtration survives the quotient by $\C$ and gives a filtration in powers of $q$ that can be viewed as a deformation over a formal disc, with the central fiber $z=0$, corresponding to
$q=\exp{2\pi i z}$ and $\exp(2\pi i 0)=1$, so that $q\to 1$ corresponds to $z\to 0$.

\begin{lem}
\label{barlem}
If $\O$ is split unital and $|$ is group--like then
$\Delta(|^n)=|^n\otimes |^n$, and for $a\in \Coop^{nc}(n,p)$ with $\eps_{tot}(a)=0$:
\begin{eqnarray}
\label{barcounteq}
\Delta(a)&=&|^p\otimes a+a\otimes |^n + \bar\Delta(a) \text{ with } \bar\Delta=\sum_{k=p}^n \bar\Delta_k \text{ and }\nn\\
\bar\Delta_k(a)&=&\sum_{i=0}^{k-1} a^{(i,1)}_k\otimes |^{i} a^{(i,2)}_{n-k+1} |^{k-i-1}
+  R  \text{ with  }R \in \J^{\leq p-1} \otimes \J^{\leq k -2}\nn\\
&\text{where}& a_k^{(i,1)}\in \J^{\leq p-1}, a^{(i,2)}_{n-k+1}\in  \Coop^{red}(n-k+1)
\end{eqnarray}
\end{lem}

\begin{proof}
The first statement follows from the bi--algebra equation \eqref{bialgeq} and the fact that $|$ is group like. The second statement follows from the fact that $\eps_{tot}$ is a left and right co--unit.
If we compute $(\eps_{tot}\otimes id)\circ \Delta(a)$ only terms of $\Delta(a)$ of the form $|^k\otimes f_k$  with $\length(f_k)=k$ survive. If $a$ has length $p$,
then  $(\eps_{tot}\otimes id)\circ \Delta(a)=a$, i.e.\ $(\eps_{tot}\circ id)(\sum_k |^k\otimes f_k)=\sum_kf_k=a$, forces all terms of length not equal to $p$ to vanish, $f_k=0, k\neq p$  and  $f_p=a$. Similarly computing $(id\otimes \eps_{tot})\circ \Delta(a)$ only terms of the form $f_k\otimes |^k$ survive where now $\deg(f_k)=k$, and we find that $f_k=0,k\neq n$ and $f_n=a$.

 In general one can count the factors of $|$ that may occur on the right in $\Delta$. Such factors  can only come from factors of $\Coop(1)$.
Fixing the length on the right side of $\Delta$ to be $k$, i.e.\ considering $\Delta_{k,n}$, the maximal number of factors of $|$ appearing in $\Delta_{k,n}(a)$ is $k$, that is $\Delta_{n,k}(a)\subset \B\otimes \J^{\leq k}$. By the above, the maximal number $k$ is only obtained if $k=\deg(a)=n$ the next leading order in $|$ is then $k-1$, which means that only one factor on the right is not $|$.  By the degree grading, this has to lie in $\Coop(n-k-1)$ and if $k=n$, so that $n-k+1=1$ then the factor cannot be $|$, since the maximal number is $|^{k-1}$.
This is the first term in $\bar\Delta$. By the above, we have already seen that the only term with the maximal number of $|$ on the left, i.e.\  $|^p$,
is $|^p\otimes a$, so that all remaining terms are indeed in $\J^{\leq p-1} \otimes \J^{\leq k -2}$.

\end{proof}

\begin{cor}  For elements $a\in\Coop(n,p)$:
\label{barcor}
\begin{equation}
\label{littledeltaeq}
\delta(a)=\sum_{k=p}^{n} \sum_{i=1}^{k-1} a^{(i,1)}_k\otimes a^{(i,2)}_{n-k+1}
\end{equation}
where the  $\sum a^{(i,1)}_k\otimes a^{(i,2)}_{n-k+1}$ are given in \eqref{barcounteq} and can be extracted via
\begin{equation}
\label{explicitdeltaieq}
\sum a^{(i,1)}_k\otimes a^{(i,2)}_{n-k+1}= {\bf u}\circ(id \otimes \eps_1^{\otimes i-1} \otimes id\otimes \eps_1^{\otimes k-i})
(\bar\Delta_k(a))
\end{equation}
where ${\bf u}=id\otimes u_L^{\otimes i-1}\otimes id \otimes u_R^{\otimes k-i-1} $ are the unit constraints to get rid of the factors $\unit$ stemming from the images of $\eps_1$.
\end{cor}
\begin{proof}
We proceed by induction. For elements of length $1$ the claim follows from \eqref{deltaeq} via \eqref{freecocirceq}. For length $p$ we use \eqref{freeinfalgeq} and induction to plug in. The formula \eqref{explicitdeltaieq} follows directly from \eqref{deltaeq}, since it is non--vanishing only on $im(\bar\Delta_k)\cap \J^{k-1}$.
\end{proof}

When we descend to $\H_q$ the filtration in $|$ corresponds to the powers of $q$: $im(\J^{\leq k})=\{p:p\in \H_q=\bar T\Coop^{red}[q], \deg_q(p)\leq k\}$
\begin{cor}
\label{rescor}
The co--pre--Lie Poisson $\delta$ descends to  $\H_q$, and its terms are given by
terms of  $\bar\Delta_k$ whose degree on the right is $q^{k-1}$. Alternatively, setting $\bar\Delta_{q,k}=(id \otimes q^{-k})
\circ \bar\Delta_k$, to compensate the grading on the right side, $\delta=res_{-1}(\sum_k \bar \Delta_{q,k})$ where $res_{-1}$ is the residue in $q$.
\end{cor}

\begin{proof}
This follows directly from  \eqref{barcounteq} and Corollary \ref{barcor}.
\end{proof}

\subsection{Derivations and irreducibles}

In the case of multiple zeta values, applying the Hopf quotient to the infinitesimal structure and restricting to indecomposables, yields Brown's operators $D_n$  \cite{brownann} that determine his co--action, see also \S \ref{overlapex} and \S\ref{coactionpar}. This holds more generally:

\begin{thm}
\label{indecthm}
Taking the limit $q\to 1$ or equivalently regarding the  quotient $\H$, the co-pre--Lie structure induces a co-Lie algebra structure on the indecomposables $\bar \H/\bar\H\bar\H$, where $\bar\H$ is reduced version of $\H$.
\end{thm}
\begin{proof}
By
The indecomposables in reduced Hopf algebra are precisely given by $\Coop^{red}$ and the co--pre--Lie structure is $\delta_c$ of \eqref{deltaeq}.
Its co--commutator yields to co--Lie structure corresponding dually the  usual Lie structure of Gerstenhaber \cite{Gerst}.
\end{proof}

\subsubsection{The example of overlapping sequences and Goncharov's Hopf algebra}
\label{overlappar}
Considering Example \ref{overlapex}, the bi--algebra structure $\B$ is simply given by re--interpreting the r.h.s.\ of \eqref{overlapeq} as a free product.
There is a co--unit for the co--operad which is given by $\eps(a_0;a_1)=1$ and $0$ else. Indeed under $\eps\otimes id$ the only non--zero term of the r.h.s.\ of \eqref{overlapeq} is $(a_0;a_n)\otimes (a_0;a_1\kdk a_{n-1}; a_n)$. Applying $id \otimes \eps^{\otimes k}$ a non--zero term occurs only when $k=n$ as $(a_0; a_1 \kdk a_{n-1}; a_n)\otimes (a_0;a_1)\otimes (a_1;a_2)\odo (a_n;a_{n+1})$.
In the Hopf quotient $(a_0;a_n)$ and the $(a_i;a_{i+1})$ will be set to $1$.

\begin{prop}
Consider the projection $\pi:\B'=T\Coop\to S\Coop$
Let $\pi(a_0;a_1\dots a_{n-1}; a_n)=\hat I(a_0; a_1\kdk a_{n-1};a_n)$.
In the case of $S=\{0,1\}$
Then after taking the Hopf quotient, that is considering the induced map $\pi_\H:\H\to \H_{com}$, where $\H_{com}=\H_{com}/[\H_{com},\H_{com}]$ is the Abelianization,
we obtain Goncharov's first Hopf algebra $\H_{com}=\H_G$.
The operadic degree of $(a_0;\dots ; a_n)$ is $n$ and hence the weight is $n-1$ $\wt(a_0;a_1\kdk a_{n-1}; a_n)= n-1$ is the weight of $\hI(a_0;a_1\kdk a_{n-1};a_n)$ as defined previously.
\end{prop}

\begin{proof}
The relations on the symbols $\hat I$ are precisely that they commute and \eqref{emptycondeq}. The latter condition already holds in $\H$, namely $(a_0,a_1)=1$ and hence $\hat I(a_0,a_1)=1$.
\end{proof}

\begin{rmk} \mbox{}
\begin{enumerate}
\item In the case $S\subset {\mathbb C}$, we get the Hopf algebra for the polylogs, \cite{Gont}.
\item The role of the depth as the number of $1$s  is not as clear in this formulation. See
Remark \cite[\ref{P2-depthrmk}]{HopfPart2}
for a possible explanation using Joyal duality.  That remark also links it to the depth filtration used in \S\ref{gencooppar}.

\end{enumerate}

\end{rmk}
\begin{prop}
The action of $\delta$ is given by
\begin{multline}\delta (\hI(a_0;a_1\kdk a_{n-1};a_{N+1})=
\sum_{n=3}^{N-3} \sum_{p=0}^{N-n}\\
\hI(a_0;a_1\kdk a_p,a_{p+n-1}\kdk a_N;a_{N+1})\otimes \hI(a_p;a_{p+1}\kdk a_{p+n};a_{p+n-1})
\end{multline}
\end{prop}
\begin{proof}
This is a straightforward calculations. Note that the sub--sequences of length $2$ are in $\Coop(1)$ and are set to $q$  respectively $1$ in the different quotients, so that the reduced structure starts with words of length 3 on the left and stops with words of length 3 on the right.
\end{proof}
\begin{cor}
Incorporating the projection  and the co--action this  recovers the co--derivations $D_n$ of Brown \cite[(3.4)]{brownann} as the degree $n$ part. Moreover, this is exactly in $q$ degree $N-n$.
\end{cor}

\subsection{Bi-- and Hopf algebras from symmetric (co)--operads}
\label{freesymsec}
Before treating co--operads, we consider operads as a warm--up.

\subsubsection{Operads based on sets and invariants}
To dualize symmetric operads, it will be important to go the generality of indexing by in arbitrary sets, see e.g.\ \cite{MSS}. This means that for any finite set $S$ we have an $\O(S)$ and any isomorphism $\sigma:S\to S'$ an
isomorphism $\O(S)\to \O(S')$. The composition is then  defined for any map $f:S\to T$ as a morphism $\O(T)\otimes \bigotimes_{t\in T} \O(f^{-1}(t))\to\O(S)$ which is equivariant for any diagram of the form
\begin{equation}
\label{square}
\xymatrix{
S\ar[r]^f\ar[d]_{\sigma'}&T\ar[d]^\sigma\\
S'\ar[r]^{f'}&T'\\
}
\end{equation}
Thus, $f'=\sigma \circ f\circ \sigma^{\prime -1}$ and the partition $S=\amalg_{t \in T} f^{-1}(t)$ maps to the partition $S'=\amalg_{t'\in T'}f^{\prime -1}(t')=
\amalg_{\sigma(t): t\in T} \sigma'(f^{-1}(t))$. That is $\sigma'$ maps the fibers to the fibers and $\sigma$ permutes the fibers.

Recall that if we are only given the $\O(n)$ then the extension to finite sets is given by $\O(S):=\colim_{f:S\leftrightarrow \underline{n}}\O(n)$. Where $\underline{n}=\{1,\dots, n\}$ and the co--limit is over bijections. Concretely in an Abelian category this is
$\O(S)=(\bigoplus_{\phi:S\stackrel{1-1}{\leftrightarrow}\underline {n}}\O(n))_{\SS_n}$ where $\SS_n$ acts by post--composing.
 This actually yields an equivalence of categories between  operads over finite sets and operads.
 Notice that $\underline{0}=\emptyset$ and  we can restrict the
considerations to non--empty sets and surjections with the skeleton consisting of $\underline{n},n\geq 1$ and surjections.

\begin{rmk} In the following, we will use the language of  limits and co--limits, which are categorical notions, see e.g.\ \cite{MacLane}.
The important examples here are invariants and coinvariants. In particular, if $G$ acts on $X$, then $\lim_G(X)=X^G$ are  and $\colim_G(X)=X_G$. There is one more feature that is exploited in the following.
The invariants are a sub--object $X^G\to X$ and the coinvariants are a quotient object $X\to X_G$ and they have certain universal properties.
For the notation, see the Introduction.

\end{rmk}

\begin{lem}
Using the equivariance of the maps $\gamma_f$ one  obtains an induced map on invariants $\bar\gamma_{k;n_1,\dots,n_k}:\O(k)^{\SS_k}\otimes Symm(\bigotimes_{i=1}^k\O(n_i)^{\SS(n_i)})\to \O(n)^{\SS_n}$, where $Symm(\bigotimes_{i=1}^k\O(n_i)^{\SS(n_i)})$ are the symmetric tensors.
\end{lem}

\begin{proof}

The following proof is technical and involves limits. In the concrete example of vector spaces, the morphisms $i$ and $\imath$ are natural  inclusions.

Let $|S|=|S'|=k$, $n_t=|f^{-1}(t)|=|f'^{-1}(\sigma(t))|$ and $\sigma'_t:f^{-1}(t)\to (f')^{-1}(\sigma(t))$  be the restriction.
For all pairs of isomorphisms $(\sigma',\sigma)$ given above,  we have a natural diagram

\begin{equation}
\label{symopcompeq}
\resizebox{\textwidth}{!} {
\xymatrix{
\O(T)\otimes \bigotimes_{t\in T} \O(f^{-1}(t))\ar@(dl,ul)@<-5pc>[dd]_{\sigma
\otimes\bigotimes_T\sigma'_t}\ar[rr]^{\gamma_f}&&\O(S)\ar@(dr,ur)@<1pc>[dd]^{\sigma'}\\
(\O(k)\otimes \bigotimes_{i=1}^k\O(n_i)^{\SS(n_i)})^{\SS_k}\ar[u]_i
\ar[d]^i&
\O(k)^{\SS_k}\otimes Symm(\bigotimes_{i=1}^k\O(n_i)^{\SS(n_i)})
 \ar[l]_{\imath}\ar[r]^-{\bar \gamma_{k;n_1,\dots,n_k}}\ar[ur]^{\imath\circ i\circ\gamma_f}\ar[dr]_{\imath\circ i\circ\gamma_{f'}} &\O(n)^{\SS_n}
\ar[d]\ar[u]\\
\O(T')\otimes \bigotimes_{t'\in T'} \O(f^{\prime \, -1}(t'))\ar[rr]_{\gamma_{f'}}&&\O(S')\\
}}
\end{equation}
in which the outer square of the diagram below commutes. Moreover, this diagram exists for all $\sigma'$ and any fixed $f:T\to S$ whose $k$ fibers have the right cardinalities $n_i$ One choice is given by $\sigma'=id$ and $f'=\sigma\circ f$ .
The morphisms are defined as follows: let $Iso(n,k)$ be the category with objects
the surjections $S\to T$ with $|S|=n$ and $|T|=k$ and morphisms the commutative
diagrams of the type \eqref{square} with $\sigma,\sigma'$ bijections and $f,f'$ surjections, and $Iso(n)$ the category with objects $S$, with $|S|=n$ and bijections.
Then
\begin{enumerate}
\item $\lim_{Iso(n)}\O=\O(n)^{\SS_n}$ are the invariants.
\item $\lim_{Iso(n.k)}\O= \bigoplus_{(n_1,\dots,n_k):\sum n_i=n} \left(\O(k)\otimes \bigotimes_{i=1}^k\O(n_i)^{\SS(n_i)}\right)^{\SS_k}$ where on $\O(f:S\to T)$ = $ \O(T) \otimes_T\bigotimes_{T} \O(f^{-1}(t))$ with
$Aut(T)\simeq \SS_k$ acting anti--diagonally as $\sigma\otimes \sigma^{-1}$. These invariants include into the $\O(T)\otimes \bigotimes_T \O(f^{-1}(t))$ by virtue of being a limit.
\item The invariants under the full $\SS_k\times \SS_k$ action on $\O(k)\otimes \bigotimes_{i=1}^k\O(n_i)^{\SS(n_i)}$ are
are $\O(k)^{\SS_k}\otimes Symm(\bigotimes_{i=1}^k\O(n_i)^{\SS(n_i)})$, where $Symm$ is the subspace of symmetric tensors,
These include into the invariants of only the $\SS_k$ action using the anti--diagonal embedding $\SS_k\subset \SS_k\times \SS_k$, since technically this is again a limit.

\item The map $\bar \gamma$ exists by the universal property of limits applied to $\O(n)^{\SS_n}$ and the cone given by the  $\gamma_f\circ i\circ \imath$.
\end{enumerate}
\end{proof}
\begin{rmk}
\label{coinvrmk}
These are exactly universal operations in the sense of \cite[\S6]{feynman} for the Feynman category $\FF_{May}$, see \cite[\S6]{feynman} and \cite[\S\ref{P2-universalpar}]{HopfPart2}.
\end{rmk}

\subsubsection{Symmetric co--operads}
The dual notion  is a symmetric co--ope\-rad indexed by finite sets. This is  a collection of objects $\{\Coop(S)\}$ for all finite sets $S$, isomorphisms $\sigma^*:\Coop(S')\to \Coop(S)$ for bijections $\sigma:S\to S'$ for any surjection $f:S\ta T$
\begin{equation}
\check\gamma_f :\Coop(S)\to  \otimes \bigotimes_{t\in T}\Coop(f^{-1}(t))
\end{equation}
that are equivariant with respect to the isomorphisms.

\begin{lem}
The degree--wise dual of a symmetric operad is a symmetric co--operad and the association is functorial. \qed
\end{lem}

\begin{prop}
The maps $\check\gamma_f$ descend to the coinvariants as a map
$\bar\gamma^{\vee}_{k;n_1,\dots,n_k}:\Coop(k)_{\SS_k}\otimes \bigodot_{i=1}^k\O(n_k)_{\SS(k)}\to \Coop(n)_{\SS_n}$, where $\bigodot$ denotes the symmetric (aka.\ commutative) tensor product.
\end{prop}

\begin{proof}
The dual diagrams, to \eqref{symopcompeq} are
\begin{equation}
\label{coprodsqeq}
\resizebox{\textwidth}{!} {
\xymatrix{
\Coop(S)\ar[rr]^{\check\gamma_f}\ar[d]\ar@(dl,ul)@<-1pc>[dd]_{\sigma'}\ar[dr]^{\gamma_f\circ p\circ\pi}&&\Coop(T)
\otimes \bigotimes_{t\in T}\Coop(f^{-1}(t))\ar[d]^p
\ar@(dr,ur)@<5pc>[dd]^{\sigma\otimes\bigotimes_T \sigma'_t}\\
\Coop(n)_{\SS_n} \ar[r]^-{\bar \gamma^{\vee}_{k;n_1,\dots,n_k}} &\Coop(k)_{\SS_k}\otimes \bigodot_{i=1}^k\O(n_k)_{\SS(k)}&\check\O(k)\otimes_{\SS_k}\bigotimes _{i=1}^k\check\O(n_k)\ar[l]_--{\pi}\\
\check\O(S')\ar[rr]^{\check\gamma_{f'}}\ar[u]\ar[ur]_{\gamma_f\circ p\circ\pi}&&\Coop(S')\otimes \bigotimes_{t'\in T'}\Coop(f')^{-1}(t')\ar[u]_p\\
}
}
\end{equation}
where
\begin{enumerate}
\item $\colim_{Iso(n)}\check\O=\O(n)_{\SS_n}$ are the coinvariants.
 \item $\colim_{Iso(n.k)}\check\O=(\check\O(k)\otimes_{\SS_k}\bigotimes_{i=1}^k\check\O(n_k)$, where $\SS(k)$ acting anti--diago\-nal\-ly yields the relative tensor product.
 \item These  project to the full coinvariants under the $\SS(k)\times \SS(k)$ action:
 $\check\O(k)_{\SS_k} \otimes \bigodot_{i=1}^k\check\O(n_k)_{\SS(k)}$
\item The map $\bar\gamma^{\vee}_{k;n_1,\dots,n_k}$ exists by the universal property of co--limits applied to $\check\O(n)_{\SS_k}$ and the co--cone given by the $\pi\circ p\circ\gamma_f$. Again fixing an $f$ with the correct cardinality of the fibers, we have morphisms $f'$ defined by any given $\sigma'$, and some choice of $\sigma$, e.g.\ $\sigma=id$.

\end{enumerate}
\end{proof}

\begin{rmk}
\label{rigidrmk}
In order to obtain a system of representatives one has to ``enumerate everything'' in order to rigidify. This means that for $S$ and $T$ one fixes an isomorphism to $\{1\kdk |S|=n\}$ and $\{1\kdk |T|=k\}$ {\it and} considers them as ordered sets.
Then the unique order preserving map $f$ with fiber cardinalities $n_1\kdk n_k$ is a representative for $\bar\gamma^{\vee}_{k;n_1\kdk n_k}$
\end{rmk}
\begin{rmk}
\label{equirmk}
Using the co--cone $\gamma_f\circ p$
there is the intermediate possibility to have the co--multiplication $\check\gamma_{k,n_1,\dots,n_k}$ as a morphism $\check\gamma : \check\O(n)_{\SS(n)}\to\check \O(k)\otimes_{\SS_k}\bigotimes_{i=1}^k \check\O(n_k)_{\SS(k)}$. This is an interesting structure that has  appeared for instance in \cite{Kurush}.
This corresponds to Figure \ref{fig:coproducts2b}.
\end{rmk}
\subsubsection{Bi--algebras and Hopf algebras in the symmetric case}
We now proceed similarly to the non--$\Sigma$ case, but using symmetric algebras on the coinvariants instead.

For a  locally finite symmetric co--operad $\{\Coop\}$, set  $\Coop_{\SS}=\bigoplus_k\Coop(k)_{\SS_k}$  and define $\B=\bar S\Coop_\SS$, the reduced symmetric algebra on $\Coop_\SS$, aka.\ the free commutative algebra,  and $\B'=S\Coop_\SS$, the free unital symmetric algebra, aka.\ the free unital commutative algebra.
By summing over the $k$ and $n_i$, we obtain the morphism $\bar\gamma^{\vee}:\Coop_\SS\to \Coop_\SS\otimes \B$. We extend to a  co--product $\Delta$ using the bi--algebra equation \eqref{bialgeq} to obtain a co--multiplication on $\B$ and define $\Delta(1)=1\otimes 1$ to extend to $\B'$.
\begin{equation}
\Delta:\B\to \B\otimes \B \text{ and } \Delta:\B'\to \B'\otimes \B'
\end{equation}
$\B$ and $\B"$ still have the double grading by degree and length, that is
$\B=\bigoplus_{n\geq k\geq  1}\B(n,k)$ and $\B'=\bigoplus_{n\geq k\geq  0}\B'(n,k)=\unit \oplus \B$.
Given a co--unit for the co--operad\footnote{There is no additional assumption as the symmetry group that acts on $\Coop(1)$ is trivial}, we see that $\eps^{\otimes k}$ is symmetric and hence gives a morphisms $\eps_k:\Coop(k)_{\SS_k}\to \unit$. Set $\eps_{tot}=\sum_k \eps_k:\Coop_\SS\to \unit$. We set $\Coop^{red}(n)=\Coop(n)$ for $n\geq 2$ and $\Coop^{red}_\SS=\bigoplus_{k\geq 1}\Coop^{red}_{\SS_k}(k)$.
\begin{thm}
\label{symbialgthm}
$\Delta$ is co--associative,
 $\B$ is a commutative bi--algebra and $\B'$ is a commutative unital bi--algebra.
The bi--algebra $\B$ and the unital bi--algebra $\B'$ have a co--unit if and only if the co--operad $\{\Coop(n)\}$ has a split co--unit.
In this case, $\B'\simeq\bar S\Coop^{red}_\SS[q]$.
These associations are functorial.

\end{thm}

\begin{proof} The fact that the co--product $\Delta$ is co--associative follows from the equivariance and co--associativity of the co--operad $\{\Coop(n)\}$ in a straightforward fashion.
The statements about $\B$ and $\B'$ then follow from their definition.
The fact that $\eps_{tot}$ is a co--unit is verified as in the non--$\Sigma$ case and the other direction follows from Proposition \ref{freeunitprop}.
If the co--operad has a split co--unit, we denote the image of $|$ in $\Coop_\SS$ by $q$. Since $\B'$ is already commutative, $q$ commutes with everything, and we can collect the powers of $q$ in each monomial leading to the identification with the polynomial ring over $\bar S\Coop^{red}$.

Finally, it is clear that a morphism of symmetric operads $\O\to \P$ induces a morphism $\check\P\to \Coop$ and due to the compatibility of the $\SS_n$ actions a morphism from $\check\P_\SS\to \Coop_\SS$ which is compatible with the $\bar\gamma^{\vee}$. The free functors are also functorial showing the functoriality of $\B$ and $\B'$.
\end{proof}

{\em A symmetric operad $\O$ is co--connected}, if it is split unital, $|$ is group--like for $\Delta$ and $(\Coop(1),|,\Delta|_{\Coop(1)})$ is connected.
Let $\I'$ be the two sided ideal of $\B'$ generated by $1-q$. It is straightforward to check that this is also a co--ideal using the proof of Proposition \ref{Iprop}.

\begin{thm}
\label{symhopfthm}
 If $\O$ is co--connected, $\H=\B'/\I'$ is a Hopf algebra. Under the identification $\B=\bar S\Coop^{red}_\SS[q]$, the limit $q\to 1$ yields $\H$.
\end{thm}
\begin{proof}
Analogous to Theorem \ref{hopfthm} and Corollary \ref{defcor}.
\end{proof}
\begin{ex}
\label{symmetryex}

In the examples of leaf--labeled rooted trees and surjections, extra multiplicities appear from considering the symmetric product, which identifies certain contributions. In particular, let $[c_n]$ be the class
of an un-labeled non--planar corolla then
\begin{equation}
\Delta([c_n])=\sum_{\mbox{\tiny
\begin{tabular}{l}
$k,1\leq n_1 <\dots < n_l$\\$1\leq m_1,\dots \kdk m_l:$\\
$\sum m_j=k,\sum m_ln_i=n $
\end{tabular}}
}
{k\choose m_1 \cdots m_k} [c_k]\otimes [c_{n_1}][c_{n_2}]\cdots [c_{n_k}]
\end{equation}

\end{ex}
\begin{rmk}
Since $\H$ is commutative in the symmetric case, its dual is co-commutative and $\H^*=U(Prim(\H^*))$ by the Cartier--Milnor--Moore theorem.
This relates to the considerations of \cite{del,CL}.
We leave the complete analysis for further study.
\end{rmk}
\begin{ex}
Reconsidering the examples in this new fashion, we see that:
\begin{enumerate}
\item For the ordered surjections, in the symmetric version, we get all the surjections, since the permutation action induces any order. These are pictorially represented by forests of non--planar corollas.
Taking coinvariants makes these forests unlabeled and the forests commutative.

\item This carries on to the graded case like in Baues.
\item For the trees, we go from planar planted trees to leaf--labeled rooted trees in the symmetric version. The trees/forrests are again un--labeled on the coinvariants.
\end{enumerate}
\end{ex}

\subsubsection{Infinitesimal structure}
For the infinitesimal structure, we notice that although the individual $\circ_i$ are not well defined on the invariants $\Coop(n)_{\SS_n}$ their sum $\circ$ is. This was first remarked in \cite{KapMan}.
Again, these are universal operations for the Feynman category for pseudo--operads $\operads$, see \cite[\S6]{feynman}.
Thus, the dual $\check\circ=\delta_c$ is well defined and we can again use \eqref{freeinfalgeq} to extend to  maps
\begin{equation}
\delta:\B\to \B\otimes \B \text{ and } \delta:\B'\to \B'\otimes \B'
\end{equation}
this map is again pre--co--Lie and Poison. The analogue of Lemma \ref{barlem} and Corollary \ref{barcor} hold, as does:
\begin{thm}
\label{symcopreliethm}
The co-pre--Lie structure induces a co-Lie algebra structure on the indecomposables $\bar \H/\bar\H\bar\H$, where $\bar\H$ is reduced version of $\H$.

If $\O$ is split unital, the terms of the co--pre--Lie Poisson $\delta$  on $\B=\bar S\Coop^{red}_\SS[q]$ are given by
terms of  $\bar\Delta_k$ whose $q$--degree on the right is $q^{k-1}$. Alternatively, setting $\bar\Delta_{q,k}=(id \otimes q^{-k})
\circ \bar\Delta_k$, to compensate the grading on the right side, $\delta=res_{-1}(\sum_k \bar \Delta_{q,k})$ where $res_{-1}$ is the residue in $q$.
\end{thm}
\begin{proof}
Analogous to Theorem \ref{indecthm} and Corollary \ref{rescor}.
\end{proof}

\begin{rmk}
There are actually no new symmetry factors in the examples appearing for $\delta$ since it only involves products with one term that is not $|$.
\end{rmk}
\subsection{Connes--Kreimer co--limit quotient}
\label{freeampsec}
To obtain the original Hopf algebra of Connes and Kreimer and its planar analogues, we have to take one more step, which in the general case is only possible if an
additional structure  is present.

\begin{df}
A clipping or amputation structure,  for a non-$\Sigma$ co-operad $\Coop$ is a
co--semisimplical structure that is compatible with the operad compositions, i.e.\

\begin{enumerate}
\item   There are maps $\sigma_i:\Coop(n)\to \Coop(n-1)$, and for $i\leq j: \sigma_j\circ \sigma_i=\sigma_i\circ\sigma_{j+1}$. In case that there is no $\Coop(0)$, these maps are defined for $n\geq 2$.
\item  For all $n$, and partition $(n_1,\dots, n_k)$ of $n$ and each  $1\leq i\leq  n$,
with $i$ in the $n_j$ component of the partition and $i_j$ its  position within this block, i.e.\ $i=\sum_{k=1}^{j-1}n_j +i_j$:
\begin{equation}
\label{cutgammaieq}
\check\gamma_{k;n_1,\dots,n_j-1,\dots, n_k}\circ \sigma_i=
(id\otimes id \odo \sigma_{i_j}\otimes id  \odo id)\circ
\check\gamma_{k;n_1,\dots,n_j,\dots, n_k}
 \end{equation}
 with the factor $\sigma_{i_j}$ in the $j+1$--st position acting on $\Coop(n_j)$.
 \end{enumerate}
In case that there is no $\Coop(0)$, we demand that $n\geq 2$ and also that $n_j\geq 2$.

 For  symmetric co--operads,
in the set indexed  version, the data and compatibilities for and amputation structure are given by a functor $\Coop$ from the category of finite set and surjections together with a compatible co--operad structure on the $\Coop(S)$. This boils down to:
 \begin{itemize}

\item[(1')]  Amputation morphisms $\sigma_s:\Coop(S)\to \Coop(S\setminus s)$ for all $s\in S$. Which commute $\sigma_s\circ\sigma_{s'}=\sigma_{s'}\circ\sigma_{s}$, and are compatible with isomorphisms induced by bijections. Namely, for $\phi:S\to T$ a bijection, let $\phi|_{S\setminus \{s\}}:S\setminus \{s\}\to T\setminus \{\phi(s)\}$ be the  bijection  induced by restriction, then:
\begin{equation}
\sigma_{\phi(s)}\circ\Coop(\phi)=\Coop(\phi|_{S\setminus \{s\}})\circ \sigma_s
 \end{equation}
Where $\Coop(\phi):\Coop(S)\to \Coop(T)$ are the structural isomorphisms, see \S\ref{freesymsec}.
Note that if $S=\{s\}$ then  $\sigma_s:\O(\{s\})\to \O(\emptyset)$ and this is excluded, if we make the assumption  that there is no $\Coop(0)=\Coop(\emptyset)$.
\item[(2')] Given a surjection $f:S\ta T$ for any $s \in S$, let $f|_{S\setminus \{s\}}$ be the restriction. Let $t=f(s)$ then the map
$f|_{S\setminus \{s\}: S\setminus \{s\}}\ta T$ is surjective as long as $f^{-1}(t)\neq \{s\}$. In this case, we set  $f_{S\setminus s}=f|_{S\setminus s}$.
This is the only case, if we exclude $\Coop(\emptyset)$. In case $\Coop(\emptyset)$ is allowed, if $f^{-1}(t)= \{s\}$ then we define $f_{S\setminus s}=f|_{S\setminus s}:S\setminus\{s\}\to T\setminus \{t\}$.
With this notation the compatibility  equation is
\begin{equation}
\label{cutequseq}
\check\gamma_{f_{S\setminus \{s\}}}\circ \sigma_s=\sigma_s\circ \check\gamma_{f}
\end{equation}
\end{itemize}
\end{df}

\begin{ex} If $\O(n)$ is free, e.g.\ if it comes from a co--simplicial $\Set$ operad, then one can use the maps
maps $\sigma_i:\O(n)\to \O(n-1)$, to induce map $\Coop(n)\to \Coop(n-1)$ by using a basis $\delta_{\tau}$ of $\Coop(n)$ and sending $\delta_{\tau}\to \delta_{\sigma_i(\tau)}$.
If these are compatible with the operad composition and in the symmetric case with the permutation group action, then one obtains an amputation structure. Note, we do not change the variance which would be the case if we were to use functoriality. Instead, we are using a basis to retain the variance.

The paradigmatic example is the deletion of tails planar planted trees or labeled rooted trees. Here a $\sigma_i$ removes the $i$-th tail.  This is either the $i$-th tail in the planar order or the tail labeled by $i$. Likewise $\sigma_s$ removes a tail labeled by $s$.  The equations \eqref{cutgammaieq} and \eqref{cutequseq} then just state that it does not matter whether one first removes a tail and then cuts an internal edge, or vice--versa.
\end{ex}

For a non--$\Sigma$ co--operad without an $\Coop(0)$, but with an amputation structure,  we define $\B^{\it amp}=\colim_{\DDelta}\Coop^{nc}$, where the co--limit is taken along the directed system of amputation morphisms $\sigma_i$ encoded by $\DDelta$.
In the symmetric case,  again without $\Coop(\emptyset)$ we take the co--limit over the directed system $\DDelta S$ induced by the $\sigma_s$: $\B^{\it amp}:=\colim_{\DDelta S}\B$.
Due to the equivariance the co--limit factors through to the coinvariants.

In particular, under our standard assumption, the co--limit is the quotient $\B^{\it amp}=\bar T(\bigoplus_{n}\Coop(n)/\!\sim)$ where $\sim$ is the equivalence relation generated by $\O(n)\ni a_n\sim b_{n-1}\in \O(n-1)$ if there exists an $i$ such that $\sigma_i(a_n)=b_{n-1}$. In the symmetric case the co--limit can be computed as $\B^{\it amp}=\bar S(\bigoplus_{n}\Coop(n)_{\SS_n}/\!\sim)$ where now $\O(n)_{\SS_n}\ni [a_n]\sim [b_{n-1}]\in \O(n-1)_{\SS_{n-1}}$ if there exists an $i$ such that $\sigma_i(a_n)=b_{n-1}$.

\begin{thm}
\label{freeampthm}
 For an $\O$ with an amputation structure, $\B^{\it amp}$ is a bi--algebra.
In both cases, these co--limits descend to the quotient $\H$. Set $\H^{\it amp}=colim_{\DDelta} \B/\I$.
 It has a co--unit if and only if $\O$ has a split unit. If $\O$ is co-connected,
$\H^{\it amp}$ is a Hopf algebra, and if $\O$ is symmetric co--connected $\H^{\it amp}$ is a commutative Hopf algebra.\qed
\end{thm}

\begin{proof}
The co--product is defined via the $\check\gamma$ and these are compatible with the amputation, so that the co--product descends.
The product and compatibilities also descend as the $\sigma_i$ and $\sigma_s$ commute with disjoint unions. Since there is no $\Coop(0)$ the generator of the ideal $1-|$ remains untouched.
\end{proof}

\begin{ex}
Taking this co--limit on the planar planted trees or the labeled rooted trees yields the Hopf algebras of Connes and Kreimer. What is left in the co--limit are representatives which are trees without tails, sometimes called amputated trees \cite{Kreimer,BBM}.  The root half--edge, which is not amputated by the co--limit, simply indicates the planar order at the root in the planar case and can be forgotten in the rooted case, by marking the root vertex instead. The co--limit also removes any tails  after an edge is cut, thus it also removes any tails of appearing from the cut edges effectively deleting the edges; the leaf half edge is clipped, and the non--deleted half--edges of the cut edges are all root half--edges, and can also be forgotten by the procedure above.  Also notice that if the tail itself is cut, the condition that $|=1$ takes care of the  cuts ``above'' a leaf vertex or ``below'' the root vertex in the conventions of Connes and Kreimer and hence the co--product is exactly that of Connes and Kreimer, both in the commutative/symmetric and noncommutative/non--$\Sigma$ case.
A pictorial representation is given in Figure \ref{CKcoprod}.
\end{ex}

\subsubsection{Adding a formal $\Coop(0)$ to compute $\H^{\it amp}$}
\label{colimpar}
It is convenient to think of the elements  representing the co--limits as being elements of a $\Coop(0)$. In this case, $\Coop(0)$ becomes a final object and the co--limits are more easily computed. Notice that $\Coop(1)$ is not a final object for the clipping structure, as there are $n$ morphisms from $\Coop(n)$ to $\Coop(1)$ ``forgetting'' all but one $i$.
Considering the co--limits to lie in an additional $\Coop(0)$ yields an equivalent formulation whose construction is more involved, but whose pictorial representations are more obvious.
We assume that $\O$ is split unital and, we now allow $\Coop(0)$ and consider $\sigma_1:\Coop(1)\to \Coop(0)$ with the conditions that

\begin{enumerate}
    \item $\Coop(0)$ only contains elements that are images of $\sigma_i$ from higher $\Coop(n)$ and $\sigma_1(a)=[a]$ that is the class that $a$ represents in the co--limit if $a\notin \unit$ and
     $\sigma_1(|)=1$.
       \item Define a co--product and product structure on $\Coop(0)$ by using representatives in $\Coop(n),n\geq1$ and the co--operad and multiplication structure for the $\Coop(n),n \geq 1$, where  $\sigma_1(|)=1$ is identified with as the unit of the product on $\Coop(0)$.

\end{enumerate}
It is clear then, that $\Coop(0)$ is just the co--limit $(\unit\oplus\B^{\it amp})/(1-|)=\H^{\it amp}$
with the induced structures.

\begin{prop}
 Enlarging $\Coop$ in this way and taking the co--limit directly yields $\Coop(0)=\H^{\it amp}$.
\qed
\end{prop}

The computation in the co--limit version can be made using the formalism of \cite[\S\ref{P2-kreimerapp}]{HopfPart2}.

\subsection{Coaction}
\label{coactionpar}
\subsubsection{$\Coop(0)$ and coaction}
As we have seen in the last section, although it sometimes does make sense to include $\Coop(0)$,
in general $\Coop(0)$  may prevent local finiteness and is hence be a potential hindrance to being co-nilpotent. It may also cause issues for summing over the maps $\Delta_k$.
There is a remedy in which $\Coop(0)$ is viewed as a co--algebra over a co--operad and likewise $\Coop^{nc}(0)$ as a co--module over the Hopf algebra.

For this consider a co--operad with  $\Coop(0)$.
Then $\Coop(0)$  becomes a co-algebra over the co--operad $\{\Coop(n),n\geq 1\}$,  whose  co--operad structure is defined by using only the terms of the co-product of the form $\check \gamma_{n_1,\dots,n_k}$ with $k$ and $n_i>0$, via $\check\gamma_{0^k}:=\check \gamma_{k;0,\dots,0}:\Coop(0)\to \Coop(k)\otimes \Coop(0)^{\otimes k}$, $k>0$.
We will assume that this co-action is locally finite, that is for any $a\in \Coop(0)$ there sum over all $\check \gamma_{0^k}(a)$ is finite.

Set $\B=\bar T\Coop_{\geq 1}$ or $\B=\bar S\Coop_{\geq 1}$ in the symmetric case,  omitting the zero summand $\Coop_{\geq 1}=\bigoplus_{n\geq 1}\Coop(n)$.
Denote by $\mu$ the free (symmetric) multiplication, then   $\mu:\Coop^{nc}(0)\otimes \Coop^{nc}(0)\to \Coop^{nc}(0)$  makes $\Coop^{nc}(0)$ into an algebra whose multiplication is compatible with the $\check\gamma_{0^k}$ by definition.

Summing over all the $\check\gamma_{0^k}$ and post--composing with the multiplication on the factors of $\Coop(0)^{\otimes k}$ we get a co--algebra map
\begin{equation}
\label{coperadmap}
  \check\rho:  C\to \B\otimes C
\end{equation}
where now $C=\bar T\Coop(0)=\Coop^{nc}(0)$.

\subsubsection{Motivating examples}
Consider $\O$ with $\O(0)$ then $\O(0)$ is an algebra over $\O^\oplus=\bigoplus_{n\geq 1} \O(n)$ via: $\gamma:\O(k)\otimes\O(0)^{\otimes k}\to \O(0)$.
Dually, we see that $\Coop(0)$ is a co--algebra over the co--operad $\Coop(n)$.
This construction extends to $\Coop^{nc}$, where $\Coop^{nc}(0)=T\Coop(0)=\bigoplus_n \Coop(0)^{\otimes n}$.
Now, we do have a well defined co--operad with multiplication structure on $\B=\bigoplus_{n\geq 1}\Coop^{nc}(n)$
and by restriction, we have a co--module given by extending
\begin{equation}
\label{coaction}
    \Coop(0)\stackrel{\check\rho}{\to} \bigoplus_{k\geq 1}\Coop(k)\otimes \Coop(0)^{\otimes k}
\end{equation}

\begin{ex}
In the case of trees, we can also consider trees without tails. These will have {\em leaf vertices}, i.e.\ unary non-root vertices. Making admissible cuts those that leave a trunk that has only tails and no leaves, and branches that only have leaf vertices and no tails, we get a natural co--algebra structure.
This is precisely the co--action \eqref{coaction}.

The construction is related to \S\ref{colimpar} in that the latter differs from the former only in applying the co--limit deleting the tails on the left  hand side as well.

\end{ex}

\subsubsection{Co-algebras} The construction above, immediately generalizes to
 any co--algebra $C$  over a co--operad $\Coop$ with a compatible multiplication.  Such a co--algebra by definition has $\check\rho^n :C \to \Coop(n)\otimes C^{\otimes n},n>0$. The extra datum is an associative algebra structure for $C$: $\mu_C:C\otimes C\to C$, which is compatible with the co--algebra over the co--operad  $\Coop$ in the usual way. We furthermore assume that the $\check\rho^n$ are locally finite.
Then  we can define co--algebra maps $\check\rho:C\to \B\otimes C$ by setting $\B=\bigoplus_{n\geq 1}\Coop(n)$ as usual and defining the co--action by $\check\rho=\sum_n \mu^{\otimes n-1}\circ \check \rho^{n}$.

\subsubsection{Half--infinite chains, co--algebra}
\label{semiinfex}
One interesting algebra comes from adding $\dottree \, \infty$, representing a half--infinite rooted chain, with
$\Delta(\dottree\, \infty)=\sum_{n\geq 0}
\dottree \, n\otimes \dottree \, \infty$. This is an example where there is a bi--grading in which the co--product is finite in each bi--degree, the degrees lying in $\bf N_0\cup \{\infty\}$.  With $s=\sum_{n\geq 0} \dottree \, n$, we see from the associativity that $\Delta^n(\dottree \, \infty)=s^{\otimes n-1}\otimes \dottree \, \infty $ and $s$ is a group like element.
This fact leads to interesting physics, \cite{KreimerLinear}.

We can also treat the half--infinite chain as a co--algebra $C=span(\dottree\, \infty)$, and $\H$ being the graded Hopf algebra of trees, graded by the number of vertices or its sub--algebra of finite linear trees,
where $\check\rho(\dottree\, \infty)=\sum_{n\geq 0}
\dottree \, n\otimes \dottree \, \infty$. Everything is finite in each degree.

Lastly, we can consider the larger co--algebra is spanned by Dirac--trees
that is rooted trees with semi--infinite chains as leaves. The co--action is to cut the semi--infinite tree with a cut that leaves a finite base tree and infinite Dirac--sea--type tree branches.

Having infinite chains is not that easy, but this will be considered elsewhere.

\subsection{Grading for the quotients and Hopf algebras}
\label{gradingpar}

Before taking the Hopf quotient there was the grading by $n-p$ for the graded case and filtered accordingly in the general case. Now, $|-1$ has degree $0$,
so the grading descends to the quotient $\H=\B/\I$. This works in the symmetric and the non--symmetric case. In the amputation construction, this grading will not prevail  as the co--limit kills the operadic grading.
However, in the case that $\H$ is indeed a connected Hopf algebra, there is the additional filtration by the coradical degree $r$. We can lift this coradical filtration to $\B$ and $\H_q$.

\begin{lem}
For an almost connected (symmetric or non--$\Sigma$) co--operad with multiplication, the coradical filtration lifts to $\B$ and $\H_q$.
The lifted coradical filtration $\Rad$ is compatible with multiplication and co--multiplication and in particular satisfies.
$\bar\Delta (\Rad^d) \subset \Rad^{d-1} \otimes \Rad^{d-1}$.
This filtration descends to $\H$ and $\H^{\it amp}$ respectively.

\end{lem}
\begin{proof}
The coradical degree of an element $a$ is  given by its quotient image, in which any occurrence of $|$ is replaced by $1$.
Since the lift or $|$ will lie in $\Rad^0$ and both $1$ and $|$ are group like due to the bi--algebra equation, the filtration is compatible with the multiplication and the co--multiplication.
Due to the form of $\Delta$ in Proposition \ref{pointedprop}, see \eqref{reducedeq}, we see that the first term of $\Delta$ descends as the only term of the type to $1\otimes a+a\otimes 1$ and hence $\bar\Delta$ descends to $\bar\Delta$ on $\H$.
This shows the claimed property of $\bar\Delta$ on $\B$. The fact that the filtration descends through amputation is clear.
\end{proof}

\begin{df}
We call the coradical filtration of $\B$ and consequentially of $\H_q,\H$ well behaved, if \begin{equation}
    \bar\Delta(\Rad^i)\subset \bigoplus_{p+q=i}\Rad^{p}\otimes \Rad^q
\end{equation}
We will use the same terminology for all the cases that is $\H$ and $\H^{\it amp}$ in both
the symmetric and the non--$\Sigma$ case.
\end{df}
Since $\mu$ is always additive in the coradical filtration due to the bi--algebra equation, we obtain:

\begin{prop}
If the coradical filtration is well behaved, then $\H$ and $\H^{\it amp}$ are graded by the co-radical degree. \qed
\end{prop}

\begin{lem}
For $\Coop^{nc}$, if $\Coop(1)$ is reduced, then the maximal coradical degree for an element in $\Coop(n,p)$ is  $n-p$.
\end{lem}
\begin{proof}
In $\Coop(n)$, applying $\bar\Delta$ we generically get a term $\Coop(n-1)\otimes \Coop(2)\otimes \Coop(1)^{\otimes n-2}$, repeating this procedure and ``peeling off'' an $\Coop(2)$ then the maximal coradical degree will be $n-1$. Both $n$ and $p$ are additive under $\mu$ which finishes the argument.
\end{proof}

\begin{ex}
\label{surjgradex}
Goncharov's and Baues' Hopf algebras are examples where this maximum is attained. Indeed any surjection $\underline {n}\twoheadrightarrow \underline{1}$ factors as
$(\pi_2\amalg id \amalg\dots\amalg id)\circ \dots \circ (\pi_2\amalg id)\circ \pi_2:\underline{n}\twoheadrightarrow \underline{n-1}\twoheadrightarrow \dots \twoheadrightarrow \underline{1}$ where $\pi_2:\underline{2}\twoheadrightarrow \underline{1}$ is the unique surjection. Using Joyal duality, see Appendix \ref{Joyalapp}, the same holds true for  double base--point preserving injections.
\end{ex}

\begin{ex}
\label{gradingex}
Another instructive example is the Connes--Kreimer Hopf--algebra of rooted forests with tails. Here the coradical degree of a tree is simply $E+1=V$, where $E$ is the number of edges and $V$ is the number of vertices. This is so, since each application of $\bar\Delta$ will cut at least one edge and cutting just one edge is possible. Since we are dealing with a tree $E+1=V$. For a forest with $p$ trees, this it is $E+p=V$

This is the grading that descends to $\H^{\it amp}$. The same reasoning holds for the symmetric and the non--symmetric case.

Now there are two different gradings. The coradical degree and the original grading by $n-p$. There is a nice relationship here. Notice that $n$ is the number of tails, $p$ is the number of roots thus for a forest the number of flags $F=n+p+2E=n-p+2(E-p)=n-p+2V$ and this is a third grading that is preserved. It is important to note that in the Hopf algebra $|=1$ and does not count as a flag.

Vice--versa since the flag grading and the $n-p$ grading are preserved it follows that the coradical grading is preserved,
giving an alternative explanation of it.
\end{ex}

\begin{prop}
 A unital operad $\{\O(n)\}$, with an almost connected $\Coop(1)$,  $\Coop^{nc}$  has a well behaved coradical grading.
\end{prop}

 \begin{proof}
Fix an element $\check a$ dual to $a\in \O(n)$.
Due to the assumptions there are finitely many
iterated $\gamma$ compositions that result in $a$.
Each of these can be presented by a level tree whose vertices $v$ are decorated by elements of $\O(|v|-1)$. If we delete the nodes decorated by the identities,  what remains are trees with vertices decorated by non--identity operad elements, see \cite{feynman}[2.2.1] for details about this construction. The number of edges of the tree then represent the number of operadic concatenations, and dually the number of co--operad and hence $\Delta$ operations that are necessary to reach the decomposition. It follows that the coradical degree is equal to the maximal value of $E+1=V$.
This is easily seen to be additive under $\mu$ and preserved under $\Delta$.
The computations is parallel to the one above for the Connes--Kreimer algebra of trees, only that now the vertices are decorated by operad elements. In this picture, $V$ is also the word length of the expression of an element as an iterated application of $\circ_i$ operations.
 \end{proof}

\begin{rmk}
This proposition also reconciles the two examples, Goncharov and Connes--Kreimer. Furthermore it explains the ``lift'' of Goncharov to the Hopf algebra of trivalent forests. Indeed the expression in Example \ref{surjgradex} is word of length $n-1$ represented by a binary rooted tree. See also Example \cite[\ref{P2-operadex1}]{HopfPart2}.

The fact that there is no general Hopf algebra morphism between the two Hopf algebras is explained in
\cite[\S \ref{P2-functpar}]{HopfPart2}
\end{rmk}

\section{A Generalization: Hopf algebras from  co--operads with multiplication}
\label{gencooppar}

In this section, we generalize \S\ref{operadpar} by replacing  the free algebra with a more general multiplicative structure. This is captured in the concept of a co-operad with multiplication which is enough to obtain a bi--algebra. Already at this level, the theory is more involved,  as units and co--units for the bi--algebra become more difficult and lead to restrictions. Another aspect is that the natural gradings in the free case only correspond filtrations. However, regarding the  graded objects associated to the filtered objects, we can prove a theorem stating that if units and co--units exist, the bi--algebra is a deformation of the associated graded, which in turn is the quotient of a free algebra on a co--operad.

\subsection{From non-\SSigma{} co--operads with multiplication to bi--algebras}

We will now generalize the co--operad $\{\B(n)\}=\{\Coop^{nc}(n)\}$ of  \S\ref{freecooppar} with the free multiplication
$\mu=\otimes:\B(m)\otimes\B(n)\to \B(m+n)$ to a co--operad with a compatible multiplication, where compatibility guarantees the bi--algebra structure.
 $(\{\B(m)=\Coop^{nc}(m)\},\otimes)$ is a special case and will be called the free construction and all theorems apply to this special case.

\begin{df}\label{ourdef}
A \emph{non--\SSigma{} co--operad with multiplication} $\mu$ is a non--\SSigma{} co--operad $(\Coop,\check\gamma)$ together with a family of maps,
$n,m\geq0$,
\begin{equation}\mu_{n,m}:\Coop(n)\otimes \Coop(m)\to \Coop(n+m),\end{equation}
which satisfy the following compatibility equations:
\begin{enumerate}
\item The maps $\mu_{n,m}$ are associative: $\mu_{l,n+m}\circ (id \otimes \mu_{n,m})=\mu_{l+n,m}\circ (\mu_{l,n}\otimes id)$.
    \item For any $n,n'\geq 1$ and partitions $m_1+\dots+ m_k=n$ and $m_1'+\dots+ m_{k'}'=n'$, write $\check\gamma$ and $\check\gamma'$ for $\check\gamma_{k;m_1,\dots,m_k}$ and $\check\gamma_{k';m'_1,\dots,m'_{k'}}$ respectively, and write $\check\gamma''$ for $\check\gamma_{k+k';m_1,\dots,m_k,m'_1,\dots,m'_{k'}}$.
    Then the following diagram commutes
\begin{equation}
\label{compateq}
\vcenter{\xymatrix@C+2pc{\dto_{\mu_{n,n'}}\Coop(n)\otimes\Coop(n')
\rto^-{\pi\,{(
\check\gamma\otimes \check\gamma'
)}}&
\displaystyle
\Coop(k)\otimes
 \Coop(k')\otimes
\bigotimes_{r=1}^k\Coop(m_r)\otimes
\bigotimes_{r'=1}^{k'}\Coop(m_{r'}')
\ar[]+<0pt,-8pt>;[d]+<0pt,11pt>^{\mu_{k,k'}\otimes\id}\\
\Coop(n+n')
\rto^-{\check\gamma''}&
\displaystyle\Coop(k+k')\otimes\bigotimes_{r=1}^{k}\Coop(m_{r})
\otimes\bigotimes_{r'=1}^{k'}\Coop(m_{r'}')
}}
\end{equation}
Here $\pi$ is the isomorphism which permutes the $k+k'+2$ tensor factors according to the $(k+1)$-cycle $(2\,3\,\ldots\, k+2)$.
\item
If $m_1''+\dots+m_{k''}''=n+n'$ is a partition of $n+n'$ which does not arise as the concatenation of a partition of $n$ and a partition of $n'$ (that is,  there is no $k$  such that $m_1''+\dots+m_k''=n$ and $m_{k+1}''+\dots+m_{k''}''=n'$) then the composite
\begin{equation}
\Coop(n)\otimes\Coop(n')\xrightarrow{\mu_{n,n'}}\Coop(n+n')\xrightarrow{\check\gamma_{k'';m_1'',\dots,m_{k''}''}}
\Coop(k'')\otimes\bigotimes_{r''=1}^{k''}\Coop(m_{r''}'')
\end{equation}
is zero.
\end{enumerate}

Under the completeness assumption, the $\mu_{n,m}$ assemble into a map $\mu$ and the $\check\gamma_{k,n_1\cdc n_k}$ assemble to a map $\check\gamma$.
satisfying the compatibility relation
\begin{equation}\label{oureqn}
\check\gamma(\mu(a\otimes b))=\mu(\pi(\check\gamma(a)\otimes\check\gamma(b)))
\end{equation}
where $\pi$ is the permutation that permutes the first factor of $\check\gamma(b)$ next to the first factor of $\check\gamma(a)$.

A {\it morphism of co--operads with multiplication} $f:\Coop\to\check\P$ is a morphism of co--operads
which commutes with the multiplication, $f_{m+n}\mu_{n,m}=\mu_{n,m}(f_n\otimes f_m)$.

\end{df}

Denote by $\mu^{k}$  the $k$-th iteration of the associative product $\mu$, e.g.: $\mu^1=\mu, \mu^k=\mu\circ (\mu^{k-1} \otimes id)$.

\begin{thm}
\label{genbialgthm} Using the basic assumption \ref{basicas},
let $\Coop$ be a co--operad with compatible multiplication $\mu$ in an Abelian symmetric monoidal category with unit $\unit$.
Then
\begin{equation}
\B:=
\bigoplus_n \Coop(n)
\end{equation}
is a (non-unital, non-co--unital) bi--algebra, with multiplication $\mu$, and comultiplication $\Delta$ given in short form notation by $(\id\otimes\mu)\check\gamma$:
In more details: $\Delta=\sum_k\Delta_k$, $\Delta_n=\sum_k \Delta_{k,n}$ with $\Delta_{k,n}= (\id\otimes\mu^{k-1})\check\gamma_{k,n}$, where $\check\gamma_{k,n}$ is the sum over all $\check\gamma_{k;n_1\cdc n_k}$ with $\sum_i n_i=n$.
\begin{equation}
\label{deltadefeq}
\vcenter{\xymatrix@R-1.3pc{\ar@/_6mm/[rrdd]_(0.3){\textstyle{\Delta_n:=\sum_k(\id\otimes\mu^{\otimes  k-1})\check\gamma_{k,n}}\;\;}\Coop(n)
\rrto^-{\;\check\gamma\;}
&&\hspace*{-6mm}
{\displaystyle\bigoplus_{\substack{k\geq1,\\ \!\!\!n=m_1+\dots+m_k\;}}}
\hspace*{-4mm}
\left(\Coop(k)\otimes{\displaystyle\bigotimes_{r=1}^k\Coop(m_r)}\right)
\ar[]-<0mm,6mm>;[dd]^(0.4){\;\id{}\,\otimes\,{}\mu^{k-1}\;}\\ \\
&&\displaystyle\bigoplus_{k\geq1}
\Coop(k)\otimes\Coop(n).
}}
\end{equation}
Morphisms of co--operads with co--multiplication induce homomorphisms of bi--algebras.
\end{thm}

\begin{proof}
The multiplication $\mu$ is associative by definition.
The compatibility of $\mu$ with $\check\gamma$, together with the associativity of $\mu$, shows that $\mu$ is a morphism of co--algebras, $\Delta\mu=(\mu\otimes\mu)\pi(\Delta\otimes\Delta)$:
\begin{equation}
\resizebox{\textwidth}{!} {
\xymatrix{
\dto_{\mu_{n,n'}}\Coop(n)\otimes\Coop(n')\rto^-{\pi\,\check\gamma^{\otimes 2}}
\ar@{}[dr]|{\text{{compatibility}}}
&
\ar@{}[dr]|{\quad\qquad\text{ { associativity}}}
\displaystyle\Coop(k)\otimes\Coop(k')\otimes
\bigotimes_{r=1}^k\Coop(m_r)
\otimes
\bigotimes_{r'=1}^{k'}\Coop(m_{r'}')
\ar[]+<0pt,-9pt>;[d]+<0pt,11pt>^{\mu_{k,k'}\otimes\id}
\ar[r]-<14mm,0mm>^-{\id\otimes\mu\otimes \mu}&
{\begin{array}{c}\Coop(k)\otimes\Coop(k')
\\[1mm]\;\;\;\otimes\;
\Coop(n)\otimes\Coop(n')
\dto^{\mu_{k,k'}\otimes\mu_{n,n'}}\end{array}}\\
\Coop(n+n')
\rto^-{\check\gamma}&
\displaystyle\Coop(k+k')\otimes\bigotimes_{r=1}^{k}\Coop(m_{r})
\otimes\bigotimes_{r'=1}^{k'}\Coop(m_{r'}')
\rto^-{\id\otimes\mu}&
\Coop(k+k')\otimes\Coop(n+n').
}
}
\end{equation}

For the co--associativity,
one has to prove that $(id\otimes \Delta_{l,n})\Delta_{k,n} =(\Delta_{k,l}\otimes id )\Delta_{l,n}:\Coop(n)\to \Coop(k)\otimes \Coop(l)\otimes \Coop(n)$, which can be done term by term using (\ref{coopcoasseq}) and (\ref{compateq}).

Explicitly fix a $k$--partition $n_1,\dots n_k$ of $n$ an $l$ partition $(m_1,\dots, m_l)$ of $n$.
by compatibility the left hand side vanishes unless $(m_1,\dots, m_l)$ naturally decomposes into the
list $(n^1_1,\dots,n^1_{l_1},n^2_{1},\dots, n^2_{l_2},\dots, n^k_{1},\dots, n^k_{l_k})$ where $n^i_j$ is a partition of $n_i$.  This yields the
$k$ partition $(l_1,\dots l_k)$ of $l$. Starting on the r.h.s.\ that is with $(m_1,\dots, m_l)$ and $(l_1,\dots l_k)$,
we decompose the list  $(m_1,\dots,m_l)$ as above, which determines the $n_i=\sum_j n_j^i$.
Set $\Delta_{n_1,\dots, n_k}=(id\otimes \mu^{\otimes k-1})\circ  \check \gamma_{k;n_1,\dots, n_k}$, then:
\begin{eqnarray}
&&(id\otimes \Delta_{l;m_1,\dots,m_l})\Delta_{k;n_1,\dots n_k}\nn\\
&=&
(id\otimes id \otimes \mu^{l-1} )(id\otimes [\check\gamma_{l;m_1,\dots,m_l}\circ \mu^{k-1}]) \circ \check\gamma_{k;n_1,\dots,n_k}\nn\\
&=&(id\otimes id \otimes \mu^{l-1} )(id\otimes \mu^{k-1} \otimes id^{\otimes l}) \circ \pi
\circ (id\otimes \check\gamma_{l_1;n^1_1,\dots, n^1_{l_1}}\otimes
\nn\\
&&\check\gamma_{l_2;n^2_1,\dots, n^2_{l_2} } \otimes\dots \otimes  \check\gamma_{l_k;n^k_1,\dots, n^k_{l_k}})\circ \check\gamma_{k;n_1,\dots,n_k}\nn\\
&=& (id\otimes \mu^{k-1} \otimes id) (id\otimes id^{\otimes k} \otimes \mu^{l-1} ) (\check\gamma_{k;l_1,\dots, l_k}\otimes id^{\otimes l})\nn\\
&&\circ \check\gamma_{l;n^1_1,\dots,n^1_{l_1},n^2_{1},\dots, n^2_{l_2},\dots, n^k_{1},\dots, n^k_{l_k}}\nn\\
&=&[([id\otimes \mu^{k-1}] \check\gamma_{k;l_1,\dots, l_k}) \otimes id](id\otimes \mu^{l-1})\check\gamma_{l;m_1,\dots,m_l}\nn\\
&=&(\Delta_{k;l_1,\dots,l_k}\otimes id)\Delta_{l;m_1,\dots,m_l}
 \end{eqnarray}

where $\pi$ is the permutation that shuffles all the right factors next to each other as before.

\end{proof}

\subsection{A natural depth filtration and the associated graded}
In the free construction  of \S\ref{subsect:free}  there is a natural grading by tensor length. In the general case, there is only a filtration, the depth filtration.
The grading appears, as expected, on the associated graded object.

\begin{df}
We define the {\em decreasing depth filtration} on a co--operad $\Coop$ as follows: $a\in F^{\geq p}$  if
$\check\gamma(a)\in \bigoplus_{k\geq p}\bigoplus_{(n_1,\dots,n_k):\sum_i n_i=m}\Coop(k)\otimes \Coop(n_1)\odo \Coop(n_k)$.
So $\B=F^{\geq 1}\supset F^ {\geq 2}\supset \dots $ and $\bigcap_p F^{\geq p}=0$, since we assumed that there is no $\Coop(0)$ or at least that $\Coop$ is locally finite.

We define the depth of an element $a$ to be the maximal $p$ such that $a\in F^{\geq p}$.

This filtration induces a depth filtration $F^{\geq p}T\B$ on the tensor algebra $T\B$ by giving
$F^{\geq p_1}\odo F^{\geq p_k}$ depth $p_1+\dots +p_k$. Note that any element in $T^p\B$ will have depth at least $p$.

\end{df}

\begin{prop}
\label{structureprop}
The following statements hold for a co--operad with multiplication with empty $\Coop(0)$:
\begin{enumerate}
\item   The algebra structure is filtered: $F^{\geq p}\cdot F^{\geq q}\subset F^{\geq p+q}$.
\item  The co--operad structure satisfies $\check\gamma(F^{\geq p})\subset F^{\geq p}\otimes T^{\geq p} \B$
where:\\ $T^{\geq p}\B=\bigoplus_{i=p}^{\infty}(\B)^{\otimes i}\subset F^{\geq p} T\B$;  and more precisely:\\
$\check\gamma_{k;n_1,\dots,n_k}:\Coop(n)\cap F^{\geq p}\to [\Coop(k)\otimes \Coop(n_1)\odo \Coop (n_k)]\cap F^{\geq p}\otimes F^{\geq k}T\B$.
\item  The co--algebra structure satisfies: $\Delta(F^{\geq p})\subset F^{\geq p}\otimes F^{\geq p}$
and more precisely $\Delta_k(F^{\geq p})\subset F^{\geq p}\otimes F^{\geq k}$.
\item  \label{partd} $\Coop(n)\cap F^{\geq n+1}=\emptyset$.
\end{enumerate}
 \end{prop}
\begin{proof}
The first statement follows from the compatibility (\ref{compateq}).
The second statement follows from the Lemma \ref{depthlem} below. The more precise statement on the right part of the filtration stems from
the fact that $T^k\B\subset F^{\geq k}T\B$.
The third statement then follows from a) and b), since
 there are at least $p$ factors on the right before applying the multiplication  and the filtration starts at $1$.
 This shows that the right factor is in $F^{\geq p}$.
 Finally, for $\Coop(n)$ the greatest depth that can be achieved  happens when all the $n_i=1:i=1, \dots ,k$ and since they sum up to $n$ this is precisely at $k=n$.
\end{proof}

\begin{lem}\label{depthlem}
If $a^p\in \B$ of depth $p$ let $\check\gamma_{k;n_1,\dots,n_k}(a^p)=\sum a^{(0)}_{(p_0)}\otimes a^{(1)}_{(p_1)}\odo a^{(k)}_{(p_k)}$, where we used a generalized Sweedler notation for both the co--operad structure and
the depth,
then the terms of lowest depth will satisfy $p_0= \sum_{i=1}^k p_i\geq p$.

\end{lem}
\begin{proof}
To show the equation, we  use co--associativity of the co--operad structure.  If we apply $\id \otimes \check\gamma^{\otimes k}$ we get  least  $1+k+\sum_{i=1}^k p_i$ tensor factors from the lowest
depth term, since we assumed that $\Coop(0)$ is empty.  On the other hand applying $\check\gamma\otimes id^{\otimes k}$ to the terms of lowest depth,
 we obtain  elements with at least $1+p_0+k$ tensor factors. Since elements of higher depth  due to equation (\ref{coopcoasseq}) produce more tensor factors  these numbers have to agree. Since all the $p_i\geq 1$ their sum is $\geq p$.

\end{proof}

\subsubsection{The associated graded bi--algebra}

We  now consider the associated graded objects $Gr^p:=F^{\geq p}/F^{\geq p+1}$ and denote the image of $\Coop(n)\cap F^p$ in
$Gr^p$ by $\Coop(n,p)$.
 An element of depth $p$ will have non--trivial image in $Gr^p$ under this map.
We denote the image of an element  $a^p$ of depth $p$ under this map by $[a^p]$ and call it the principal part.

We set $Gr=\bigoplus Gr^p$, by Proposition \ref{structureprop} \eqref{partd}: $Gr=\bigoplus_p \bigoplus_{n=1}^p \Coop(n,p)$ and define a grading by
 giving the component $\Coop(n,p)$ the total degree $n-p$.

\begin{cor}
\label{grcor}
By the Proposition \ref{structureprop} above we obtain maps
\begin{enumerate}
\item $\mu:Gr^p\otimes Gr^q\to Gr^{p+q}$  by taking the quotient by $F^{\geq p+1}\otimes F^{q+1}$ on the left and $F^{\geq p+q+1}$ on the right
\item \label{grcoopcond} $\check\gamma^{p,k} :Gr^p\to Gr^p\otimes (Gr^1)^{\otimes k}$ by taking the quotient by $F^{\geq p+1}$ on the left and by $F^{\geq k+1}T\B\cap T^{\otimes k}\B$ on the right.  In particular, $\check \gamma(Gr^1)\subset Gr^1\otimes TGr^1$
\item $\Delta^{p,k}:Gr^p\to Gr^p\otimes Gr^k$ by taking the quotient by $F^{\geq p+1}$ on the left and by $F^{\geq k+1}$ on the right.
\item \label{grcomodcond} $\Delta^p:Gr^p\to Gr^p\otimes Gr$ via
$\Delta^p =\sum_k \Delta^{p,k}$
\item $\Delta:Gr\to Gr\otimes Gr$ via $\Delta=\sum_p \Delta^{p}$
\end{enumerate}

\end{cor}

\begin{prop}
\label{gradedprop}
$\Gr$ inherits the structure  of a non-unital, non-co--unital graded bi--algebra with the degree of $\Coop(n,k)$ being $n-k$.
Each $\Gr^p$ is a non-co--unital  comodule over $Gr$, and $\Gr^1$ is a co--operad.
\end{prop}

\begin{proof}
Most claims are straightforward from the definitions in the corollary.   The two final statements in particular, are a reformulation of Corollary \ref{grcor} \eqref{grcomodcond} and \eqref{grcoopcond}, respectively.

Due to the bi-algebra equation the multiplication preserves grading:
$\Coop(n,p)\otimes \Coop(m,q)\stackrel{\cdot}{\to} \Coop(n+m,p+q)$ . The degree on both sides is $m+n-p-q$. For the comultiplication:
 $\Delta_{k,n}(\Coop(n,p)) \subset  \Coop(k,p)\otimes \Coop(n,k)$. The degree on the left is $n-p$, and on the right it is $k-p+n-k=n-p$.
 Thus the comultiplication also preserves degree.
 \end{proof}

\begin{ex}
For the free construction $\B=\Coop^{nc}$,
we obtain

\begin{eqnarray}
F^{\geq p}&=&\bigoplus_{k\geq p}\bigoplus_{(n_1,\dots,n_k)}\Coop(n_1)\odo \Coop(n_k)\\
Gr^k&=& \bigoplus_{(n_1,\dots,n_k)}\Coop(n_1)\odo \Coop(n_k)\\
\Coop^{nc}(n,k) &=& \bigoplus_{(n_1,\dots,n_k):\sum_i n_i=n}\Coop(n_1)\odo \Coop(n_k)
\end{eqnarray}

This means that  the depth of an element of $\B$ given by an elementary tensor is its length. The associated graded is isomorphic to the $\B$ which has a double grading by depth and operadic degree. Furthermore $Gr^1=\check \O$ and $\B=(Gr^1)^{nc}=\Coop^{nc}$.
\end{ex}

\begin{cor}
\label{Grcor}
Since  $Gr^1$ is a co--operad $(Gr^1)^{nc}=\bar T Gr^1$ yields a co--operad with (free) multiplication. The multiplication $\mu$ and its iterates
define a morphism  $(Gr^1)^{nc}\to Gr$ of co--operads with multiplication preserving the filtrations and
hence give a morphism of (non-unital, non--co--unital) bi--algebras.
\end{cor}
\begin{proof}
Indeed the multiplication map gives such a map of algebras, since $Gr^{nc}$ is the free algebra.
The compatibility  \eqref{oureqn} ensures that this is also a map of co--operads with multiplication. The compatibility with the filtration follows from the definitions.
\end{proof}

\subsection{Unital and co--unital bi--algebra structure }
There is no problem adding a unit. For the existence of a bi--algebraic co--unit  in the free construction $\Coop^{nc}$
 the existence of a  co--operadic co--unit for $\Coop$ is sufficient.
For the general case, things are more complicated and worked out in detail in this section. An upshot is that in the free construction the existence of a co--operadic co--unit is also necessary.

In general, the existence of a right bi--algebra co--unit, is equivalent to the co--operad having a right co--unit, which extends to a multiplicative family. So that
the existence of a right co--operadic co--unit is a necessary condition.
As proven below, a co--operadic co--unit determines a unique candidate for a bi--algebra co--unit, which, however, does not automatically work in general; it does in the free construction. We give several conditions that are necessary for this, treating the cases of left and right co--units separately with care.

Having a left co--algebra co--unit for $\B$ fixes the structure of the associated graded as a quotient of the free construction on $Gr^1$ via the map of Corollary \ref{Grcor} and  $\B$ is a deformation of this quotient, see Theorem \ref{defthm}.

\subsubsection{Unit} If there is no element of operad degree $0$ then, as the multiplication preserves operad degree,
$(\B,\mu)$ cannot have a unit.  In this case we may formally adjoin a unit $1$ to $\B$: $\B'=\unit\oplus \B$, with $\eta$ be the inclusion  of $\unit$  and $\proj$ the projection to $\B$. We extend $\mu$ in the obvious way,
and set $\Delta(1)=1\otimes 1$, making $\B'$ into a unital bi--algebra.
In the full detail: $1=id_{\unit}\in Hom(\unit,\unit)$ which is the ground ring/field.
In the free construction, we think of $\unit$ as the tensors of length $0$.

In the Feynman
category interpretation, see \cite{HopfPart2}, $1=id_{\unit}$ where $\unit$ is the empty word.

\subsubsection{Co--unit and multiplicativity}

We will denote putative co--units on $\B$ by $\eps_{tot}:\B\to \unit$ and decompose $\eps_{tot}=\sum_{k\geq 1}\eps_k$ according to the direct sum decomposition on $\B$: $\eps_k:\Coop(k)\to \unit$ extended to zero on all other components.
We will also use the truncated sum $\eps_{\geq p}=\sum_{k\geq p} \eps_k$ which is set to $0$ on all $\Coop(k)$ for $k<p$.

\begin{rmk} There is a 1--1 correspondence between (left/right) co--units on $\B$ and on $\B'$. This is given by adding $\eps_0$ on the identity component via the definition $\eps_0\circ \eta=id$ and vice--versa truncating the extended sum $\eps_{tot}=\sum_{k\geq 0}\eps_k$ at $k=1$.
\end{rmk}

 A family of morphisms $\eps_k:\Coop(k)\to \unit$ is called multiplicative if $\kappa\circ(\eps_k\otimes \eps_l)=\eps_{k+l}\circ \mu$, where $\kappa:\unit\otimes\unit\to \unit$ is the unit constraint --- e.g.\ multiplication in the ground field in case we are in $\kVect$ --- which we will omit from now on.

\begin{lem}
\label{multlem}
If $\eps_{tot}$ is a co--unit (left or right) then the $\eps_k$ are a multiplicative family.
More generally $\eps_{n_1}\odo \eps_{n_k}=\eps_{\sum n_i} \circ \mu^{k-1}$ and in particular $\eps_1^{\otimes k}=\eps_{k}\otimes \mu^{k-1}$.
If $\eps_k$ is a any multiplicative family and $\eta_1$ is a section of $\eps_1$ then
$\mu^{k-1}\circ \eta_1^{\otimes k}$ is a section of $\eps_k$.

Furthermore $\eps_{tot}$ descends to the associated graded.

\end{lem}

\begin{proof}
The first statement is equivalent to $\eps$ being an algebra morphism. The other equations follow readily. Now  $\eps_p(F^{\geq p+1})=0$, since $\Coop(p,p+1)=0$ and hence each $\eps_p$ descends
to $Gr^p$. The sum $\eps_{tot}$ then descends as the sum of the $\eps_p$ with each $\eps_p$ defined on the summand $Gr^p$.
\end{proof}

\subsubsection{Right co--algebra co--units}
\begin{lem}
If $\B$ has a right bi--algebra co--unit $\eps_{tot}$, then $\eps_1$ is a right co--operadic co--unit.
If there are elements of depth greater than one, there can be no left co--operadic co--unit.
\end{lem}

\begin{proof}

For the first statement, we  verify \eqref{rcouniteq} using Lemma \ref{multlem}:
\begin{equation}
\label{co--unitcalceq}
\sum_k (id \otimes \eps_1^{\otimes k})\circ\check\gamma=\sum_k (id \otimes \eps_k)\circ \mu^{k-1}\circ\check\gamma=(id\otimes\eps_{tot})\circ \Delta=id
\end{equation}
If  $\eps_1$ would be a left co--operadic unit, we would need an element of $\Coop(1)$ on the left of $\Delta(a)$, to obtain a non--zero answer. Thus, there can be no left co--unit for elements in $F^{\geq 2}$.
\end{proof}

 A necessary condition
for the existence of a right co--unit for $\B$ is  hence:

\begin{prop} $\eps_{tot}$ is a right bi--algebraic co--unit if and only if $\eps_1$ is a right co--operadic co--unit  which extends to a  multiplicative family $\eps_k$.
\end{prop}
\begin{proof}
This follows by reading equation (\ref{co--unitcalceq}) right to left.
\end{proof}

\subsubsection{Left co--algebra co--units}

\begin{prop}
\label{freestructureprop}
If $\B$ as a co--algebra has a left co--unit $\eps_{tot}$, then $F^{\geq p}=(F^{\geq 1})^{\geq p}$, where the latter denotes the sum of the k--th powers of $F^{\geq 1}$ with $k\geq p$. Moreover, the morphism of co--operads with multiplication and of bi--algebras
$(Gr^1)^{nc}\to \Gr$  given by Corollary \ref{Grcor}   is surjective.
\end{prop}

\begin{proof} The inclusion $F^{\geq p}\supset (F^{\geq 1})^{\geq p	}$ is in Proposition \ref{structureprop}. For the reverse inclusion, let
$a\in F^{\geq p}$, then after applying $(\eps_{tot}\otimes id)\circ \Delta$ we are left with a sum of products of at least $p$ factors and hence the reverse inclusion follows.

In the same way, we see that $\Gr^p=(\Gr^1)^p$ and that the map in question is surjective.

\end{proof}

We recall from \cite{Gerst} that a filtered algebra/ring $(\B,F^{\geq p})$ is {\em pre--developable} if there exists
for each $p$ an additive mapping $q_p:
Gr^p\to F^{\geq p}$ which is a section of $p_p:F^{\geq p}\to Gr^p=F^{\geq p}/F^{\geq p+1}$ i.e.\ $p_p\circ q_p(a) = a$ for all $a\in Gr^p$.
It is {\em developable} if   $\bigcap_p F^{\geq p}=0$,  and the ring is complete in the topology induced by the filtration.

\begin{prop}
If $\B$ has a left co--algebra co--unit then $P_p=(\eps_{\geq p}\otimes id)\circ\Delta$ is a projector to $F^{\geq p}$. Hence the
short exact sequence $0\to F^{\geq p+1}\to F^{\geq p}\to Gr^p\to 0$ splits  and   $\B$ is pre--developable. Under the Assumption \ref{basicas}, $\B$ is developable.
\end{prop}
\begin{proof}
If $\eps_{tot}$ is a left co--algebra co--unit then using multi--Sweedler notation for $a\in \Coop(n):a=(\eps_{tot}\otimes id)\circ \Delta(a)=\sum_{k} \eps_k(a^{(0)}_k)\otimes  a^{(1)}_{n_1}\cdots a^{(1)}_{n_k}=:
\sum_k a_k$ with $a_k$ a product of $k$ factors and hence in $F^{\geq k}$.
 Since $\eps_{\geq p}=0$ on $\Coop(k):k<p$, we see that $P_p(a)=\sum_{k\geq p}^n a_k$ and hence the image of $P_p$ lies
 in $F^{\geq p}$. If on the other hand $a\in F^{\geq p}$ then $a= \sum_k a_k= \sum_{k\geq p} a_k=P_p(a)$, since all lower terms
 do not exist as the summation for $\Delta$ stands at $p$.

Under Assumption \ref{basicas}, due to local finiteness,  the first condition is true, and  since we only use finite sums, the algebra is complete.
\end{proof}
Note that $T_i(a)=[P_{i-1}\cdots P_1(a)]$ gives the development of $a$ in $Gr$ in the notation of \cite{Gerst}.

\begin{cor}
\label{leftco--unitcor}
If $\eps_{tot}$ is a left bi--algebra unit, then for $a\in \Coop(n)\cap F^{\geq p}$ there is a decomposition $a=\sum_{k\geq p}^n a_k$ with each $a_k\in F^{\geq k}$ and (after possibly collecting terms) this gives the development of $a$.\qed
\end{cor}

\begin{cor}
\label{structurecor}
If $\eps_{tot}$ is a left co--algebra co--unit for $\B$, then $\eps_p$ descends to a well defined map $Gr^p\to \unit$.
and on $Gr^p:(\eps_p\otimes id) \circ \Delta_p = id$. Thus $\eps_{tot}$ understood
as acting on  $Gr^p$ with $\eps_p$ is a left co--unit for $Gr$. Furthermore
$(\eps_k\otimes id)\circ \Delta|_{Gr^p}=\delta_{k,p} id$.
\end{cor}
\begin{proof}
First  $\eps_p(F^{\geq p+1})=0$, since $\Coop(p,p+1)=0$. The statements then follows from the development.
\end{proof}

It is known \cite{Gerst} that if $\B$ is developable then $Gr$ is a deformation of $\B$. Coupled with the results above one has:
\begin{thm}
\label{defthm}
if $\B$ has a left co--algebra co--unit,
then  $\B$ is a deformation $Gr$, which is a quotient of the free construction on $Gr^1$.
\end{thm}
\qed

\subsubsection{Units and co--units for  the free case $\Coop^{nc}$}
\label{gradedfree}
In this section, we let $\Coop$ be a co--operad and consider $\Coop^{nc}(n)=\bigoplus_k\bigoplus_{(n_1,\dots,n_k):\sum_i n_i =n}\Coop(n_1)\odo \Coop(n_k)$ and its  bi--algebra $\B=\bigoplus \Coop^{nc}(n)$.

\begin{prop}
\label{freeunitprop}
The bi--algebra $\B=\bigoplus_n\check\O^{nc}(n)$ has a bi--algebraic co--unit if, and only if, $\Coop$ has a co--operadic co--unit.
\end{prop}

\begin{proof}
We already know that  a right co--operadic co--unit for $\check\O^{nc}$ is necessary. This yields a right co--operadic co--unit for $\Coop$ by restriction to $Gr^1=\Coop$. Then for $a\in \Coop=Gr^1$
$a=\eps_1\otimes id\circ \Delta(a)=\sum_k \eps_1\otimes id^{\otimes k}\circ \check \gamma$, since all terms with $k\neq 1 $ vanish and for the term with $k=1$ $\Delta=\check\gamma$. Thus
$\eps_1$ is also a left co--operadic co--unit for $\Coop$. We stress for $\Coop$ not for $\Coop^{nc}$.

Now assume that $\eps_1$ is a  co--operadic co--unit for $\Coop$. It follows that $\eps_1$ is a right co--operadic co--unit for $\Coop^{nc}$ by compatibility. Now since $\mu=\otimes$: the extension $\eps_k=\eps_1^{\otimes k }$ is multiplicative and hence a right bi--algebra co--unit. It remains to check whether it is bi--algebraic, which reduces to checking that it is a left co--algebraic unit. The multiplicativity is clear, so, we only need to check on $Gr^1$, that is for all $a\in \Coop^{nc}(n,1)=\Coop(n)$. On the $\Coop(n)$  the equation says exactly that $\eps_1$ is a left co--operadic unit for $\Coop$.
\end{proof}

\subsubsection{Co--units summary}
\label{unitsummarypar}
If $\B$ comes from $\Coop^{nc}$ then having a bi--algebra unit $\eps_{tot}$ is equivalent
to $\eps_1$ being a co--operad co--unit on $\Coop$.

In general, for $\B$ to have a bi--algebra co--unit, it is necessary, that

\begin{enumerate}
\item
$\eps_1$ is a right co--operadic co--unit.
\item $F^{\geq p}=(F^{\geq 1})^{\geq p}$ .
\item
$P_k=(\eps_{\geq k} \otimes id)\circ \Delta$ are projectors onto $F^{\geq k}$.
\item $\B$ is developable and a deformation of the associated graded $Gr$
\end{enumerate}
On the associated graded $Gr$. If $\eps_{tot}$ is a putative bi--algebra co--unit
\begin{enumerate}
\setcounter{enumi}{4}
\item $\eps_p$ is uniquely determined from $\eps_1$.
\item Lifted to $(Gr^1)^{nc}$, $\eps_1$ is a co--operadic unit, which ensures that the lift of $\eps_{tot}$ is a bi--algebra unit.
\item For $\eps_{tot}$ to descend to $Gr$, it needs to vanish on the
kernel of the, by (2) surjective, map $\mu^{\otimes p-1}:(Gr^1)^{\otimes p}\to Gr^p$.
\end{enumerate}

The first statement holds by Proposition \ref{freestructureprop} and Corollary \ref{structurecor}  which says that $Gr^p=(Gr^1)^p$ and  hence Lemma \ref{multlem} determines $\eps_p$.
Since co--units are multiplicative, they lift  via Proposition \ref{freeunitprop}.

\begin{df}
In general, we say that a co--operadic right co--unit $\eps_1$ is {\em bi--algebraic}, if it extends to  a bi--algebraic co--unit $\eps_{tot}$ for $\B$. If such an $\eps_{tot}$ exists, we will call $\Coop$ bi--algebraic.
 \end{df}

\subsection{The pointed case}\label{2nd-inter}

\begin{df}
A co--operad  $\Coop$ with a right  co--operadic co--unit  $\eps_1$ is called  {\em pointed}
if the co--unit $\eps_1$ is split, i.e.\  there is a section $\eta_1:\unit\to \Coop(1)$ of $\eps_1$.

We call $\Coop$ {\em reduced} if it is pointed and $\eta_1$ is an isomorphism $\unit \simeq \Coop(1)$; it is then automatically pointed.

A bi--algebra unit will be called split and $\B$ pointed if the associated right co--operadic unit $\eps_1$ is split.
\end{df}

We will denote $|:=\eta_1(1)$. For pointed co--operads Lemma \ref{multlem} applies, and we split each $\Coop(n)=\unit\oplus \bar \Coop(n)$
where $\bar\Coop(n)=ker(\eps_n)=ker(\eps_{tot}|_{\Coop(n)})$ and  $\unit$ is the component of $|^n$.
We set $\bar B=\bigoplus \bar\Coop(n)$. Notice that this is smaller than the augmentation ideal $\Bred=ker(\eps_{tot})$.

\begin{ex}
The connection to the free construction of \S\ref{subsect:free} is as follows:  $\Coop^{nc}$ for a unital operad $\O$ is  pointed if  the unit morphism $u:\unit \to \O(1)$ split via a morphism $c$.
The element $|$ is then  the dual element to  the unit $u(1)\in \O(1)$. Here $|=\check c(1)=\eta_1(1)$
and being the dual element means that $\check u(|)=\eps_1\circ\eta_1(1)=(c\circ u)^{\vee}(1)=1$.
All of the examples of \S\ref{examplesec} have this property.
\end{ex}

\begin{lem} If $\B$ has a  split bi--algebraic  co--unit, then have $\Delta(|)=|\otimes |+\bar\Delta(|)$ with $\bar\Delta(|)\in \bar\Coop(1)\otimes \bar\Coop(1)$
and hence $\Delta(|^p)=|^p\otimes |^p +$ terms of lower order in $|$.
Thus the image of $|^p$ is not $0$ in $Gr^p$ and we can split $Gr^p=\unit \oplus \overline{Gr}^p $ where $\unit$ is the component if the image of $|^p$.
\end{lem}
\begin{proof}
The first statement follows since $\eps_{tot}$ is a bi--algebraic unit. The second statement follows, from the bi--algebra compatibility  condition.
\end{proof}

\begin{prop}
\label{pointedprop}
Let $\Coop$ be a co--operad with multiplication and  a pointed  bi--algebraic co--unit on $\B$, then

\begin{eqnarray}
\Delta(|)&=&|\otimes |+ \bar\Delta(|)\text{ with }\nn\\
\bar\Delta(|)&\in &\bar\Coop(1)\otimes \bar\Coop(1)\\
\Delta(|^p)&=&|^p\otimes |^p+\bar\Delta(|^p) \text{ with }\nn\\
\bar\Delta(|^p)&\in &\bar\Coop(p)\otimes \bar\Coop(p)
\end{eqnarray}

And for $a\in \bar \Coop(n)\cap F^{\geq p}$
\begin{eqnarray}
\label{reducedeq}
\Delta(a)&=&\sum_{k\geq p}^{n} |^k \otimes a_k+ a\otimes |^n + \bar\Delta(a) \text{ with }\nn\\
&&a_k\in \bar\Coop(n), \bar\Delta(a)\in \bar\B\otimes \bar\Coop(n)
\end{eqnarray}
with $a=\sum_{k\geq p}^n a_k$ and the $a_k$ are as in Corollary \ref{leftco--unitcor}.

Likewise, in the associated graded case, for $a\in \bar \Coop(n,p)$
\begin{eqnarray}
\label{gradedreducedeq}
\Delta(a)&=&|^p \otimes a + a\otimes |^n + \bar\Delta(a) \text{ with }\nn\\
&&\bar\Delta(a)\in \bar\Gr\otimes \bar\Gr
\end{eqnarray}

Again, if these equations hold having a bi--algebraic co--unit $\eps_{tot}$ is equivalent to $\eps_1$ being a right co--operadic co--unit.

\end{prop}

\begin{proof}
Using Corollary \ref{leftco--unitcor}  and applying $\eps_{tot}$ on the left, we obtain the first term. Applying $\eps_{tot}$ on the right yields the second term. These are different if $a\neq |^k$ for some $k$. In the case $a=|^k$, the equation follows from the lemma above. In general,
the remaining terms lie in the reduced space.
Replacing $\B$ with $Gr$ proves the rest.
\end{proof}

We also get a practical criterion for a bi--algebra co--unit.
\begin{cor}
Conversely,
assume the equations in Propositions \ref{pointedprop} hold, then
having a  bi--algebraic co--unit $\eps_{tot}$ is equivalent to $\eps_1$ being a right co--operadic co--unit.
\end{cor}
\begin{proof}
By Lemma \ref{multlem}, we see that $\eps_k$ is the projection to the factor $|^k$ of $\Coop(k)=\unit\oplus \bar\Coop(k)$ and on that factor it is $\eps_1^{k}\circ \mu^{k-1}$ and hence determined by $\eps_1$. Now the second term of \eqref{reducedeq} is equivalent to $\eps_{tot}$ being
a right bi--algebra co--unit. Furthermore, since this is the term relevant for the right co--operad co--unit,
we obtain the equivalence for the right bi--algebra co--unit.
Similarly, applying the given $\eps_{tot}$ as a potential left bi--algebra co--unit, we see that having a left bi--algebra co--unit is equivalent to $a=\sum_k a_k$, i.e.\ the first term in \eqref{reducedeq}.
\end{proof}

\subsection{Hopf Structure}
In this section, unless otherwise stated, we will assume that $\Coop$ is a
co--operad with multiplication and a split bi--algebraic co--unit and assume Assumption \ref{Hopfassump}.

\begin{df}
\label{almostconnectedcoopdef}
We call a pointed
co--operad $\Coop$ with multiplication and bi--algebraic co--unit $\eps_{tot}$
{\em  connected} if
 \begin{enumerate}
  \item The element $|$ is group--like: $\Delta(|)=|\otimes |$.
\item  $(\Coop(1),\eta_1,\eps_1)$ is connected as a co--algebra.
\end{enumerate}
\end{df}
Notice that a reduced $\Coop$ is automatically  connected, but this is not a necessary condition.
In the free case, the co--operad $\Coop^{nc}$ is connected if $\Coop$ is.
If we start with an operad $\O$ as in \S\ref{operadpar},  the co--operad $\Coop^{nc}$ is  connected if $\O$ is co-connected.

\begin{rmk}
Notice that for a connected  co--operad $\{\Coop(n)\}$ the bi--algebras $\B=\bigoplus \Coop(n)$ and $\B'=\unit \oplus \B$ are usually not connected,
since  all powers $|^k$ are group--like:
 $\Delta(|^k)=|^k\otimes |^k$, $\eps_{tot}(|^k)=1$.
\end{rmk}

For a pointed co--operad  with multiplication,
let $\I$ be the two-sided ideal spanned by $1-|$.
Set
\begin{equation}
\H:=\B'/\I
\end{equation}
Notice that in $\H$ we have that $|^{k}\equiv 1 \mod \I$ for all $k$.

\begin{prop}
If  $\Coop$ is connected, then $\I$ is a co--ideal and hence $\H$ is a co--algebra. The unit $\eta$ descends to a unit $\bar \eta:\unit\to \H$
and the co--unit $\eps_{tot}$ factors as $\bar\eps$  to make $\H$ into a bi--algebra.
\end{prop}
\begin{proof}
Analogous to Proposition \ref{Iprop}.
\end{proof}

\begin{thm}
If  $\Coop$ is  connected as a co--operad then $\H$ is co--nilpotent and hence admits a unique structure of  Hopf algebra.
\end{thm}

\begin{proof} Analogous to Theorem \ref{hopfthm}
\end{proof}

\subsection{Symmetric  and infinitesimal versions, deformations, amputation and grading}
Assuming that we have a co--operad with multiplication and a split bi--algebraic co--unit, we get the analogous results to \S\S\ref{defsec}, \ref{freeinfsec}, \ref{freeampsec} and \ref{gradingpar}, see below.

\subsubsection{Symmetric version}
We now assume that the co--operad $\Coop$ is symmetric and in finite sets. We then have the same diagram as \eqref{coprodsqeq}.
\begin{df}
A  co-operad with multiplication in finite sets, is a co--operads in finite sets with  multiplications
$\mu_{S,T}:\Coop(S)\otimes \Coop(T)\to\Coop(S\sqcup T)$, such that the following diagram commutes.
\begin{equation}
\label{multequieq}
\xymatrix{
\Coop(S)\otimes \Coop(T)
\ar[d]_{\sigma\otimes \sigma'}\ar[r]^{\mu_{S,T}}& \Coop(S\sqcup T)\ar[d]^{\sigma\sqcup\sigma'}\\
\Coop(S')\otimes \Coop(T')\ar[r]^{\mu_{S',T'}}&\Coop(S'\sqcup T')
}
\end{equation}
and the analogue of \eqref{oureqn} holds equivariantly.
\end{df}

\begin{lem}
For a co--operad with multiplication in finite sets the co--operad structure and the multiplication descend to the coinvariants.
\end{lem}
\begin{proof}
The well--definedness of $\check\gamma$ on the coinvariants is guaranteed by \eqref{coprodsqeq}.
Due to the diagram \eqref{multequieq} the multiplication descends to the coinvariants as well. This means that $\Delta$ is well defined and it follows that it is co--associative.
Finally, since \eqref{oureqn} holds equivariantly, $\Delta$ is compatible
\end{proof}

Set $\B=\bigoplus_n \Coop(n)_{\SS_n}$. A bi--algebraic co--unit $\eps$ is called invariant
if for all $a_S\in \Coop(S)$ and any isomorphism $\sigma:S\to S'$, $\eps\circ \sigma=\eps$.

Retracing the steps of the non-$\Sigma$ version and \S\ref{operadpar}, the following is straight--forward:

\begin{prop}
With the assumption above, $\B$ is a non-unital, non-co--unital,  bi--algebra.
If we furthermore assume that an invariant bi--algebraic co--unit for $\B$ exists then $\B'=k\oplus\B$ is a unital and co--unital
bi--algebra and $\bar\H:=\B/ \bar{\mathcal {I}}$, where $\bar{ \mathcal{I}}$ is the image of $\mathcal{I}$ in $\B'$, is a bi--algebra, that is if it is connected a commutative Hopf algebra.
\qed
\end{prop}

\subsubsection{The free example}
 In the free example, starting with a symmetric operad, we do not only have to take the sum, but also induce the representation to $\SS_n$ in
order to obtain a symmetric co--operad with multiplication. Let
\begin{equation}
\label{ncsymeq}
\Coop^{symnc}(n)=\bigoplus_{k} \bigoplus_{(n_1,\dots,n_k):\sum_i n_i=n}Ind_{(\SS(n_1)\times \dots \times \SS(n_k))\wr \SS(k)}^{\SS_n}\Coop(n_1)\odo \Coop(n_k)
\end{equation}

\begin{rmk}
When taking coinvariants, this induction step is canceled and we only have to take coinvariants with respect to
$(\SS(n_1)\times \dots \times \SS(n_k))\wr\SS(k)$:

\begin{multline}
\B=\bigoplus  \Coop^{symnc}(n)_{\SS_n}= \\
\bigoplus_{k} \bigoplus_{(n_1,\dots,n_k):\sum_i n_i=n}\left(\Coop(n_1)_{\SS_{n_1}}\odo \Coop(n_k)_{\SS_{n_k}}\right)_{\SS_k}=\\
\bigoplus_{(n_1,\dots,n_k):\sum_i n_i=n}\Coop(n_1)_{\SS_{n_1}}\odot \cdots \odot \Coop(n_k)_{\SS_{n_k}}
\end{multline}
where $\odot$ is the symmetric product.
\end{rmk}

\begin{prop} The
$\Coop^{symnc}(n)$ form a symmetric co--operad with multiplication and
$\B=\bigoplus  \Coop^{symnc}(n)_{\SS_n}$ forms a bi--algebra, and if $\Coop$ has an operadic co--unit, then $\B'$ is a unital an non--unital bi--algebra. Furthermore if $\Coop(1)$ is  almost connected, then the quotient $\B'/\I$ is a Hopf algebra.
\end{prop}
\begin{proof}
It is clear that the free multiplication  also satisfies \eqref{multequieq} and the equivariant version of \eqref{oureqn} holds. A co--unit for a symmetric co--operad is by definition a morphism $\Coop(\{s\})\to k$ that is invariant under isomorphism, hence so is its extension. The rest of the statements are proved analogously to the non--symmetric case.

\end{proof}

\subsubsection{Infinitesimal version}

\begin{df}
A (non--\SSigma) \emph{pseudo-co--operad with multiplication} $\mu$ is a (non--\SSigma) pseudo-co--operad (cf.\ \S\ref{pseudocooperadpar}) $\Coop$  with a family of maps
\begin{equation}
\mu_{n,m}:\Coop(n)\otimes \Coop(m)\to \Coop(n+m) \quad  n,m\geq0
\end{equation}
which satisfy the equation
\begin{equation}
\label{infalgeq}
\delta\circ\mu\;\;=\;\;(\mu\otimes id)\circ \pi_{23}\circ(\delta \otimes \id)+(\mu\otimes \id)\circ(\id\otimes \delta)
\end{equation}
where  $\delta:=\check\circ$ is the  co--multiplication  \eqref{deltaeq}.
\end{df}

\begin{prop}
If $\Coop$ is a (non-\SSigma{}) co--operad with multiplication and multiplicative
right co--operadic co--unit
then the multiplication is also compatible with (the non-\SSigma{})
pseudo-co--operad structure.
\end{prop}
\begin{proof}
Straightforward using equation \eqref{infalgeq}.
\end{proof}

\begin{prop}
If the co--operad $\Coop$ is  connected, then in the Hopf quotient, the co-pre--Lie structure induces a co-Lie algebra structure on the indecomposables $\H_>/\H_>\H_>$, where $\H_>$ is reduced version of $\H$.
\end{prop}

\begin{proof}
Analogous to the proof of Theorem \ref{indecthm}.
\end{proof}

\begin{ex}
In the free case $\Coop^{nc}$, the indecomposables are precisely given by $\Coop$ and the co-pre--Lie structure is $\check\circ$. If $\Coop$ is the dual of $\O$ then  the co-Lie structure corresponds dually to  the usual Lie structure of Gerstenhaber.
\end{ex}

\subsubsection{Deformations}

Let $\C$ be the ideal spanned by $|a-a|$ of $\B$.  This is again a co--ideal by the calculation of Proposition \ref{Cprop}.  This ideal is trivial in the symmetric case. Denote by $q$ the image of $|$ under $\pi:\B\to \B/\C=:\H_q$.

\begin{thm}
Let $\Coop$ be a   co--operad with multiplication and a split bi--algebraic co--unit, then
 \begin{equation}
 \H_q(d)\simeq \bigoplus_{n\leq d} q^{d-n} \bar\B(n)
 \end{equation}
and $\H_q$ is a deformation of $\H$ given by $q\to 1$.
\end{thm}
\begin{proof}
Analogous to Proposition \ref{Hqdecompprop} and Corollary \ref{defcor}.
\end{proof}

\subsection{Grading}

\begin{prop}
For an almost connected co--operad with multiplication, the depth filtration descends to $\H_q$, $\H$ and $\H^{\it amp}$ in both the symmetric and non--$\Sigma$ case. It satisfies
$\bar \Delta_n(F^{\geq p})\subset \bigoplus F^{\geq p}\otimes F^{\geq 1}$. The depth of $1\in \H$ is $0$.
\end{prop}
\begin{proof}
Since $|$ is group--like $\Delta(a|)=\Delta(a)(|\otimes |)$, it is clear that the depth filtration descends to $\H_q$.
Any lift of $a\in \H$ to $\H_q$ is of the form $aq^k$. We define the depth of $a\in \B$ to be the minimal depth of a lift or equivalently for any lift the difference between the depth and the $q$ degree. This gives the element $1$ depth $0$.
The relation then follows from Proposition \ref{structureprop} and Lemma \ref{barlem}.
The fact that the filtration descends through amputation is clear.
\end{proof}

\section{A class of examples: Co--operads from simplicial objects}
\label{simplicialpar}

In this section, we present an important (but accessible) construction of some co--operads with multiplication. This construction is best expressed in the language of simplicial objects, and so we will first recall some of the basic notions. Some of the examples, however, can be understood with no simplicial background.  For an arbitrary set $S$, we will see that the set $X$ of all sequences or words in $S$ has the structure of a co--operad, and Goncharov's Hopf algebra may be obtained from the case $S=\{0,1\}$.  Elements of $X$ can be understood as strings of consecutive edges in the complete graph (with vertex loops) $K_S$, and further geometric intuition can be obtained by considering also strings of triangles or more generally $n$-simplices. The way to encode this construction is to think of the graph $K_S$ as defining a groupoid $G(S)$, i.e.\ a category whose morphisms are invertible. The set of objects is $S$ and for any pair of objects there is a unique invertible morphism between them. The transition to the simplicial setting is then made by considering the nerve of this category.

In fact, our construction defines a co--operad with multiplication, and hence a bi--algebra (or Hopf algebra) for any (reduced) simplicial set $X$, see Proposition~\ref{bialgfromsimp}. In this guise, we also recover the Hopf algebra of Baues.

\subsection{Recollections: the simplicial category and simplicial objects}

Let $\Delta$ be the small category whose objects are the finite non-empty ordinals $[n]=\{0<1<\dots<n\}$ and whose morphisms are the order-preserving functions between them. Of course, each $[n]$ can itself be regarded as a small category, with objects $0,1,\dots,n$ and a (unique) arrow $i\to j$ if and only if $i\leq j$, and order preserving functions are just functors. Thus $\Delta$ is a full subcategory of the category of small categories.

Among the order-preserving functions $[m]\to[n]$ one considers the following generators: the injections
${\partial^i}:[n-1]\to[n]$ which omit the value $i$, termed co--face maps, and the surjections
${\sigma^i}:[n+1]\to[n]$ which repeat the value $i$, termed co--degeneracy maps. These maps satisfy certain obvious co--simplicial relations.

For $D$ a small category, and $\C$ any category, we can consider the  contravariant functors or the covariant functors $X$ from $D$ to $\C$. For $D=\Delta$ these are termed the simplicial and the co--simplicial objects in $\C$.
A functor $D^{op}\to\Set$ is  \emph{representable} if it is $\hom_D(-,d)$ for some object $d$.  In general, such functors are also called pre--sheaves on $\D$. If $\D$ is monoidal then so is the category of pre--sheaves, with the product given by Day convolution.
The Yoneda Lemma gives a bijection between the set of natural transformations $\hom_D(-,d)\to X$ and the set $X(d)$, and in particular $d\mapsto\hom_D(-,d)$ defines a full embedding $y$ of $D$ into the functor category $\Set^{D^{op}}$. This category together with the embedding $y$ is also called the co--completion and has the universal property that any functor from $D$ to a co--complete category (one that contains all colimits) factors through it.

The following result is central to the classical theory  and in particular for us it will yield the construction of a nerve of a small category.
\begin{lem}\label{realnerve} Let $D$ be a small category and $\C$ a co--complete category.
Any functor $r:D\to\C$ has a unique extension along the Yoneda embedding to a functor $R:\Set^{D^{op}}\to C$ with a right adjoint $N$,
\begin{equation}\xymatrix@C+3em@R+1em{D\;\ar@{^{(}->}[r]^{y}\ar[rd]_r&\ar@/_1ex/[d]_R\Set^{D^{op}}\!\!\!\!\!\!\!\\&\C.\ar@/_1ex/[u]^{\dashv}_N}\end{equation}
If $r:D\to\C$ is a monoidal functor between monoidal categories, then $R$ sends monoidal functors $D^{op}\to\Set$ to monoids in $\C$.
\end{lem}

The functor $R$ is sometimes denoted $(-)\otimes_D r$, where the tensor over $D$ is thought of as giving an object of $D$ for every pair of $D^{op}$- and $D$-objects in $\C$, analogously to the language of tensoring left and right modules or algebras over a ring.
The right adjoint $N$ is termed the \emph{nerve}, and is given on objects by
\begin{equation}
N(C)=\hom_{\C}(r(\text{-}),C).
\end{equation}

Now a simplicial object is determined by the sequence of objects $X_n$, and the face and degeneracy maps $d_i:X_n\to X_{n-1}$ and $s_i:X_n\to X_{n+1}$, given by the images of $[n]$, and $\partial^i$ and $\sigma^i$, and dually for co--simplicial objects.
Maps $X\to Y$ of (co)simplicial objects, that is, natural transformations, are just families of maps $X_n\to Y_n$ that commute with the (co)face and (co)degeneracy maps.

We write $\Delta[n]$ for the representable simplicial set $\hom_\Delta(-,[n])$ so, by Yoneda, simplicial maps $\Delta[n]\to X$ are just elements of $X_n$ and maps $\Delta[m]\to\Delta[n]$ are just order preserving maps $[m]\to[n]$. For such a map $\alpha$ we use the notation $\alpha^*=X(\alpha):X_n\to X_m$ and
\begin{equation}x_{(\alpha_0,\dots,\alpha_m)}\in X_m\end{equation} to denote the image under $\alpha^*$ of an $n$-simplex $x$ in a simplicial set $X$.

If $D=\Delta$ and $X$ is a simplicial set then $R(X)$ is usually called the \emph{realization} of a simplicial set with respect to the models $r$.
Considering for example the embedding $r:\Delta\to\cat$ we obtain the notion of the simplicial nerve of a category:
for  $C$  a small category, there is a natural bijection between the functors from $[n]$ to $C$ and the $n$-simplices of
the nerve $NC$,
\begin{equation}
N(C)_n=\hom_{\cat}([n],C).
\end{equation}
\begin{example}\label{simplicialwords}
Let $S$ be a set, and let $X(S)$ be the simplicial set given by the nerve of the {\em contractible $G(S)$ with object set $S$},

\begin{equation}
X(S)=NG(S).
\end{equation}
If $S=[n]$, for example, we may identify $G(S)$ with the fundamental groupoid of $\Delta[n]$, and
\begin{equation}X([n])\cong N\pi_1\Delta[n].\end{equation}
Giving a functor from $[n]$ to the contractible groupoid $G(S)$ is the same as giving the function on the objects, so an $n$-simplex of $X(S)$ is just a sequence of $n+1$ elements of $S$,
\begin{equation}X(S)_n\;\;=\;\; S^{n+1}\;\;=\;\;\{\,(a_0;a_1,a_2,\dots,a_{n-1};a_n)\,:\,a_i\in S\,\}.\end{equation}
In the case $S=\{0,1\}$, the groupoid $G(S)$ is
\begin{equation}\xymatrix
@C+=1.5cm{*+[r]{0}
\ar@(ul,l)_(0.42){\text{}}[]
\ar@<0.6ex>@/^0.9ex/^-{}[r]&\ar@<0.6ex>@/^0.8ex/^-{}[l]1
\ar@(ur,r)^(0.42){\text{}}[]}
\end{equation}
and the $n$-simplices of $X$ are words of length $n+1$ in the alphabet $\{0,1\}$.
\end{example}

\subsection{The operad of little ordinals}

The category of small categories, and the category of simplicial sets, can be regarded as monoidal categories with the disjoint union playing the role of the tensor product, and the initial object $\varnothing$ the neutral object. In this context, we have the following result, compare for example \cite[Example 3.6.4]{DyKa}.

\begin{figure}
    \centering
    \includegraphics{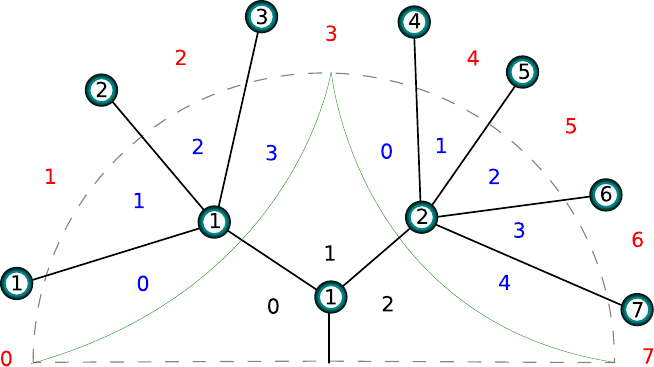}
    \caption{An example of a factorization
$\underline{7}\to\underline{2}\to\underline{1}$ of order preserving surjections and,
reading outwards from the root to the leaves,
the corresponding operad structure map
$\gamma_{3,4}:[2]\cup{\color{blue}{[3]\cup[4]}}\to\color{red}[7]$.}
    \label{fig:2347}
\end{figure}

\begin{prop}\label{deltaoperad}
The sequence of finite nonempty ordinals $([n])_{n\geq 0}$
forms an operad in the category of small categories.
For any partition $n=m_1+m_2+\dots+m_k$,
consider the subset $\{0=n_0<n_1<n_2<\dots<n_k=n\}$ of $[n]$  given by $n_r=m_1+\dots+m_r$. Then
the structure map \begin{equation}\gamma_{m_1,\dots,m_k}=(\gamma^0,\,\gamma^1,\dots,\gamma^k):[k]\cup[m_1]\cup\dots\cup[m_k]\to [n]\end{equation}
is defined by
\begin{equation}
\gamma^0(i)=n_i\quad(0\leq i\leq k)\text{ and }\;
\gamma^r(j)=n_r+j\quad(0\leq j\leq m_r,1\leq r\leq k).
\end{equation}
This operad clearly has a unit $u:\varnothing\to [1]$.
\end{prop}

This construction is related, via Joyal duality (see Appendix C),
to
the factorisations of maps $\underline{n}\to\underline{1}$ into order preserving surjections $\underline n\to\underline k\to\underline 1$,
where
$\underline n=\{1,\dots,n\}$.
Under the Joyal duality between end-point preserving ordered maps ---see Appendix C--- $[k]\to[n]$ and ordered maps $\underline{n}\to\underline k$,
the injection $\gamma^0:[k]\to[n]$ defined in the Proposition corresponds to the order preserving surjection $\underline n\to\underline k$ whose fibres over each $i$ have cardinality $m_i$ (see Figure~\ref{fig:2347}).

\medskip

The image of the operad structure in Proposition~\ref{deltaoperad} under the Yoneda embedding
gives:
\begin{cor}
The collection of representable simplicial sets $(\Delta[n])_{n\geq0}$ forms a unital operad in the category of simplicial sets.
\end{cor}

If $X$ is a simplicial set, then the unital operad structure on the sequence $\Delta[n]$, $n\geq 0$, induces a co--unital co--operad structure on the sequence $X_n=\hom(\Delta[n],X)$.
That is, the sequence $(X_n)_{n\geq0}$ forms a co--unital co--operad with
\begin{align}
\nonumber
X_n\xrightarrow{\displaystyle\check\gamma_{m_1,\dots,m_k}}
&\qquad\; X_k \;\quad  \times \quad\; X_{m_1}\;\;\times\;\;\dots\;\;\times\;\; X_{m_k}\\
\label{simpcoopformula}x\;\;\quad\longmapsto\;\quad&\left(x_{(n_0,n_1,\dots,n_k)},\,x_{(n_0,n_0+1,\dots,n_1)},\dots,x_{(n_{k-1},n_{k-1}+1,\dots,n_k)})
\right)
\end{align}
where $0=n_0<n_1<n_2<\dots<n_k=n$ are  given by $n_r=m_1+\dots+m_r$ as usual.
The co--unit is given by the unique map
$X_1\to \{*\}$.

More generally:

\begin{cor}\label{co--operadsimplicial}
Let $X$ be a simplicial object in a category $\C$ with finite products. Then the sequence $(X_n)_{n\geq0}$ forms a co--unital co--operad in $\C$.
\end{cor}

\begin{example}\label{Gongamma}
The set of all words in a given alphabet $S$ is naturally a simplicial set (see Example \ref{simplicialwords} above) and so by Corollary \ref{co--operadsimplicial} it forms a co--unital co--operad $X$ in the category of sets. The elements of arity $n$ in this co--operad are the words of length $n+1$ in $S$,
\begin{equation}X_n\;\;=\;\; S^{n+1}\;\;=\;\;\{\,(a_0;a_1,a_2,\dots,a_{n-1};a_n)\,:\,a_i\in S\,\}\end{equation}
and the operation $\check\gamma_{m_1,\dots,m_k}$ sends such an element  $(a_0;a_1,a_2,\dots,a_{n-1};a_n)$ to
\begin{equation}
\left((a_{n_0};a_{n_1},\dots;a_{n_k}),\;(a_{n_0};a_{n_0+1},\dots;a_{n_1}),\dots,(a_{n_{k-1}};a_{n_{k-1}+1},\dots;a_{n_k})
\right)
\end{equation}
where $n_0=0$, $n_k=n$ and $n_r-n_{r-1}=m_r$.
\end{example}

This construction can also be carried out in an algebraic setting.
\begin{prop}
Let $X$ be a simplicial set, and let $\Coop(n)$ be the free Abelian group on the set $X_n$, for each $n\geq 0$. Then $\Coop$ forms a co--unital co--operad in the category of Abelian groups, with the co--operadic structure given by
\begin{align*}
\nonumber
\Coop(n)&\stackrel{\displaystyle\check\gamma\;}\longrightarrow
\quad \Coop(k) \qquad  \otimes \qquad \Coop(m_1)\;\;\otimes\;\;\dots\;\;\otimes\;\; \Coop(m_k)\\
x&\longmapsto
x_{(n_0,n_1,\dots,n_k)}\;\:\otimes\;\; x_{(n_0,n_0+1,\dots,n_1)}\otimes\;\dots\;\otimes x_{(n_{k-1},n_{k-1}+1,\dots,n_k)})
\end{align*}
and the co--unit given by the augmentation
\begin{equation}
\Coop(1)\longrightarrow\mathbb{Z}.
\end{equation}
\end{prop}
\begin{proof}
This follows by applying free Abelian group functor (which carries finite cartesian products of sets to tensor products) to the co--operad structure considered in \eqref{simpcoopformula}.
\end{proof}

From section \ref{subsect:free}
we therefore have

\begin{prop}\label{bialgfromsimp}
Let $X$ be a simplicial set.
The co--operad structure $\Coop$ on $(\mathbb{Z}X_n)_{n\geq1}$ of the previous proposition extends to a structure of a co--operad with (free) multiplication, and hence to a graded bi--algebra structure, on the free tensor algebra
\begin{equation}
\B(X)=\bigoplus_n\Coop^{nc}(X)(n)=\bigoplus_{n_1,n_2,\dots\geq 1} \bigotimes_i\mathbb{Z}X_{n_i}
\end{equation}
generated by $X$, where elements of $\Z X_n$ have degree $n-1$.
\end{prop}

\subsubsection{Goncharov's first Hopf algebra}

Let $S$ be the set $\{0,1\}$. We considered in Example~\ref{simplicialwords} the contractible groupoid $G(S)$ with object set $S$, and the simplicial set $X=X(S)$ given by its simplicial nerve. If we denote the simplices of $X_n$ by tuples $(a_0;a_1,\dots,a_{n-1};a_n)$ as in Example \ref{Gongamma} and apply Proposition \ref{bialgfromsimp} we obtain a graded bi--algebra
\begin{equation}
\B(X)=\Z[(a_0;a_1,\dots,a_{n-1};a_n);a_i\in\{0,1\}]
\end{equation}
with the co--product that sends a polynomial generator $(a_0;a_1,\dots,a_{n-1};a_n)$ in degree $n-1$ to
\begin{equation}
\sum_{0=n_0<n_1<\dots<n_k=n}(a_{n_0};a_{n_1},\dots;a_{n_k})\;\otimes\; \prod_{i=0}^{k-1}
(a_{n_i};a_{n_i+1},\dots;a_{n_{i+1}})
\end{equation}
When we identify all generators in degree 0 we obtain Goncharov's connected graded Hopf algebra $\HG$, as in Theorem \ref{gone}.

For any simplicial set $X$, let $C_n(X)$ be the free Abelian group on the $n$-simplices $X_n$. This defines a chain complex $(C(X),d_X)$ where
\begin{equation}
d_X(x)=\sum_{i=0}^n(-1)^id_ix.
\end{equation}
Diagonal approximation makes $CX$ a differential graded co--algebra,
\begin{equation}
C(X)\longrightarrow C(X\times X)\longrightarrow CX \otimes CX
\end{equation}
whose classical cobar construction is the tensor algebra on the desuspension of the reduced co--algebra
\begin{equation}
\Omega CX=(T\Sigma^{-1}\overline{C}X,d_\Omega)
\end{equation}
where the differential $d_\Omega$ is formed from $d_X$ and the co--product. For the moment, however, we merely observe that if one takes the symmetric rather than the tensor algebra then the underlying graded Abelian group is isomorphic to Goncharov's $\HG$.

\subsection{Simplicial strings}
\label{stringsec}

For $(D,\otimes)$ a strict monoidal category, consider $(\Omega' D,\boxtimes)$ the strict monoidal category generated by $D$ together with morphisms $a\boxtimes b\to a\otimes b$ for objects $a,b$ of $D$, subject to the obvious naturality and associativity relations.
In this way a strict monoidal functor on $\Omega' D$ is exactly a (strictly unital) \emph{lax monoidal} functor on $D$: a functor $F$ on $D$ together with maps $Fa\otimes Fb\to F(a\otimes b)$ satisfying appropriate naturality and associativity conditions.

\begin{df}
\label{doublebasedef}
Let $\Delta_{*,*}$ be the strict monoidal category given as the subcategory of $\Delta$ containing just the generic (that is, end-point preserving) maps $[m]\to [n]$, with the monoidal structure $[p]\otimes [q]=[p+q]$ given by identifying $p\in[p]$ and $0\in[q]$.

We define the category of \emph{simplicial strings} $\Omega\Delta$ to be  the strict monoidal category $\Omega'\Delta_{*,*}$.
\end{df}
This agrees with Baues' construction in \cite[Definition 2.7]{BauesMAMS}.
Now a contravariant monoidal functor on the category of simplicial strings is just an oplax monoidal functor on $\Delta_{*,*}^{op}$.
Explicitly, if $\C$ is a category with the cartesian monoidal structure, then to give a monoidal functor $(\Omega\Delta)^{op}\to\C$ is to give a functor $X:\Delta_{*,*}^{op}\to\C$ together with associative natural transformations $\mu_{p,q}=(\lambda_{p,q},\rho_{p,q}):X_{p+q}\to X_p\times X_q$.
Note that $X$ becomes a simplicial object, if we define outer face maps $X_n\to X_{n-1}$ by $d_0=\rho_{1,n-1}$ and $d_n=\lambda_{n-1,1}$. Moreover these determine all maps $\rho_{p,q}$ and $\lambda_{p,q}$ via the naturality conditions $(d_1^{p-1}\times\id)\mu_{p,q}=\mu_{1,q}d_1^{p-1}$ and $(\id\times d_1^{q-1})\mu_{p,q}=\mu_{p,1}d_{p+1}^{q-1}$.
Thus we have:
\begin{prop}\label{oplax}Let $\C$ be a cartesian monoidal category. Then the following categories are equivalent:
\begin{enumerate}
\renewcommand{\theenumi}{\roman{enumi}}
\item The category of simplicial objects in $\C$.
\item The category of oplax monoidal functors $\Delta_{*,*}^{op}\to\C$.
\item The category of monoidal functors $(\Omega\Delta)^{op}\to \C$.
\end{enumerate}
Given a simplicial object $X$, the corresponding oplax monoidal functor is given by the restriction of $X$ to the endpoint preserving maps, with the structure map
\begin{equation}(d_{p+1}^q,d_0^p):X_{p+q}\to X_p\times X_q\end{equation}
\end{prop}

\begin{df}
An \emph{interval object}~\cite{crscobar} (or a \emph{segment}~\cite{BM2})
in a monoidal category $(\D,\otimes,\unit)$
is an augmented monoid $(L,L^{\otimes 2}\xrightarrow{\mu} L,\unit\xrightarrow{\eta} L,L\xrightarrow{\varepsilon} \unit)$ together with an \emph{absorbing object}, that is, $\overline\eta:\unit\to L$ satisfying $\mu(\id_L\otimes\overline\eta)=\overline\eta\varepsilon=\mu(\overline\eta\otimes\id_L)
$, $\varepsilon\overline{\eta}=\id_I$.
\end{df}

To any augmented monoid $L$ one associates a simplicial object or,
under Joyal duality, a covariant functor $L^\bullet$ on $\Delta_{*,*}$ with
$L^0=L^1=\unit$, $L^n=L^{\otimes(n-1)}$,
\begin{align*}
s^0=\varepsilon\otimes\id,\;
s^n=\id\otimes\varepsilon,\;
s^i=\id\otimes\mu\otimes\id:L^{\otimes n}\to L^{\otimes(n-1)},
\\
d^i=\id\otimes\eta\otimes\id:L^{\otimes (n-2)}\to L^{\otimes (n-1)},
\end{align*}
If in addition $L$ has an absorbing object then
$L^\bullet$ has a lax monoidal structure
\begin{equation}\id\otimes\overline\eta\otimes\id:L^{\otimes (p-1)}\otimes L^{\otimes (q-1)}\to L^{\otimes (p+q-1)}\end{equation}
so we obtain a monoidal functor $L^\bullet:\Omega\Delta\to\D$.
\begin{df}
Let $X$ be a simplicial set, or the corresponding contravariant monoidal functor on the category of simplicial strings (Proposition~\ref{oplax}).
Baues' \emph{geometric cobar construction} $\Omega_LX$ with respect to an interval object $L$  in a co--complete monoidal category $\D$ is defined as the monoid object in $\D$ given by the realization functor  (see Lemma~\ref{realnerve}),
\begin{equation}
\Omega_L(X)=X\otimes_{\Omega\Delta}L^\bullet
\end{equation}
\end{df}

We have four fundamental examples:
\begin{enumerate}
\item Let $L=[0,1]$ be the unit interval in the category of CW complexes, with unit and absorbing objects $0,1:\{*\}\to [0,1]$, and multiplication given by $\max:[0,1]^2\to[0,1]$. Then the geometric cobar construction on a 1-reduced simplicial set is homotopy equivalent to the loop space of the realization of $X$.
\item Taking the cellular chains on the previous interval object we gives an interval object $L$ in the category of chain complexes. In this case $\Omega_L(X)$ coincides with Adams' cobar construction, which has the same homology as the loop space on $X$, if $X$ is 1-reduced.
\item If we forget the boundary maps in example (2) we obtain an interval object $L$ in the category of graded Abelian groups, and $\Omega_L(X)$ coincides as an algebra with the object $\B(X)$ of Proposition \ref{bialgfromsimp}: it is just the free tensor algebra whose generators in dimension $n$ are the $n+1$-simplices of $X$.
\item Let $L=\Delta[1]$ in the category of simplicial sets, with unit and absorbing object $d^1$ and $d^0:\Delta[0]\to\Delta[1]$, and multiplication $\mu:\Delta[1]^2\to\Delta[1]$ defined by \begin{equation}\mu_n([n]\xrightarrow{x}[1],[n]\xrightarrow{x'}[1])=(i\mapsto\max(x_i,x_i')).\end{equation}
Berger has observed that, up to group completion, $\Omega_LX$ has the same homotopy type as the simplicial loop group $GX$ of Kan.
\end{enumerate}

Note that the CW complex given by the simplicial realization of $\Delta[1]^2$ does not have the same cellular structure as $[0,1]^2$: to relate examples (1--3) with (4) requires appropriate diagonal approximation and shuffle maps.

In example (3) the multiplication is free, and we have seen that the co--operad structure $\check\gamma$ on the simplicial set $X$ gives a comultiplication
and hence a bi--algebra structure on $\Omega_L(X)=\B(X)$.
Baues showed that essentially the same co--product gives a differential graded bi--algebra structure on $\Omega_L(X)$ in example (2), and used this to iterate the classical cobar construction to obtain an algebraic model of the double loop space. In example (4) we remain in the category of simplicial sets, and we have the following result:

\begin{prop}
Let $X$ be a simplicial set, and $\Omega_L(X)$ the simplicial monoid given by the geometric co--bar construction on $X$ with respect to the interval object $L=\Delta[1]$. Then the co--operad structure $\check\gamma$ on $X$ induces a map
\begin{equation}
\Omega_L(X)_n\longrightarrow\prod_{m_1+\dots+m_k=n}\Omega_L(X)_{k-1}\times \Omega_L(X)_{n-k}
\end{equation}
for each $n,k\geq1$.
\end{prop}
\begin{proof}
Let $Y=\Omega_L(X)$. For each partition $m_1+\dots+m_k=n$ the co--operad structure map $\check\gamma_{m_1,\dots,m_k}$ of \eqref{simpcoopformula} induces a map
$Y_{n-1}\longrightarrow Y_{k-1}\times Y_{n-k}$ as follows.
The map $\gamma_{m_1,\dots,m_k}$ of Proposition \ref{deltaoperad} restricts to give a bijection $\underline{k-1}\cup\underline{m_1-1}\cup\dots\cup\underline{m_k-1}\to\underline{n-1}$ and hence an isomorphism
\begin{equation}\Delta[1]^{n-1}\longrightarrow \Delta[1]^{k-1}\times  \Delta[1]^{m_1-1}\times\dots\times  \Delta[1]^{m_k-1}.\end{equation}
Together with the map $\check\gamma_{m_1,\dots,m_k}$ of \eqref{simpcoopformula} this defines a map
\begin{equation}X_n\times\Delta[1]^{n-1}\longrightarrow X_{k}\times \Delta[1]^{k-1}\times (X_{m_1}\times \Delta[1]^{m_1-1}\times\dots\times X_{m_k}\times \Delta[1]^{m_k-1})\end{equation}
which induces the map on $Y$ as required.
\end{proof}

\subsection{Comparison with Goncharov's second Hopf algebra}

We have seen above that Goncharov's first Hopf algebra $\HG$ and Baues Hopf algebra $\Omega_L(X)$ are closely related. The differences between Baues' and Goncharov's algebras are as follows
\begin{enumerate}
\item Baues' Hopf algebra has a differential, and the underlying graded Abelian group $\B(X)$ is the free tensor algebra, that is, a free associative algebra. No differential is given on Goncharov's algebra, which is a free polynomial algebra, that is, a free commutative and associative algebra.
\item To obtain a model for the double loop space Baues requires $X$ to have trivial 2-skeleton (only one vertex, one degenerate edge, and one degenerate 2-simplex), but to construct Goncharov's bi--algebra we take $X$ to be 0-co--skeletal (a unique $n$-simplex for any $(n+1)$-tuple of vertices).
In the latter construction, however, one may still impose the relations $x\sim 1$ and $x\sim 0$  for 1- and 2-simplices $x$ \emph{after} taking the polynomial algebra (compare \eqref{emptycondeq} and \eqref{2skeq} respectively).
\end{enumerate}

For Goncharov's second Hopf algebra $\tilde\HG$, and the variants due to Brown, one imposes extra relations such as the shuffle formula \eqref{shufflecond}. This has the following natural expression in the language of the cobar construction.
Let $X=X(S)$, the 0-co--skeletal simplicial set with vertex set $X_0=S$.
The cobar construction $\Omega_LX$ is a colimit of copies of  $C(x_{n+1})=L^{\otimes n}$ for each $(n+1)$-simplex $x_{n+1}=(s;w_n;s')$, where $w_n$ is a word of length $n$ in the alphabet $S$.
In a symmetric monoidal category
each $(p,n-p)$-shuffle corresponds to a natural isomorphism
$L^{\otimes p}\otimes L^{\otimes(n-p)}\to L^{\otimes n}$ and the content of the shuffle relation is that this isomorphism is also obtained from the shuffle of the letters of a word $w_p$ with a word $w_{n-p}$ to obtain a word $w_n$.

\subsection{Cubical structure}
\label{cubicalpar}
Baues' and Goncharov's co--multiplications come from path or loop spaces and  may be seen having natural cubical structure.
The space of paths $P$ from $0$ to $n$ in the $n$-simplex $|\Delta[n]|$
is a cell complex homeomorphic to the $(n-1)$-dimensional cube.

  Cubical complexes have a natural  diagonal approximation,
\begin{equation}
\!\!\!\! \!\!\delta:
 P=[0,1]^{n-1}\xrightarrow{\;\;\cong\;\;}
\!\!\!\! \!\!\!\! \!\!
\bigcup_{K\cup L=\{1,\dots,n-1\}}
\!\!\!\! \!\!\!\! \!\!
\partial_K^-[0,1]^{n-1} \times
\partial_L^+[0,1]^{n-1} \xrightarrow{\;\;\subset\;\;}
P\times P
\end{equation}

 One can identify faces $\partial^-_i$
 of the cube $P$ as the spaces of paths through the
faces $x_{(0,\dots,\widehat i,\dots,n)}$
of the $n$-simplex $x$. Faces $\partial^+_i$ are products of a $(i-1)$-cube and $(n-i-1)$-cube:
the spaces of paths through
$x_{(0,\dots,i)}$ and  through $x_{(i,\dots,n)}$.

The term for $L=\{i_1,\dots,i_{k-1}\}$  under this identification is
\begin{equation}
x_{(0,i_1,\dots,i_{k-1},n)}\;\times\;
x_{(0,1,\dots,i_1)}\, x_{(i_1,i_1+1,\dots,i_2)}\,\ldots\,
x_{(i_{k-1},i_{k-1}+1,\dots,n)}.
\end{equation}
which reproduces the summands of the co--product.

The cubical structure is illustrated for the case of $\Delta^3$ in Figure \ref{squarefig}

\begin{figure}[htbp]
    \centering
    \includegraphics[width=.6\textwidth]{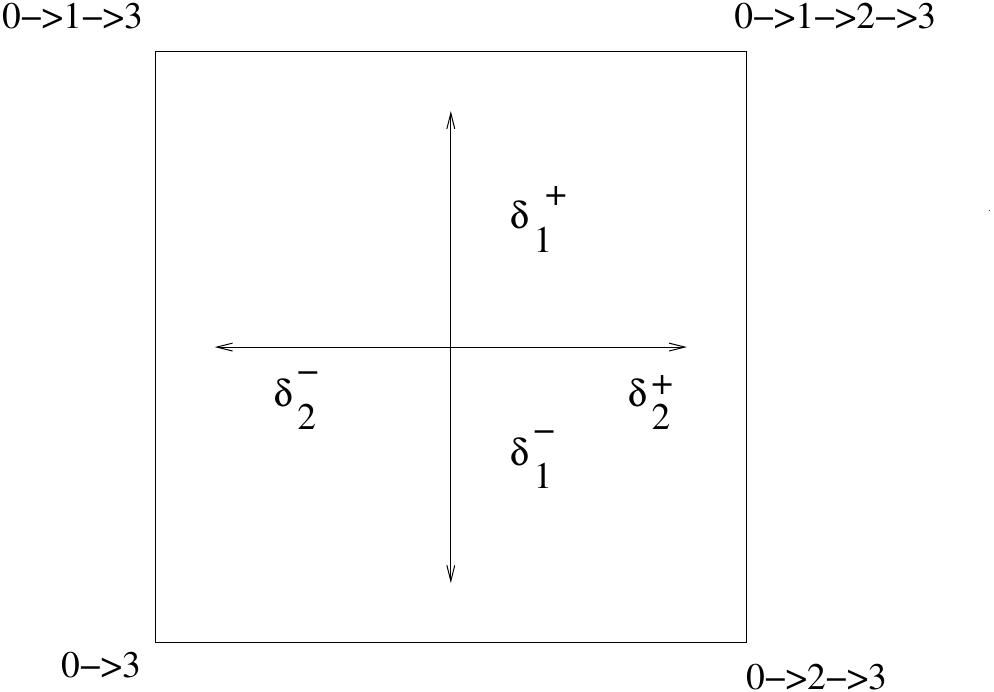}
    \caption{The cubical structure in the case of $n=3$}
    \label{squarefig}
\end{figure}

To get into this analysis, we can choose two other alternative presentations.
The first is given by a self--explanatory bar notation and the second is a parameterized notation. For the latter, we use $0\stackrel{a}{\rightarrow}1\stackrel{b}{\rightarrow}2\stackrel{c}{\rightarrow}3$. Then $s,t$ are formal parameters.  At $t=1$ an edge disappears, while for $t=0$ the morphisms are composed. The latter also explains the shuffles very nicely. Indeed in the usual diagonal approximation there is a shuffle of inner degeneracies. The degeneracies are the composition and the square modulo the symmetric group action yields the simplex. Lifting this yields the terms in the shuffle product.

\begin{figure}[htbp]
    \centering
\includegraphics[width=.4\textwidth]{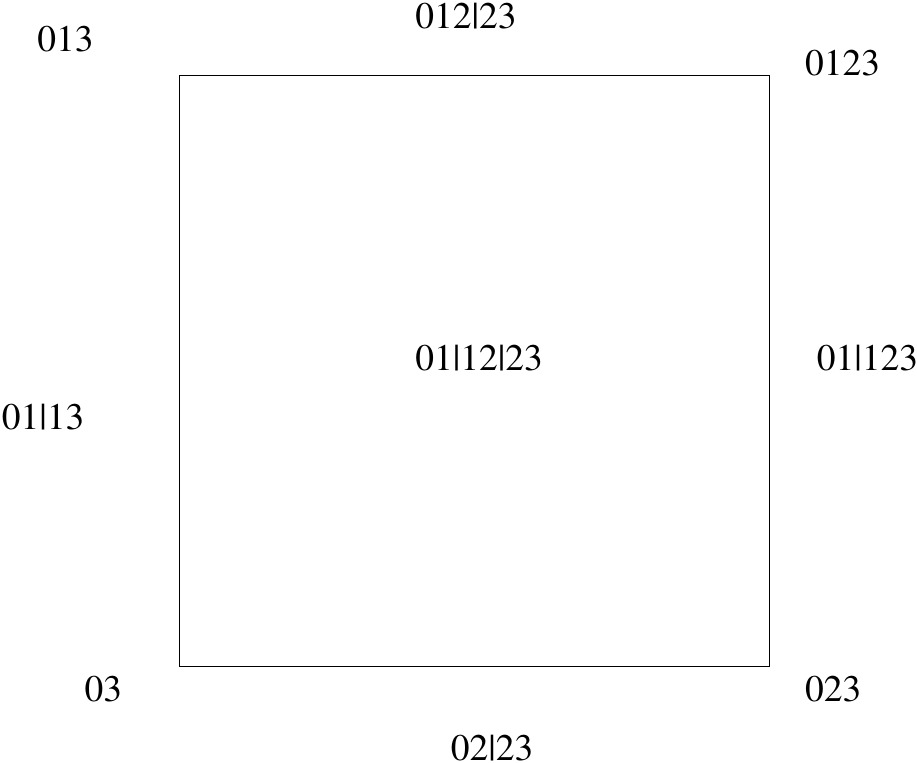}
    \includegraphics[width=.4\textwidth]{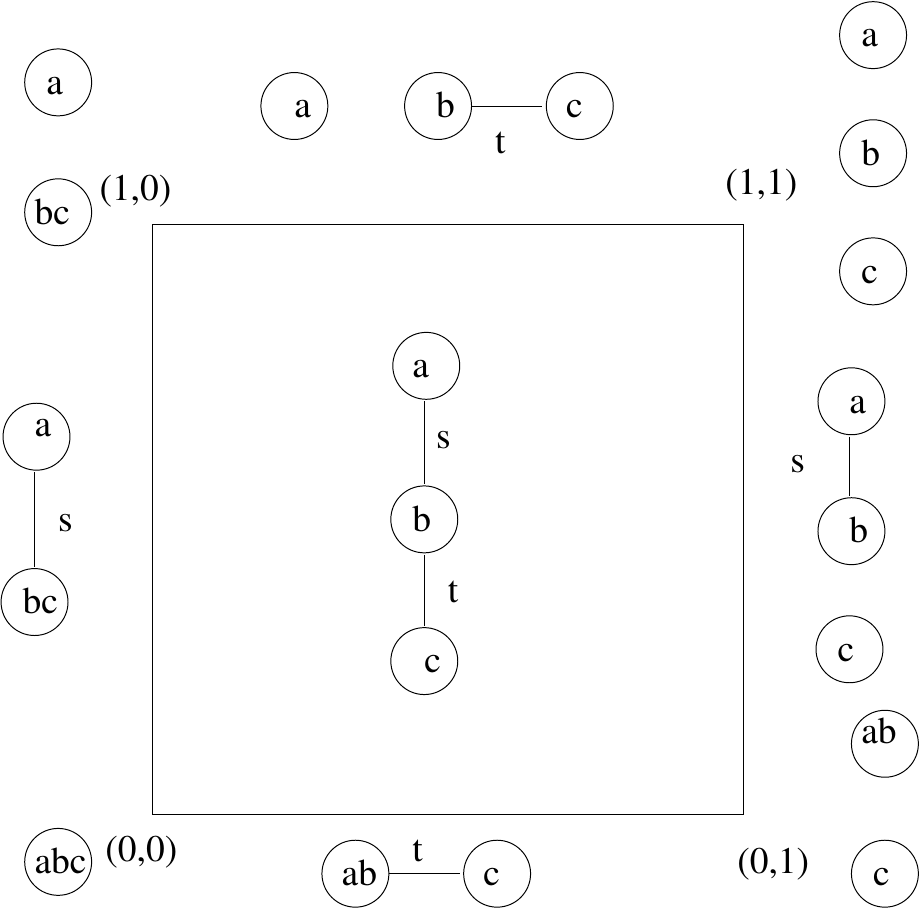}
    \caption{Two other renderings of the same square}
    \label{squarefigs}
\end{figure}

 The cubical structure is also related to
Cutkosky rules \cite{Bloch,BlochKreimer,Kreimercut}  Outer Space \cite{ConantVogtmann}.
This natural appearance of cubical chains can be understood using  decorated Feynman categories \cite{decorated} and the W--construction \cite{feynman}, as explained in \cite{BergerKaufmann}.

\appendix

\section{Co--algebra and Hopf algebras}
\label{Quilapp}
A good source for this material is \cite{Cartier}.

\begin{df}
A  co--algebra with a split co--unit is a triple  $(\H,\eps,\eta)$, where $(\H,\eps)$ is a co--algebra and $\eta:\unit\to \H$ is a section of $\eta$, such that if
$|:=\eta(1)$, $\Delta(|)=|\otimes |$.
\end{df}
Using $\eta$,  split $\H=\unit\oplus \bar{\H}$ where $\bar{\H}:=ker(\eps)$.
Following Quillen \cite{Quillen}, set $\bar{\Delta}(a):=\Delta(a)-|\otimes a-a\otimes |$ where $|:=\eta(1)$.

If $(\H,\mu,\Delta,\eta,\eps)$ is a bi--algebra then  the restriction
$(\H,\Delta,\eps)$ is a co--algebra with split co--unit.

A co--algebra with split co--unit $\H$ is said to be co--nilpotent if for all $a\in \bar\H$ there is an $n$ such that\ $\bar\Delta^n(a)=0$ or equivalently if for some $m: a\in ker (\proj^{\otimes m+1}\circ \Delta^m)$.

Quillen defined the following filtered object.
\begin{equation}
F^0=\unit; F^m =\{a:\bar\Delta a\in F^{m-1}\otimes F^{m-1}\}
\end{equation}

$\H$ is said to be connected, if $\H=\bigcup_m F^m$.
If $\H$ is connected, then it is nilpotent, and conversely if taking kernels and the tensor product commute then
co--nilpotent implies connected where $F^m=ker (\proj^{\otimes m+1}\circ \Delta^m)$.

For a co--nilpotent bi--algebra algebra there is a unique formula for a possible antipode given by:
\begin{equation}
    S(x)=\sum_{n\geq 0} (-1)^{n+1}\mu^n\circ \bar\Delta^n(x)
\end{equation}
where $\bar \Delta^n:\H\to \H^{\otimes n}$ is the  $n-1$-st iterate of $\bar \Delta$ that is unique due to co--associativity and $\mu^n:\H^n\to \H$ is the  $n-1$-st iterate of the multiplication $\mu$ that is unique due to associativity.

\section{Joyal duality}

\label{Joyalapp}
There is a well known
duality \cite{joyal} of two subcategories of $\simpcat_+$. This history of this duality can be traced back to \cite{street}. Here we review this operation and in \cite{HopfPart2} we show how it can be graphically interpreted. The graphical notation we present in \S\ref{P2-constexpar} of {\em loc.\ cit.} in turn connects to the graphical notation in \cite{Gont} and \cite{gangl}.

The first of the two subcategories of $\simpcat_+$ is $\simpcat$ and the second is the category of intervals $\Delta_{*,*}$. Since we will be dealing with both $\simpcat$ and $\simpcat_+$, we will use the notation $\underline{n}$ for the small category $1\to  \dots \to n$ in $\simpcat$ and $[n]$ for $0\to 1 \to \dots \to n$ in $\simpcat_+$.
The subcategory of intervals $\Delta_{*,*}$ is then the wide subcategory of $\simpcat_+$ whose morphisms  preserve both the beginning and the end point.
We will denote these maps by  $Hom_{*,*}([m],[n])$. Explicitly $\phi\in
Hom_{*,*}([m],[n])$ is $\phi(0)=0$ and $\phi(m)=n$.

The contravariant duality is then given by the natural bijjections
\begin{equation}
Hom_{*,*}([m],[n])\simeq Hom(\underline{n},\underline{m})\quad
\end{equation}
defined  by the contravariant identification   $\phi\stackrel{1-1}{\leftrightarrow}\psi$
\begin{equation}
\label{joyaleq}
\psi(i)=\min\{j:\phi(j)\geq i\}-1, \quad
\phi(j)=\max\{i:\phi(i)<j\}+1.
\end{equation}

\bibliography{hopfbib}
\bibliographystyle{halpha}

\end{document}